\documentclass[10pt]{article} 
\usepackage{amsmath}
\usepackage{amssymb}
\usepackage{mathrsfs}
\usepackage{amsthm}
\usepackage[english]{babel}
\usepackage[margin=2.5cm]{geometry}
\usepackage{xcolor}
\usepackage{verbatim}
\usepackage{stmaryrd}
\usepackage{hyperref}
\usepackage[nameinlink,capitalize]{cleveref}
\usepackage[style=alphabetic,maxalphanames=99,maxnames=99]{biblatex}
\newcommand{\R}{\mathbb{R}}
\newcommand{\E}{\mathbb{E}}

\newcommand{\esssup}{\mathop{\mathrm{ess\,sup}}}

\theoremstyle{definition}
\newtheorem{defn}{Definition}[section]

\newtheorem{rmk}[defn]{Remark}

\newtheorem{assumptions}[defn]{Assumptions}
\theoremstyle{plain}
\newtheorem{thm}[defn]{Theorem}
\newtheorem*{thm*}{Theorem}
\newtheorem{lemma}[defn]{Lemma}
\newtheorem{coroll}[defn]{Corollary}
\newtheorem{prop}[defn]{Proposition}

\title{Parameter dependent rough SDEs with applications to rough PDEs}
\author{Fabio Bugini\thanks{Technische Universit\"at Berlin. Email: fabio.bugini@tu-berlin.de} , Peter K. Friz\thanks{Technische Universit\"at Berlin and Weierstrass Institut, Berlin, Germany. Email: friz@math.tu-berlin.de} , Wilhelm Stannat\thanks{Technische Universit\"at Berlin, Berlin, Germany. Email: stannat@math.tu-berlin.de}}

\begin{document}
\maketitle

\begin{abstract}
Rough stochastic differential equations (rough SDEs), recently introduced by Friz, Hocquet and L\^e in \cite{FHL21} have emerged as a versatile tool to study ``doubly'' SDEs under partial conditioning (with motivation from pathwise filtering and control, volatility modelling in finance and mean-field stochastic dynamics with common noise \dots).  While the full dynamics may be highly non-Markovian, the conditional dynamics often are. In natural (and even linear) situations, the resulting stochastic PDEs can be beyond existing technology. The present work then tackles a key problem in this context, which is the well-posedness of regular solution to the rough Kolmogorov backward equation. To this end, we study parameter dependent rough SDEs in sense of $\mathscr{L}$-
differentiability (as in Krylov, 2008). 
In companion works (\cite{BFLZ25p}, \cite{BFS25p}), we will show how this removes dimension-dependent regularity assumptions for well-posedness of the Zakai, Kushner--Stratonovich and nonlinear Fokker--Planck stochastic equations. 

\end{abstract}



\setcounter{tocdepth}{2}
\tableofcontents

\section{Introduction}
Consider the following parameter-dependent stochastic equation in $\R^d$,
\begin{equation} \label{eq:RSDEwithparameter_intro}
    dX^\zeta_t =  b_t(\zeta,X^\zeta_t) dt + \sigma_t(\zeta,X_t^\zeta) \ dB_t + \beta_t (\zeta,X_t^\zeta) \ d\mathbf{W}_t \qquad t \in [0,T], 
\end{equation} 
where $\zeta$ denotes a parameter in a certain parameter space $U$, typically endowed with a Banach space structure.  
Here $B=(B_t)_{t \in [0,T]}$ is an $m$-dimensional $\{\mathcal{F}_t\}_t$-Brownian motion, given on a complete filtered probability space $(\Omega,\mathcal{F},\{\mathcal{F}_t\}_{t \in [0,T]},\mathbb{P})$, and $\mathbf{W}=(W,\mathbb{W})$ is a deterministic $\alpha$-H\"older rough path in $\R^n$, with $\alpha \in (\frac{1}{3},\frac{1}{2}]$.
The coefficients of the equation are progressively measurable mappings \begin{equation*} \begin{aligned}
    b : \Omega \times [0,T] \times U \times \R^d \to \R^d, \quad \sigma:\Omega \times [0,T] \times U \times \R^d \to \R^{d\times m}, \quad  \beta:\Omega \times [0,T] \times U \times \R^d \to \R^{d\times n}.
\end{aligned}
\end{equation*}
At this stage \eqref{eq:RSDEwithparameter_intro} is completely formal, and it is called a rough stochastic differential equation (rough SDE or RSDE).
In the first part of this work we introduce and develop a concept of continuity  and differentiability of the solution $X$ with respect to the parameter $\zeta$ in the spirit Krylov's treatise of parameter-dependent SDEs (cf.\ \cite[Section 1.2.8]{KRYLOV}) and allows us to deduce some continuity and differentiability properties of the function \begin{equation} \label{eq:functional_intro}
    U \ni \zeta \longmapsto J(\zeta) := \E(\phi(X^\zeta)) \in \R. 
\end{equation}
Here $\phi(\cdot)$ denotes some observable of $X^\zeta$; typical examples are a final time evaluation $\phi(X^\zeta)=g(X^\zeta_T)$ for some $g:\R^d \to \R$, or some integral functional of the form $\phi(X^\zeta)=\int_0^T \mu(X^\zeta_r) dr + \int_0^T \nu(X_r^\zeta) d\mathbf{W}_r$
for some $\mu:\R^d \to \R$ and $\nu:\R^d \to \R^{1\times n}$. \\

The second main goal of this paper is to provide an existence-and-uniqueness theory for a certain class of linear rough PDEs, by proving a rough path version of the well-known Feynman--Kac formula. 
We consider the following purely deterministic terminal value problem \begin{equation} \label{eq:roughPDE_intro}
    \begin{cases}
        -\partial_t u(t,x) &= L_t u_t(x) dt + \Gamma_t u_t (x) d\mathbf{W}_t \qquad t \in [0,T], \ x \in \R^d \\
        u(T,\cdot) &= g
    \end{cases},
\end{equation}
where $\mathbf{W}=(W,\mathbb{W})$ is an $\alpha$-H\"older geometric rough path in $\mathbb{R}^n$, with $\alpha \in (\frac{1}{3},\frac{1}{2}]$, and $\ g:\mathbb{R}^d \to \mathbb{R}$. 
The solution is a map $u:[0,T]\times\mathbb{R}^d \to \mathbb{R}$,  and $L_t,\Gamma_t$ are the time-dependent (second and first order, respectively) differential operators defined by \begin{equation*}
    \begin{aligned}
        L_t u := \frac{1}{2}\text{tr}\left(\sigma(t,\cdot)\sigma(t,\cdot)^\top D_{xx}^2 u\right) + b(t,\cdot) \cdot D_x u + c(t,\cdot) u  \qquad \text{and} \qquad
        \Gamma_t u := \beta(t,\cdot) \cdot D_x u + \gamma(t,\cdot) u.
    \end{aligned}
\end{equation*}
Here $\sigma=(\sigma^1,\dots,\sigma^m), \ \beta_t=(\beta^1,\dots,\beta^n)$ and $b$ are time-dependent (deterministic) vector fields over $\mathbb{R}^d$, while $c$ and $\gamma=(\gamma^1,\dots,\gamma^n)$ are time-dependent scalar functions. 
In the smooth case (i.e.\ when $d\mathbf{W}_t$ is replaced by $\dot{W}_t dt$), a standard It\^o calculus result states that any bounded $C^{1,2}$ solution to \eqref{eq:roughPDE_intro} is given by the Feynman--Kac formula \begin{equation} \label{eq:FeynmanKacsolution_intro}
    u(t,x) := \mathbb{E} \Big( g(X_T^{t,x}) \exp \Big\{ \int_t^Tc_r(X_r^{t,x})dr + \int_t^T \gamma_r(X_r^{t,x})\dot{W}_r dr \Big\}  \Big), 
\end{equation}
being $X^{t,x}$ the solution to the following SDE: \begin{equation} \label{eq:roughSDE_FeynmanKac}
    dX^{t,x}_s = b(s,X^{t,x}_s) ds +\sigma(s,X^{t,x}_s) dB_s + \beta(s,X^{t,x}_s) \dot{W}_s ds \quad s \in [t,T], \quad X^{t,x}_t=x,
\end{equation}
for any given Brownian motion $B$. 
Notice how the functional defined in \eqref{eq:FeynmanKacsolution_intro} is an example of the one we introduced in \eqref{eq:functional_intro}, where initial time and initial state play the role of parameters. 
For a proof of the Feynman--Kac formula in the smooth case see \cite[Theorem 10]{KRYLOV}, which also shows how parameter-dependence results for SDEs are applied.
One of the key ingredients is indeed the understanding of the continuity of \begin{equation*}
    (t,x) \longmapsto \mathbb{E} \Big( g(X_T^{t,x}) \exp \Big\{ \int_t^Tc_r(X_r^{t,x})dr + \int_t^T \gamma_r(X_r^{t,x}) \dot{W}_r dr \Big\}  \Big), 
\end{equation*}
and of the differentiability of the previous map with respect to $x$, for fixed $t$.
In our work, we generalize this classical result to the rough case (that is, when $W$ is no longer time-differentiable and we lift it to a rough path). 
Equation \eqref{eq:roughSDE_FeynmanKac} becomes a rough SDE and the integral with respect to $W$ in the exponential term in \eqref{eq:FeynmanKacsolution_intro} becomes a rough stochastic integral. In this sense, one can appreciate how well-posedness of rough PDEs via a rough Feynman--Kac formula constitutes an application of parameter-dependent rough SDEs theory.
Understanding integrability such as \begin{equation*}
    \E \left( e ^{\int_t^T \gamma(X_r) d\mathbf{W}_r} \right) < +\infty
\end{equation*}
when $\mathbf{W}$ is a genuine rough path is a non-trivial task. In previous approaches (cf. \cite{DOR13, DFS17}) this was achieved with Gaussian isoperimetry methods, not applicable in the present generality. Instead we employ John-Nirenberg type inequalities  \cite{LE22JohnNirenberg}, perfectly adapted to our hybrid rough martingale approach, see \cref{section:exponentialterm}. \\

\textbf{What are rough SDEs? }
Rough SDEs are simultaneous generalizations of It\^o's SDEs and Lyons' rough differential equations (RDEs). 
First introduced in \cite{FHL21}, they unify classical stochastic calculus and rough path theory within a single framework. 
This gives intrinsic meaning to continuous adapted solutions $X=X^\mathbf{W}(\omega)$ to equations simultaneously driven by a Brownian motion $B$ and a deterministic rough path $\mathbf{W}$, of the form \begin{equation*}
    dX_t = b_t(X_t) \, dt + \sigma_t(X_t) \, dB_t + (\beta_t, \beta'_t)(X_t) \, d\mathbf{W}_t .
\end{equation*}
Coefficients can be random in the sense of progressive measurability.
As is typical for well-posedness in rough path theory, $\beta$ needs to have $C^{2+}$ space regularity and satisfy  a Taylor-like expansion of the form $\beta_t(\cdot) - \beta_s(\cdot) \approx \beta'_s(\cdot) (W_t - W_s)$.
In contrast, $b$ and $\sigma$ are only required to be progressively measurable in $(t,\omega)$ and Lipschitz continuous in space, as is standard in It\^o theory to have well-posedness. 
In \cite{FHL21}, all the coefficients are assumed to be globally bounded, as is common in rough path theory. The linear (unbounded) case is treated in \cite[Section 3]{BCN24}, also using 
stochastic sewing tools \cite{LE20}.
A brief yet rich and self-contained introduction to rough SDEs can be found below in \cref{section:compendiumonRSDEs}. Examples of rough SDEs can be found in the 
discussion below; for a more general overview of the theory and its motivating applications we refer to \cite{FHL21}. \\

\textbf{Why rough PDEs? } Rough PDEs of the form \eqref{eq:roughPDE_intro} arise in a variety of situations. 
A selection of these is presented below to highlight how our work can serve as a key ingredient, as well as to illustrate the versatility and applicability of the theory of (parameter-dependent) rough SDEs.
In all the examples below $B$ and $B^\perp$ are independent Brownian motions on a given filtered probability space $(\Omega, \mathcal{F}, \{\mathcal{F}_t\}_{t \in [0,T]}, \mathbb{P})$. \\

The class of regular backward equations we are dealing with serves as a key tool for establishing uniqueness results for both linear and nonlinear measure-valued Fokker--Planck equations by duality methods. 
Building on the approaches in \cite{CP24} and \cite{CG19}, our work prepares a systematic way to get rid of dimension-dependent regularity assumptions used in those papers, due to their use of backward stochastic PDEs (BSPDEs) and Sobolev embeddings. More precisely: \\

1. \ In stochastic filtering (see, for instance, \cite{CP24} for a general overview) one is interested in estimating $\varsigma_t(\varphi) := \E(\varphi(X_t) \mid Y_r, \ 0 \le r \le t)$ for a sufficiently large class of test functions $\varphi$, where
\begin{align*}
    dX_t &= b_t(X_t) \, dt + \sigma_t(X_t) \, dB_t + \beta_t(X_t) \, dY_t \quad t \in [0,T] &\text{Observation} \\
    dY_t &= \gamma_t(X_t) \, dt + dB^\perp_t . &\text{Signal}
\end{align*}
By Girsanov's theorem, the process $Y$ is a Brownian motion under a new measure $d\tilde{\mathbb{P}} \mid_{\mathcal{F}_t} = e^{I_t} \, d\mathbb{P} \mid_{\mathcal{F}_t}$ with $I_t=\int_0^t \gamma_t(X_t)^T \, dY_t - \frac{1}{2} \int_0^t |\gamma_t(X_t)|^2 \, dt $ . In particular, it admits an (It\^o) rough path lift $\mathbf{Y}(\omega) := (Y(\omega), (\int \int dY dY)(\omega))$. 
From the Kallianpur--Striebel formula (see e.g.\ \cite[Proposition 3.5]{CP24}), 
\begin{equation*}
    \varsigma_t (\varphi) (\omega) = \frac{\tilde{\E}(\varphi(X_t) e^{I_t} \mid Y_r, \ 0 \le r \le T)(\omega)}{\tilde{\E}(e^{I_t} \mid Y_r, \ 0 \le r \le T)(\omega)} =: \frac{\pi_t(\varphi)(\omega)}{\pi_t(1)(\omega)} \qquad \text{$\tilde {\mathbb{P}} (\mathbb{P})$-a.s.}
\end{equation*}
and one can show that $\pi_t$ is a measure-valued solution to the so-called Zakai stochastic PDE. One of the focuses of \cite{CP24} is in showing uniqueness for that equation, via duality arguments involving a certain class of BSPDEs, seemingly unavoidable if observation coefficients depend more generally on $(X_t,Y_t)$. 
The adoption of a rough path perspective on the stochastic filtering problem is the focus of \cite{BFLZ25p}; we briefly outline the main ideas here. For any given rough path $\mathbf{W}$, one can define \begin{equation*}
    \pi^\mathbf{W}_t (\varphi) := \tilde{\E}(\varphi(X^\mathbf{W}_t) e^{I^\mathbf{W}_t}) \qquad t \in [0,T],  
\end{equation*}
where $dX^\mathbf{W}_t = b_t(X^\mathbf{W}_t) \, dt + \sigma_t(X^\mathbf{W}_t) \, dB_t + \beta_t(X^\mathbf{W}_t) \, d\mathbf{W}_t$ is a rough SDE and $I^\mathbf{W}_t=\int_0^t \gamma_t(X^\mathbf{W}_t)^T \, d\mathbf{W}_t - \frac{1}{2} \int_0^t |\gamma_t(X^\mathbf{W}_t)|^2 \, dt $. From \cite[Lemma 3.1]{DFS17}, one has that $\pi^\mathbf{W}$ solves a measure-valued rough PDE of the form \begin{equation} \label{eq:roughZakai_intro}
    d \pi^\mathbf{W}_t(\varphi) = \pi^\mathbf{W}_t(L_t \varphi) \, dt + \pi^\mathbf{W}_t (\Gamma_t \varphi) \, d\mathbf{W}_t \quad t \in [0,T], \quad \pi^\mathbf{W}_0 = \text{Law}(X^\mathbf{W}_0)
\end{equation}
where $L_t \varphi = \frac{1}{2} tr(\sigma_t \sigma^T_t D^2_{xx}\varphi) + b_t D_x\varphi - \frac{1}{2}|\gamma_t|^2 \varphi$ and $\Gamma_t \varphi = \beta_t D_x\varphi + \gamma_t \varphi$.
By a duality argument (see \cite[Lemma 4.15]{DFS17} for a rough PDEs version), uniqueness for the forward equation \eqref{eq:roughZakai_intro} follows from the existence of a solution $u^\mathbf{W}$ to the backward equation \eqref{eq:roughPDE_intro}.
Within the framework of our paper, one appreciates how uniqueness for \eqref{eq:roughZakai_intro} therefore holds under assumptions that are not dimension-dependent. (The rough path approach works immediately with observation coefficients depending on $(X_t,Y_t)$.)
Going back to the original filtering problem, one can perform a randomisation argument as explained in \cite{FLZ25} to conclude that
\begin{equation*}
    \tilde{\E}(\varphi(X_t) e^{I_t} \mid Y_r, \ 0 \le r \le T)(\omega) =  \pi^\mathbf{W}_t (\varphi)\big|_{\mathbf{W}=\mathbf{Y}(\omega)} \qquad \text{$\tilde {\mathbb{P}} (\mathbb{P})$-a.s.} 
\end{equation*} 

\

2. \ Consider a McKean--Vlasov SDE with common noise of the form \begin{equation*} \begin{cases}
    dX_t = b_t(X_t, \mu_t) \, dt + \sigma_t(X_t,\mu_t) \, dB_t + \beta_t(X_t, \mu_t) \, dB^\perp_t \quad t \in [0,T]\\
    \mu_t(\omega) = \mathbb{P}(X_t \in \cdot \mid B^\perp_r, \ 0 \le r \le t) (\omega)
\end{cases} ,
\end{equation*}
where $\mu$ solves the following nonlinear measure-valued Fokker--Planck stochastic PDE (see, for instance, \cite[p. 110]{CARMONADELARUE_volumeII}): \begin{equation} \label{eq:stochasticFokkerPlanck_intro}
    d \mu_t(\varphi)(\omega) = \mu_t(L_t[\mu_t] \varphi)(\omega) \, dt + \mu_t (\Gamma_t[\mu_t] \varphi)(\omega) \, dB^\perp_t(\omega) \qquad t \in [0,T],
\end{equation}
with $L_t[\mu]\varphi = \frac{1}{2} tr((\sigma_t \sigma^T_t)(\cdot,\mu) D^2_{xx}\varphi) + b_t(\cdot,\mu) D_x\varphi$ and $\Gamma_t[\mu] \varphi = \beta_t(\cdot,\mu) D_x\varphi$.
In \cite{CG19} the authors prove uniqueness of \eqref{eq:stochasticFokkerPlanck_intro} via a duality argument and BSPDEs, under dimension-dependent regularity assumptions, which are due to some Sobolev embeddings. The adoption of a rough path point of view towards nonlinear Fokker--Planck equations is the object of \cite{BFS25p}; we briefly outline here the main ideas. Given any rough path $\mathbf{W}$, one can define $$\mu_t^\mathbf{W}(\varphi) := \E(\varphi(X_t^\mathbf{W})) \qquad t \in [0,T],$$ where 
$dX^\mathbf{W}_t = b_t(X^\mathbf{W}_t, \mu^\mathbf{W}_t) \, dt + \sigma_t (X^\mathbf{W}_t, \mu^\mathbf{W}_t) \, dB_t + f_t (X^\mathbf{W}_t, \mu^\mathbf{W}_t) \, d\mathbf{W}_t, \  \mu^\mathbf{W}_t = \mathbb{P}(X^\mathbf{W}_t \in \cdot \ )$
is a McKean--Vlasov SDE with rough common noise. A solution theory for the previous class of rough SDEs can be found in \cite{FHL25p}; while \cite{BFHL25p} contains some propagation of chaos results. 
In \cite{CN21} the authors have shown that, if $\beta_t(x,\cdot)$ is linear on the space of probability measures, then $\mu^\mathbf{W}$ is the (unique) solution to the following nonlinear Fokker--Planck rough PDE: \begin{equation*}
    d \mu^\mathbf{W}_t(\varphi) = \mu^\mathbf{W}_t(L_t[\mu_t^\mathbf{W}] \varphi) \, dt + \mu^\mathbf{W}_t (\Gamma_t[\mu_t^\mathbf{W}] \varphi) \, d\mathbf{W}_t \quad t \in [0,T] .
\end{equation*}
 Their argument relies on the well-posedness of the linearized (i.e.\ without measure-dependent coefficients) equation. 
Within the framework of our paper and in a similar fashion to what we discussed in 1., it is possible to obtain a similar well-posedness result under more general assumptions on $\beta$, provided that existence for equation \eqref{eq:roughPDE_intro} is known. Upon randomisation (see \cite{FLZ25}), one also notes that \begin{equation*}
    \mathbb{P}(X_t \in \cdot \ ) (\omega) = \mu_t^\mathbf{W} \big|_{\mathbf{W} = \mathbf{B}^\perp(\omega)} \qquad \text{$\mathbb{P}$-a.s.}
\end{equation*}  
being $\mathbf{B}^\perp(\omega)$ the (It\^o) rough path lift of the common noise $B^\perp$.

\

3. Interestingly, there is a third use of rough SDEs/PDEs in the context of conditional option pricing, offering an intrinsic description of the rough PDEs of 
\cite{BBFP25} (treated therein via joint lift of Brownian motion and rough path). Specifically, it is seen that PDEs of the form
\eqref{eq:roughPDE_intro} can be useful in pricing options for local stochastic volatility models of the form \begin{equation*}
    dY_{t}(\omega) = l(t,Y_t(\omega)) v_t(\omega) \left(\sqrt{1-\rho^2} \, dB_{t}(\omega) + \rho \, d{B}^\perp_{t}(\omega) \right) \in \R \quad t \in [0,T] 
\end{equation*}
for some ``leverage'' function $l$, stochastic volatility process $v$, measurable only with respect to $B^\perp$, and correlation parameter $\rho \in (-1,1)$. No Markovian structure of $(Y,v)$ is assumed, thereby also covering popular ``rough'' (local) volatility models.

The interest is then in computing the time-$T$ payoff $\psi$ of $Y$ conditional on $B^\perp$, that is
    \begin{equation*}
        u(t,y,\omega) := \E^{t,y} \big(\psi(Y_T) \mid {B}^\perp_r, \ 0 \le r \le T\big)(\omega) 
    \end{equation*}
with $\mathbb{P}^{t,y}(Y_t=y) =1$.
Rough path meaning $\mathbf{M}$ is given to $M_t(\omega) := (\int_0^t v_r \, dB^\perp_r)(\omega)$, so that \begin{equation*}
    u(t,y,\omega) = \E^{t,y} (\psi(Y^\mathbf{W}_T)) \big|_{\mathbf{W}=\mathbf{M}_\cdot(\omega)} =: u^\mathbf{W}(t,y) \big|_{\mathbf{W}=\mathbf{M}_\cdot(\omega)} \qquad \text{$\mathbb{P}$-a.s.}
\end{equation*} 
where $dY^\mathbf{W}_{t}(\omega) = l(t,Y^\mathbf{W}_t(\omega)) (v^\mathbf{W}_t\sqrt{1-\rho^2} \, dB_{t}(\omega) + \rho \, d\mathbf{W}_t )$ can be seen as an instance of a rough SDE. 
Here $\mathbf{v}^\mathbf{W}_t$ is the square-root of the derivative of $[\mathbf{M}]$, which coincides with $v_t(\omega)$ upon randomization $\mathbf{W}=\mathbf{M}_\cdot(\omega)$.
Within the framework of our paper, one can characterize
$u^\mathbf{W}$ as the unique solution to \eqref{eq:roughPDE_intro} with $L_t = \frac{1}{2} (l(t,\cdot)\mathbf{v}^\mathbf{W}_t)^2 (1-\rho^2) D^2_{xx}$ and $\Gamma_t  = l(t,\cdot) \rho D_x $. Remark that the corresponding stochastic PDE terminal value problem, obtained from \eqref{eq:roughPDE_intro} by (local martingale) randomization, is to the best of the authors' knowledge, impossible to treat with classical methods (since $M$ run backwards, i.e. $t \mapsto M_{T-t}$, may not even be a semimartingale). The situation is actually quite similar in the filtering framework of \cite{CP24}, discussed above, forcing the authors to resort to BSPDE methods. \\

\textbf{Previously on rough PDEs. } 
Deriving a rough Feynman--Kac formula to give a functional integral representation of solutions to rough PDEs was one of the main objectives of \cite{DFS17}.
\cref{section:backwardroughPDEs} of our paper offers 
some 
important 
improvements concerning 
regular solutions for backward rough PDEs. Building on the
framework of
\cite{FHL21}, 
\begin{itemize}
    \item we decrease the needed space-regularity of the coefficients in the equation. Specifically, \cite[Theorem 9]{DFS17}, required $\sigma,\beta,\gamma \in C^6_b$ and $b,c,g \in C^4_b$ to prove regular well-posedness of \eqref{eq:roughPDE_intro}. In \cref{thm:existenceRPDE} below, those assumptions are significantly reduced (especially on the diffusion coefficient $\sigma$) to $\beta,\gamma \in C^5_b$ and $b,\sigma,c,g \in C^3_b$;
    \item we allow (rough, controlled) time-dependence for the coefficients of \eqref{eq:roughPDE_intro}. Time continuity is a sufficient assumption for our well-posedness result, improving \cite{BBFP25} where the time variable is essentially treated as the space variable;
     \item we give an intrinsic meaning to the representation for the solution (in contrast to previous works which approximate the rough driver of the equation via a sequence of smooth (rough) paths); we note that such intrinsic understanding is key to analyze numerical solution methods (cf. \cite{BBFP25}).
     \item we give a direct and self-contained (``rough It\^o calculus'') proof for uniqueness%
     \footnote{We could have used a duality argument, which however does not lead to different assumption of \cref{thm:uniqueness} below.}
     of \eqref{eq:roughPDE_intro}, advertizing the convenience of ``rough stochastic calculus''.
\end{itemize}

\textbf{Content of the paper. } 
\cref{section:compendiumonRSDEs} recalls the main definitions and results on rough SDEs, including the general theory from \cite{FHL21} and the linear case from \cite[Section 3]{BCN24}.
\cref{section:parameterdependence} focuses on parameter-dependent rough SDEs: \cref{section:continuityforRSDEs} establishes a continuity result with respect to parameters; \cref{section:continuityforlinearRSDEs} and \cref{section:differentiabilityforlinearRSDEs} address linear rough SDEs, proving continuity and differentiability results; \cref{section:differentiabilityforRSDEs} applies these to prove differentiability of solutions to rough SDEs with respect to a parameter.  
\cref{section:compositionandintegration_technicalresults} collects key technical results on continuity and differentiability under composition and integration, essential for the analysis of (linear and nonlinear) rough SDEs. 
\cref{section:backwardroughPDEs} turns to rough PDEs, proving existence and uniqueness via a rough Feynman–Kac formula, with supporting auxiliary results on exponential estimates for rough SDEs and functional representations of solutions to rough SDEs in \cref{section:exponentialterm} and \cref{section:functionalsandmarkovproperty}, respectively. \\

\textbf{Acknowledgement. } FB is supported by Deutsche Forschungsgemeinschaft (DFG) - Project-ID 410208580 - IRTG2544 (”Stochastic Analysis in Interaction”). PKF acknowledges support by the Deutsche Forschungsgemeinschaft (DFG, German Research Foundation)
under Germany’s Excellence Strategy – The Berlin Mathematics Research Center MATH+ (EXC-2046/1, project ID: 390685689). PKF and WS acknowledge support from DFG CRC/TRR 388 ``Rough Analysis, Stochastic Dynamics and Related Fields", Projects A02, A07, A10, B04 and B09. 

\medskip

\textbf{Frequently used notation. }  We write $a \lesssim b$ to express that there exists a non-negative constant $C\ge 0$ such that $a \le Cb$. If such a constant depends on some parameters $\theta_1,\dots,\theta_N$ we write $a \lesssim_{\theta_1,\dots,\theta_N} b$.
We denote by $\mathcal{L}(V;\bar{V})$ the space of linear and continuous functions from $V$ to $\bar{V}$ Banach spaces.
The latter is a Banach space, if endowed with the norm $|T|_{\mathcal{L}} := \sup_{x \in V, |x|_V\le 1} |Tx|_{\bar{V}}$. 
We identify $\mathcal{L}(\R^m;\R^n)$ and the tensor product $\R^n \otimes \R^m$ with the matrix space $\R^{n \times m}$. 
The following general identifications also hold: $\mathcal{L}(V \otimes \bar{V};\tilde{V}) \equiv \mathcal{L}(V;\mathscr{L}(\bar{V};\tilde{V})) \equiv \mathcal{L}^2(V \times \bar{V};\tilde{V}),$
where $\mathcal{L}^2(V \times \bar{V};\tilde{V})$ denotes the space of bilinear and continuous maps from $V \times \bar{V}$ to $\tilde{V}$.
We denote by $\mathcal{B}(V;\bar V)$ the space of bounded functions from $V$ to $\bar V$ and for any $F \in \mathcal{B}(V;\bar V)$, $|F|_\infty := \sup_{x \in V} |F(x)|_{\bar{V}}$.
We denote by $DF:V \to \mathcal{L}(V,\bar{V})$ the first Fréchet derivative of $F$. If not otherwise specified, in this paper we always refer to differentiability in Fréchet sense simply as differentiability.  
Given $n \in \mathbb{N}_{\ge 1}$, we denote by ${C}^n_b(V;\bar{V})$ the space of bounded and continuous functions from $V$ to $\bar{V}$ which are $n$-times continuously differentiable with bounded derivatives. 
The space ${C}^n_b(V;\bar{V})$ is a Banach space if endowed with the norm $|F|_{{C}_b^n} := |F|_\infty + |DF|_\infty + \dots + |D^nF|_\infty$,
where $D^nF$ denotes the $n$-th derivative of $F$. 
Any function $F \in {C}^n_b(V,\bar{V})$ satisfies the following - often referred as mean value theorem - for any $x,h \in V$: \begin{equation*}
    F(x+h) - F(x) - \sum_{k=1}^{n-1} \frac{1}{k !} D^kF(x) h^{\otimes k}  = \frac{1}{(n-1)!} \int_0^1 (1-\theta)^{n-1} D^nF(x+\theta h) d\theta \  h^{\otimes n}.
\end{equation*}  
Given $\alpha \in (0,1]$, we define the $\alpha$-H\"older seminorm of $F$ as \begin{equation*}
    [F]_\alpha := \sup_{x,y \in V, x \ne y} \frac{|F(x)-F(y)|_{\bar{V}}}{|x-y|_V^\alpha}.
\end{equation*}
If $[F]_\alpha < +\infty$, we say that $F$ belongs to $C^\alpha(V;\bar{V})$. In case $\alpha = 1$, we say that $F$ belongs to $Lip(V; \bar V)$. 
The space $C^{n+\alpha}_b(V;\bar{V})$ is the space of $C^n_b$-functions such that $D^nF \in C^\alpha$, and we write $|F|_{C^n_b} := |F|_{C^n_b} + [D^nF]_\alpha$.

Given a probability space $(\Omega,\mathcal{F},\mathbb{P})$ and an integrability exponent $p\in [1,\infty)$, we denote by $\|\cdot\|_{p;V}$ - or simply by $\|\cdot\|_p$, if clear from the context - the usual norm on the Lebesgue space $L^p(\Omega;V)$ of $V$-valued random variables. The norm $\|\cdot\|_\infty$ denotes the essential supremum norm.
A $V$-valued stochastic process $X=(X_t)_t$ is said to be $L^p$-integrable if any $X_t$ belongs to $L^p(\Omega;V)$. If $V=\R$, we sometimes write $L^p(\Omega)$ instead of $L^p(\Omega;\R)$. For a given sub-$\sigma$-field $\mathcal{G}\subseteq \mathcal{F}$, the space $L^p(\Omega,\mathcal{G};V)$ is the space of $\mathcal{G}$-measurable random variables in $L^p(\Omega;V)$.

For any $T>0$ we define $\Delta_{[0,T]} := \{ (s,t) \in [0,T]^2 \ | \ s\le t\}$.
For any Banach space valued path $[0,T]\ni t \mapsto X_t \in E$, we define its increment as the two parameter function
$\delta X_{s,t} := X_t - X_s$ for $ (s,t) \in \Delta_{[0,T]}$.
Similarly, given a two-parameter function  $A:\Delta_{[0,T]} \to E$, $(s,t) \mapsto A_{s,t}$, we define $\delta A_{s,u,t}:= A_{s,t} - A_{s,u} - A_{u,t}$, for $s \le u \le t$.


\section{A compendium on rough SDEs} \label{section:compendiumonRSDEs}

In this section we recall the main definitions and results on rough stochastic analysis and rough SDEs contained in \cite{FHL21} and \cite[Section 3]{BCN24}. 
Even if the reader is familiar with the topic, we recommend giving a look to \cref{def:stochasticcontrolledroughpaths} (for the progressive measurability and integrability assumptions we add on stochastic controlled rough paths), to \cref{def:stochasticcontrolledvectorfield} (for a slightly generalized definition of stochastic controlled vector field), and to \cref{rmk:interpolation} (since interpolation inequalities in H\"older spaces are frequently used in the paper). \\

Let $T >0$ and let $(\Omega,\mathcal{F},\{\mathcal{F}_t\}_{t \in [0,T]},\mathbb{P})$ be a complete filtered probability space.
Let $p\in [2,\infty)$ and $\ q \in [p,\infty]$, and let $V,\bar{V},\tilde{V}$ be finite-dimensional Banach spaces. 
Let $\mathcal{G}\subseteq\mathcal{F}$ be a sub-$\sigma$-field. 
Given  a $V$-valued random variable $\xi$, we define - if it exists - its conditional $L^p$-norm with respect to  $\mathcal{G}$ as the (unique) $\mathcal{G}$-measurable $\R$-valued random variable given by \begin{equation*}
    \|\xi|\mathcal{G}\|_{p;V} := \E\left(|\xi|_V^p|\mathcal{G}\right)^{\frac{1}{p}}.
\end{equation*} 
We often abbreviate the notation as $\|\xi|\mathcal{G}\|_p$.  
In particular, the mixed $L_{p,q}$-norm $\|\|\xi |\mathcal{G}\|_p\|_q$ denotes the $L^q(\Omega)$-norm of $\|\xi|\mathcal{G}\|_p$.
We denote by $\E_t(\cdot) = \E( \, \cdot \mid \mathcal{F}_t)$ for any $ t\ in [0,T]$. 
Let $\gamma,\gamma',\delta,\delta' \in [0,1]$ and let $W \in {C}^\alpha([0,T];\R^n)$ be an $\alpha$-Hölder continuous deterministic path, for some $\alpha \in (0,1]$.

\begin{defn} \label{def:stochasticcontrolledroughpaths} (see \cite[Definition 3.1]{FHL21}) 
A pair $(X,X')$ is said to be a \textit{stochastic controlled rough path} of $(p,q)$-integrability and  $(\gamma,\gamma')$-Hölder regularity with values in $V$ if 
\begin{enumerate} \renewcommand{\labelenumi}{\alph{enumi})}
\item $X=(X_t)_{t \in [0,T]}$ is a progressively measurable $V$-valued stochastic processes such that 
\begin{equation*}
    \sup_{t \in [0,T]} \|X_t\|_p < +\infty \footnote{This condition is not present in \cite[Definition 3.1]{FHL21}, and therefore it is not necessary to develop a meaningful theory for rough SDEs. However, in \cref{section:parameterdependence} we are considering stochastic controlled rough paths for which it is necessary to assume $\|X_0\|_p < +\infty$. This is the main reason why we include this requirement in the definition of stochastic controlled rough path.  } \qquad \text{and} \qquad
    \|\delta X\|_{\gamma;p,q} := \sup_{0\le s < t\le T} \frac{\|\|\delta X_{s,t}|\mathcal{F}_s\|_p\|_q}{|t-s|^\gamma} < +\infty;
\end{equation*} 
\item $X'=(X'_t)_{t \in [0,T]}$ is a progressively measurable $\mathscr{L}(\R^n,V)$-valued stochastic processes such that
\begin{equation*}
    \sup_{t \in [0,T]} \|X'_t\|_q < +\infty \qquad \text{and} \qquad \|\delta X'\|_{\gamma';p,q} := \sup_{0\le s < t\le T} \frac{\|\|\delta X'_{s,t}|\mathcal{F}_s\|_p\|_q}{|t-s|^{\gamma'}} < +\infty;
\end{equation*}
\item denoting by $R^X_{s,t}=\delta X_{s,t} - X'_s \delta W_{s,t}$ for $(s,t)\in\Delta_{[0,T]}$, it holds that 
\begin{equation*}
    \|\E_\cdot R^X\|_{\gamma+\gamma';q}:=\sup_{0\le s < t \le T} \frac{\|\E_s(R^X_{s,t})\|_q}{|t-s|^{\gamma+\gamma'}} < + \infty.
\end{equation*} 
\end{enumerate}
We write $(X,X') \in \mathbf{D}_W^{\gamma,\gamma'} L_{p,q}(V)$ or simply $(X,X') \in \mathbf{D}_W^{\gamma,\gamma'} L_{p,q}$, if clear from the context. 
The process $X'$ is called the (stochastic) Gubinelli derivative of $X$ (with respect to $W$). In the case $p=q$, we simply write $\mathbf{D}_W^{\gamma,\gamma'} L_{p}$ and $\|\cdot\|_{\gamma;p}$ instead of $\mathbf{D}_W^{\gamma,\gamma'} L_{p,p}$ and $\|\cdot\|_{\gamma;p,p}$, respectively.
Whereas, if $\gamma=\gamma'$, we write $\mathbf{D}_W^{2\gamma} L_{p,q}$ instead of $\mathbf{D}_W^{\gamma,\gamma} L_{p,q}$.
\end{defn}
\noindent For a process $(X,X') \in \mathbf{D}_W^{\gamma,\gamma'} L_{p,q}$ its seminorm is defined by\footnote{This definition is consistent with \cite{FHL21}, hence does not include the $\sup_t \| X_t \|_p$ that we included in our previous definition.}
\begin{equation*}
    \|(X,X')\|_{\mathbf{D}_W^{\gamma,\gamma'}L_{p,q}} :=  \|\delta X\|_{\gamma;p,q} + \sup_{t \in [0,T]}\|X'_t\|_q + \|\delta X'\|_{\gamma';p,q} + \|\E_\cdot R^X\|_{\gamma+\gamma';q}.
\end{equation*}
Given another stochastic controlled rough path $(\bar{X},\bar{X}') \in \mathbf{D}_{\bar{W}}^{\gamma,\gamma'} L_{p,q}$ for some $\bar{W} \in C^\alpha([0,T];\R^n)$, we define their distance as 
\begin{multline*}
    \|X,X';\bar{X},\bar{X}'\|_{W,\bar{W};\gamma,\gamma';p} := \sup_{t \in [0,T]}\|X_t-\bar{X}_t\|_p + \|\delta (X-\bar{X})\|_{\gamma;p}  + \sup_{t \in [0,T]} \|X_t' - \bar{X}_t'\|_p + \\
    + \|\delta (X'-\bar{X}')\|_{\gamma';p} + \|\E_\cdot R^X - \E_\cdot \bar{R}^{\bar{X}}\|_{\gamma,\gamma';p}, 
\end{multline*}
where $\bar{R}^{\bar{X}}_{s,t}=\delta \bar{X}_{s,t} - \bar{X}'_s \delta \bar{W}_{s,t}$.
If $\bar{W}=W$, we often write $\|\ \cdot \ ; \ \cdot \ \|_{W;\gamma,\gamma';p}$ instead of $\|\ \cdot \ ; \ \cdot \|_{W,W;\gamma,\gamma';p}$.
The space $\mathbf{D}_W^{\gamma,\gamma'} L_{p,q}$ endowed with the distance $d((X,X'), (\bar{X},\bar{X}')):= \|X,X';\bar{X},\bar{X}'\|_{W;\gamma,\gamma';p}$ is a Banach space.

\begin{defn} \label{def:stochasticboundedandLipschitzvectorfield}
    (see \cite[Definition 4.1]{FHL21})
    We say that $g:\Omega\times[0,T]\times V\to \bar{V}$ is a \textit{random bounded vector field} from $V$ to $\bar{V}$ if $(\omega,t) \mapsto g_t(\omega,\cdot)$ is a progressively strongly measurable stochastic process taking values in $C^0_b(V;\bar{V})$ and there exists a deterministic constant, denoted by $\|g\|_\infty$, such that
				\[
					\sup_{t \in [0,T]} \esssup_\omega \sup_{x \in  V}|g_t(\omega,x)|\le\|g\|_\infty.
				\]
	We say that $g$ is a \textit{random bounded Lipschitz vector field} from $V$ to $\bar{V}$ if it is a random bounded continuous vector field, strongly measurable from $\Omega \times [0,T] \to Lip(V;\bar{V})$,  and there exists a deterministic constant, denoted by $\|g\|_{Lip}$, such that
				\[
					\sup_{t\in [0,T]}\esssup_\omega \sup_{x,\bar x \in V, x \ne \bar x}  \frac{|g(\omega,t,x)-g(\omega,t,\bar x)|}{|x-\bar x|}\le \|g\|_{Lip}.
				\]
\end{defn}

\begin{defn} \label{def:stochasticcontrolledvectorfield}
    (generalization of \cite[Definition 3.8]{FHL21})
    Let $\kappa > 1$ and denote by $\lceil \kappa \rceil \in \mathbb{N}_{\ge 2}$ its ceiling part.
    A \textit{stochastic controlled vector field} in $\mathbf{D}_W^{2\delta} L_{p,q} C_b^\kappa(V;\bar{V})$ - or simply in $\mathbf{D}_W^{2\delta} L_{p,q} C_b^\kappa$ - is a pairing $(\beta,\beta')$ satisfying the following: \begin{enumerate}  \renewcommand{\labelenumi}{\alph{enumi})}
        \item $\beta: \Omega \times [0,T] \times V \to \bar{V}$ is progressively measurable such that, for $\mathbb{P}$-almost every $\omega$ and for any $t \in [0,T]$, $\beta_t(\cdot) := \beta(\omega,t,\cdot) \in C^\kappa_b(V;\bar{V})$ and
        \begin{equation*}
            \sup_{t \in [0,T]} \| |\beta_t(\cdot)|_{C^\kappa_b}\|_q := \sup_{t \in [0,T]} \| |\beta_t(\cdot)|_{C_b^{\lfloor \kappa \rfloor}} + [D^{\lfloor \kappa \rfloor}_{x\dots x} \beta_t(\cdot)]_{\kappa-\lfloor \kappa \rfloor} \|_q  < +\infty;
        \end{equation*}
        \item $\beta': \Omega \times [0,T] \times V \to \mathcal{L}(\R^n;\bar{V})$ is progressively measurable such that, for $\mathbb{P}$-almost every $\omega$ and for any $t \in [0,T]$, $\beta_t'(\cdot) := \beta'(\omega,t,\cdot) \in C^{\kappa-1}_b(V;\mathcal{L}(\R^n;\bar{V}))$ and \begin{equation*}
            \sup_{t \in [0,T]} \| |\beta'_t(\cdot)|_{C^{\kappa-1}_b}\|_q < +\infty;
        \end{equation*}
        \item denoting by $R^\beta_{s,t}(\cdot) := \beta_t(\cdot) - \beta_s(\cdot) - \beta_s'(\cdot) \delta W_{s,t}$ and letting $\llbracket g \rrbracket_{\delta;p,q} := \sup_{0 \le s < t \le T} \frac{\| \| |g(\cdot)|_\infty \mid \mathcal{F}_s\|_p \|_q}{|t-s|^\delta}$, \begin{align*}
          \llbracket (\beta,\beta') \rrbracket_{W;2\delta;p,q;\kappa}  &:= \llbracket \delta \beta \rrbracket_{\delta;p,q} + \llbracket \delta D_x\beta \rrbracket_{\delta;p,q} + \dots + \llbracket \delta D_{x\dots x}^{\lceil \kappa \rceil-1}\beta \rrbracket_{\delta;p,q} + \llbracket \delta \beta' \rrbracket_{\delta;p,q} + \dots + \\
          &\quad + \llbracket \delta D_{x \dots x}^{\lceil \kappa \rceil-2} \beta' \rrbracket_{\delta;p,q} + \llbracket \E_\cdot R^\beta \rrbracket_{2\delta;q} +  \dots + \llbracket \E_\cdot R^{D^{\lceil \kappa \rceil-2}_{x \dots x}\beta} \rrbracket_{2\delta;q} < +\infty.
        \end{align*}
    \end{enumerate}
    We denote by \begin{equation*}
        \llbracket (\beta,\beta') \rrbracket_{\kappa;q} :=  \sup_{t \in [0,T]} \left( \| |\beta_t(\cdot)|_{C^\kappa_b}\|_q + \| |\beta'_t(\cdot)|_{C^{\kappa-1}_b}\|_q \right).
    \end{equation*} 
    Given another $\bar{W} \in C^\alpha([0,T];\R^n)$ and given $(\bar{\beta},\bar{\beta}') \in \mathbf{D}_{\bar{W}}^{2\delta}L_{p,q}C_b^\kappa(V;\bar{V})$, we denote by \begin{multline*}
            \llbracket \beta,\beta';\bar{\beta},\bar{\beta}' \rrbracket_{W,\bar{W};2\delta;p;\kappa}  := \llbracket \delta (\beta-\bar{\beta}) \rrbracket_{\delta;p}  + \dots + \llbracket (\delta D_{x\dots x}^{\lceil \kappa \rceil -1}\beta -  D_{x\dots x}^{\lceil \kappa \rceil -1}\bar{\beta}) \rrbracket_{\delta;p}  + \\
            + \llbracket \delta \beta' \rrbracket_{\delta;p,q} + \dots + \llbracket \delta (D_{x \dots x}^{\lceil \kappa \rceil-2} \beta'- D_{x \dots x}^{\lceil \kappa \rceil-2} \bar{\beta}') \rrbracket_{\delta;p,q} 
            + \llbracket \E_\cdot R^\beta \rrbracket_{2\delta;q} + \dots + \llbracket (\E_\cdot R^{D^{\lceil \kappa \rceil -2}_{x \dots x}\beta}- \E_\cdot \bar{R}^{D^{\lceil \kappa \rceil -2}_{x \dots x}\bar{\beta}}) \rrbracket_{2\delta;q}.
    \end{multline*}
    If $p=q$, we simply write $\mathbf{D}_W^{2\delta} L_{p} C_b^\kappa$ instead of $\mathbf{D}_W^{2\delta} L_{p,p} C_b^\kappa$.
    If $\beta$ and $\beta'$ are deterministic functions satisfying all the requirements above, we write $(\beta,\beta') \in \mathscr{D}_W^{2\delta}C_b^\kappa$, we say it is a deterministic controlled vector field and we drop the integrability exponents $p,q$ when denoting the corresponding norms.  
\end{defn}

\begin{rmk} \label{rmk:propertiesofstochasticcontrolledvectorfields}
    The previous definition is formulated such that by taking the derivative of a stochastic controlled vector we still obtain a stochastic controlled vector field. More precisely, given $\kappa \ge 2$ and $(\beta,\beta') \in \mathbf{D}_W^{2\delta}L_{p,q}C_b^\kappa(V;\bar{V})$, it holds that \begin{equation*}
            (D_x\beta, D_x\beta') \in \mathbf{D}_W^{2\delta}L_{p,q}C_b^{\kappa-1}(V;\mathcal{L}(V;\bar{V})).
    \end{equation*}
    This is useful when one needs to prove recursive properties, but the prize we pay is that we need to assume the same H\"older regularity $\delta$ on every object. (To wit, \cite{FHL21} would allows to replace the above exponent $2\delta$ by a suitable sum $\delta+\delta'$.)
    Let us also notice that, for any $\kappa >1$ and for any $(\beta,\beta') \in \mathbf{D}_W^{2\delta}L_{p,q}C_b^\kappa(V \times \tilde{V};\bar{V})$,   \begin{equation*}
            (\beta(\cdot,\tilde{x}), \beta'(\cdot,\tilde{x})) \in \mathbf{D}_W^{2\delta}L_{p,q}C_b^{\kappa-1}(V;\mathcal{L}(V;\bar{V})) \quad \text{for any $\tilde{x} \in \tilde{V}$},
        \end{equation*}
        and $(D_{\tilde{x}}\beta, D_{\tilde{x}}\beta') \in \mathbf{D}_W^{2\delta}L_{p,q}C_b^{\kappa-1}(V \times \tilde{V};\mathcal{L}(\tilde{V};\bar{V}))$. 
\end{rmk}

\begin{prop} \label{prop:compositionwithstochasticcontrolledvectorfields_compendium}
    (see \cite[Lemma 3.11]{FHL21} and \cite[Proposition 3.13]{FHL21}) \begin{itemize}
        \item Let $(\beta,\beta') \in \mathbf{D}_W^{2\delta}L_{p,\infty}C_b^\kappa(V;\bar{V})$ for some $\kappa \in (1,2]$ and let $(X,X') \in \mathbf{D}_W^{\gamma,\gamma'} L_{p,\infty}(V)$.  Define $(Z,Z'):= (\beta(X), D_x\beta(X)X' + \beta'(X))$. Then $(Z,Z') \in \mathbf{D}_W^{\eta,\eta'} L_{p,\infty}(\bar{V})$ with \begin{equation*}
            \|(Z,Z')\|_{\mathbf{D}_W^{\eta,\eta'} L_{p,\infty}} \lesssim_{T} ( \llbracket (\beta,\beta') \rrbracket_{{\kappa};\infty} + \llbracket (\beta,\beta') \rrbracket_{W;2\delta';p,\infty;{\kappa}} ) (1+\|(X,X')\|_{\mathbf{D}_W^{\gamma,\gamma'} L_{p,\infty}}^{\kappa}),
        \end{equation*}
        where $\eta=\min\{\delta,\gamma\}$ and $\eta'=\min\{(\kappa-1)\gamma,\gamma',\delta\}$.
        \item \textcolor{black}{Let $(X,X') \in \mathbf{D}_W^{\gamma,\gamma'} L_{p,\infty}(V)$ and $(\bar{X},\bar{X}') \in \mathbf{D}_{\bar{W}}^{\gamma,\gamma'} L_{p,\infty}(V)$, and let $M>0$ be any constant such that $\max \{\|(X,X')\|_{\mathbf{D}_W^{\gamma,\gamma'} L_{p,\infty}} , \|(\bar{X},\bar{X}')\|_{\mathbf{D}_{\bar{W}}^{\gamma,\gamma'} L_{p,\infty}} \} \le M$. 
        Let $\kappa \in (2,3]$. Let $(\beta,\beta') \in \mathbf{D}_W^{2\delta}L_{p,\infty}C_b^\kappa(V;\bar{V})$ and $(\bar{\beta},\bar{\beta}') \in \mathbf{D}_{\bar{W}}^{2\delta}L_{p,\infty}C_b^{\kappa-1}(V;\bar{V})$, and define $(Z,Z')$ and $(\bar{Z},\bar{Z}')$ as before. Then \begin{equation*}
            \|Z,Z';\bar{Z},\bar{Z}'\|_{W,\bar{W};\lambda,\lambda';p} \lesssim \|X,X';\bar{X},\bar{X}'\|_{W,\bar{W};\lambda,\lambda';p} + \llbracket(\beta-\bar{\beta}, \beta'-\bar{\beta}')\rrbracket_{\kappa-2;p} + \llbracket \beta,\beta';\bar{\beta},\bar{\beta}' \rrbracket_{W,\bar{W};\lambda,\lambda';p;\kappa-1}
        \end{equation*}
        where $\lambda=\min\{\gamma,\delta\}$,  $\lambda'=\min\{(\kappa-2)\gamma,\gamma',\delta\}$ and the implicit constant depends on $M, \llbracket (\beta,\beta') \rrbracket_{{\kappa};\infty}$ and $\llbracket (\beta,\beta') \rrbracket_{W;2\delta';p,\infty;{\kappa}}$. }
    \end{itemize} 
\end{prop}

\begin{defn}\label{def:stochasticlinearvectorfields}
(see \cite[Definition 3.1]{BCN24})
We say that $Q:\Omega \times [0,T] \times V \to \bar{V}$ is a \textit{stochastic linear vector field} from $V$ to $\bar{V}$ if the mapping $(\omega,t) \mapsto Q_t(\omega,t,\cdot) \in \mathcal{L}(V;\bar{V})$ is  progressively strongly measurable and there is a deterministic constant $\|Q\|_{\infty} >0$ such that \begin{equation*}
       \sup_{t \in [0,T]} \esssup_{\omega} | Q_t(\omega,\cdot) |_{\mathcal{L}(V,\bar{V})} \le \|Q\|_\infty <+ \infty.  \end{equation*} 
We denote by $\|Q\|_{\mathcal{L};p} := \sup_{t \in [0,T]} \| | Q_t |_{\mathcal{L}(V,\bar{V})} \|_p$, where $|\cdot|_{\mathcal{L}(V,\bar{V})}$ denotes the standard linear operator norm. If clear from the context we simply denote such a norm by $|\cdot|_\mathcal{L}$.

\end{defn}

\begin{defn} \label{def:stochasticcontrolledlinearvectorfields} 
(see \cite[Definition 3.2]{BCN24})
A pair $(f,f')$ is said to be a \textit{stochastic controlled linear vector field} from $V$ to $\bar{V}$ of $(p,q)$-integrability and $(\delta,\delta')$-H\"older regularity - and we write $(f,f') \in \mathbf{D}^{\delta,\delta'}_W L_{p,q} \mathcal{L} (V,\bar{V})$ or simply $(f,f') \in \mathbf{D}^{\delta,\delta'}_W L_{p,q} \mathcal{L}$ - if \begin{itemize}
    \item[a)] $f:\Omega \times [0,T] \times V \to \bar{V}$ is a stochastic linear vector field with \begin{equation*} \label{eq:normofcontrolledvectorfields1}
        \|\delta f\|_{\delta;p,q}:=\sup_{0\le s<t\le T} \frac{\|\||f_t-f_s|_{\mathcal{L}(V,\bar{V})}|\mathcal{F}_s\|_p\|_{q}}{|t-s|^{\delta}}<+\infty;
    \end{equation*}
    \item[b)] $f':\Omega \times [0,T] \times V \to \mathcal{L}(\R^n, \bar{V})$ is a stochastic linear vector field with \begin{equation*} \label{eq:normofcontrolledvectorfields2}
        \|\delta f'\|_{\delta';p,q}:=\sup_{0\le s<t\le T} \frac{\|\||f'_t-f'_s|_{\mathcal{L}(V,\mathcal{L}(\R^n,\bar{V}))}|\mathcal{F}_s\|_p\|_{q}}{|t-s|^{\delta'}}<+\infty;
    \end{equation*}
    \item[c)] denoting by $R_{s,t}^f:=f_t-f_s-f'_s\delta Z_{s,t}$ for any $(s,t)\in \Delta_{[0,T]}$, it holds \begin{equation*} \label{eq:remainderofstochasticcontrolledlinearvectorfields} \|\E_{\cdot}R^f\|_{\delta+\delta';q}:=\sup_{0\le s<t\le T} \frac{\||\E_s(R_{s,t}^f)|_{\mathcal{L}(V,\bar{V})}\|_{q}}{|t-s|^{\delta+\delta'}}<+\infty.
    \end{equation*}  
\end{itemize} 
\end{defn}

\noindent Given $(f,f') \in \mathbf{D}^{\delta,\delta'}_W L_{p,q} \mathcal{L}$ we define \begin{equation*} \begin{aligned}
    \llbracket (f,f') \rrbracket_{\mathcal{L};q} &:= \sup_{t\in [0,T]} \||f_t|_{\mathcal{L}}\|_q + \sup_{t\in [0,T]} \||f'_t|_{\mathcal{L}}\|_q \\
    \llbracket (f,f') \rrbracket_{W;\delta,\delta';p,q} &:= \|\delta f\|_{\delta;p,q} + \|\delta f'\|_{\delta';p,q} + \|E_\cdot R^f\|_{\delta+\delta';q}.
\end{aligned}
\end{equation*}
Given another $\bar{W}\in C^\alpha([0,T];\R^n)$ and given $(\bar{f},\bar{f}') \in \mathbf{D}^{\delta,\delta'}_{\bar{W}} L_{p,q} \mathcal{L}$, we denote by
\begin{equation*}
    \llbracket f,f';\bar{f},\bar{f}' \rrbracket_{W,\bar{W};\delta,\delta';p;\mathcal{L}} := \|\delta(f-\bar{f})\|_{\delta;p} + \|\delta (f'-\bar{f}')\|_{\delta';p} + \|E_\cdot R^f - \E_\cdot \bar{R}^{\bar{f}}\|_{\delta+\delta';p},   
\end{equation*}
where $\bar{R}^{\bar{f}}_{s,t} = \delta \bar{f}_{s,t} - \bar{f}_s' \delta \bar{W}_{s,t}$. 

\begin{prop} \label{prop:stambilityundercompositionlinearvectorfields} 
\begin{enumerate}
    \item Let $(Y,Y') \in \mathbf{D}_W^{\gamma,\gamma'}L_{p}(V)$ 
    and let $(f,f')\in  \mathbf{D}^{\delta,\delta'}_W L_{p,\infty}\mathcal{L}(V,\bar{V})$. 
    Define $(f,f')Y := (f Y,f'Y+fY')$. Then $(f,f')Y \in \mathbf{D}_W^{\eta,\eta'}L_{p}(\bar{V})$ and \begin{equation*}
    \|(f,f')Y\|_{\mathbf{D}_W^{\eta,\eta'}L_p} \le \Big(\sup_{t \in [0,T]} \|Y_t\|_p + \|(Y,Y')\|_{\mathbf{D}_W^{\eta,\eta'}L_{p}}\Big) \left(\llbracket(f,f')\rrbracket_{\mathcal{L};\infty} + \llbracket(f,f')\rrbracket_{\mathbf{D}_W^{\eta,\eta'}L_{p,\infty} \mathcal{L}}\right),
    \end{equation*}
    where $\eta=\min\{\gamma,\delta\}$ and $\eta'=\min\{\gamma,\delta,\gamma',\delta'\}$. 
    \item Let $(Y,Y') \in \mathbf{D}_W^{\gamma,\gamma'}L_{2p}(\R^d)$ 
    and let $(f,f') \in \mathbf{D}_W^{\delta,\delta'}L_{2p}\mathcal{L}(\R^d;V)$.
    Let $(\bar{Y},\bar{Y}') \in \mathbf{D}_{\bar{W}}^{\gamma,\gamma'}L_{2p}(\R^d)$ 
    and let $(\bar{f},\bar{f}') \in \mathbf{D}_{\bar{W}}^{\delta,\delta'}L_{2p}\mathcal{L}(\R^d;V)$.
    Let $M>0$ be any constant such that \begin{multline*}
        \llbracket (f,f') \rrbracket_{\mathcal{L};2p} + \llbracket (f,f') \rrbracket_{W;\delta,\delta';2p;\mathcal{L}} + \llbracket (\bar{f},\bar{f}') \rrbracket_{\mathcal{L};2p} + \llbracket (\bar{f},\bar{f}') \rrbracket_{\bar{W};\delta,\delta';2p;\mathcal{L}} + \\
        + \sup_{t \in [0,T]}\|Y_t\|_{2p} + \|(Y,Y')\|_{\mathbf{D}_W^{\gamma,\gamma'}L_{2p}} + \sup_{t \in [0,T]}\|\bar{Y}_t\|_{2p} +\|(\bar{Y},\bar{Y}')\|_{\mathbf{D}_{\bar{W}}^{\gamma,\gamma'}L_{2p}} \le M.
    \end{multline*}
    Denote by $(Z,Z')=(fY,f'Y+fY')$ and, similarly, $(\bar{Z},\bar{Z}')=(\bar{f}\bar{Y},\bar{f}'\bar{Y}+\bar{f}\bar{Y}')$.
    Then $(Z,Z') \in \mathbf{D}_W^{\lambda,\lambda'}L_p(V)$, $ (\bar{Z},\bar{Z}') \in \mathbf{D}_{\bar{W}}^{\lambda,\lambda'}L_p(V)$, and
    \begin{equation*}
        \|Z,Z';\bar{Z},\bar{Z}'\|_{W,\bar{W};\lambda,\lambda';p} \lesssim_M  \llbracket f,f';\bar{f},\bar{f}' \rrbracket_{\mathcal{L};2p} + \llbracket f,f';\bar{f},\bar{f}' \rrbracket_{W,\bar{W};\lambda,\lambda';2p} + 
        \|Y,Y';\bar{Y},\bar{Y}'\|_{W,\bar{W};\lambda,\lambda';p},
    \end{equation*}
    where $\lambda=\min\{\gamma,\delta\}$ and $\lambda' = \min\{\gamma,\gamma',\delta,\delta'\}$.
\end{enumerate}
\end{prop}

\begin{proof}
    \textit{1. } Straightforward generalization of \cite[Proposition 3.3]{BCN24}. \\

    \textit{2. } We just prove the last assertion; the fact that both $(Z,Z')$ and $(\bar{Z},\bar{Z}')$ are $L_p$-integrable stochastic controlled rough path can be proved by using similar techniques as in \textit{1}. 
    Let $(s,t) \in \Delta_{[0,T]}$. 
    For sake of compact notation we denote by $\tilde{Y}=Y-\bar{Y}$ and $\tilde{f}=f-\bar{f}$.
    One trivially has that \begin{equation*}
        f_tY_t -\bar{f}_t\bar{Y}_t =  f_t\tilde{Y}_t + \tilde{f}_t\bar{Y}_t
    \end{equation*} and \begin{equation*}
        f_t'Y_t+f_tY_t' - \bar{f}_t'\bar{Y}_t-\bar{f}_t\bar{Y}'_t = f_t' \tilde{Y}_t + \tilde{f}'_t \bar{Y}_t + f_t \tilde{Y}'_t + \tilde{f}_t \bar{Y}_t.
    \end{equation*}
    Moreover, we can write \begin{equation*}
        \delta (f Y-\bar{f} \bar{Y})_{s,t} = \delta (f \tilde{Y}+\tilde{f} \bar{Y})_{s,t} =  f_t \delta \tilde{Y}_{s,t} + \delta f_{s,t}\tilde{Y}_s+\tilde{f}_t\delta \bar{Y}_{s,t}+\delta \tilde{f}_{s,t}\bar{Y}_s
    \end{equation*}
    and \begin{equation*}
        \begin{aligned}
        \delta (f' Y + f Y' - \bar{f}' \bar{Y} - \bar{f} \bar {Y}' )_{s,t} &= \delta (f' \tilde{Y} +  \tilde{f}'\bar{Y}+ f \tilde{Y}' + \tilde{f} \bar{Y}')_{s,t} = \\ 
        &= f'_t\delta\tilde{Y}_{s,t} + \delta f'_{s,t}\tilde{Y}_s+\tilde{f}_t'\delta \bar{Y}_{s,t} + \delta \tilde{f}'_{s,t}\bar{Y}_s+f_t\delta\tilde{Y}'_{s,t}+\delta f_{s,t}\tilde{Y}'_s + \tilde{f}_t \delta \bar{Y}_{s,t}' + \delta \tilde{f}_{s,t} \bar{Y}'_s.
    \end{aligned}
    \end{equation*}
    By linearity, it is possible to write \begin{equation*}
        \begin{aligned}
            {R}_{s,t}^{{f} {Y}} - \bar{R}_{s,t}^{\bar{f} \bar{Y}} &= \delta (f Y)_{s,t} - (f'_sY_s+f_sY'_s)\delta{W}_{s,t} - \delta (\bar{f} \bar{Y})_{s,t} + (\bar{f}_s\bar{Y}_s+\bar{f}_s\bar{Y}'_s)\delta{\bar{W}}_{s,t} = \\ 
            &= f_t\delta\tilde{Y}_{s,t} + \delta f_{s,t} \tilde{Y}_s + \tilde{f}_t \delta \bar{Y}_{s,t}  + \delta \tilde{f}_{s,t}\bar{Y}_s- f_sY_s'\delta W_{s,t} + \bar{f}_s\bar{Y}'_s \delta \bar{W}_{s,t} - f'_sY_s \delta W_{s,t} + \bar{f}'_s\bar{Y}_s \delta \bar{W}_{s,t} = \\
            &= \delta f_{s,t} \delta\tilde{Y}_{s,t} + f_s (R^Y_{s,t}-\bar{R}^{\bar{Y}}_{s,t}) + (R^f_{s,t} - \bar{R}^{\bar{f}}_{s,t}) \bar{Y}_s + R^f_{s,t} \tilde{Y}_s + \tilde{f}_t \delta \bar{Y}_{s,t} - \tilde{f}_s \bar{Y}'_s \delta \bar{W}_{s,t} = \\
            &= \delta f_{s,t} \delta\tilde{Y}_{s,t} + f_s (R^Y_{s,t}-\bar{R}^{\bar{Y}}_{s,t}) + (R^f_{s,t} - \bar{R}^{\bar{f}}_{s,t}) \bar{Y}_s + R^f_{s,t} \tilde{Y}_s + \delta \tilde{f}_{s,t} \delta \bar{Y}_{s,t} + \tilde{f}_s \bar{R}^{\bar{Y}}_{s,t}.
        \end{aligned}
    \end{equation*}
    The conclusion follows by applying H\"older's inequality to all the previous identities. As an example, we notice that \begin{equation*}
        \|\E_s(\delta f_{s,t} \delta\tilde{Y}_{s,t})\|_p \le \|\delta f_{s,t}\|_{2p} \|\delta\tilde{Y}_{s,t}\|_{2p} \lesssim_M \|Y,Y';\bar{Y},\bar{Y}'\|_{W,\bar{W};\lambda,\lambda';2p} |t-s|^{\lambda+\lambda'}.
    \end{equation*}
\end{proof}

\begin{defn} \label{def:roughpath} 
(see \cite[Definition 2.1, Definition 2.4 and Section 2.2]{FRIZHAIRER})
Let $\alpha \in (\frac{1}{3},\frac{1}{2}]$.
An $\alpha$-\textit{rough path} on $[0,T]$ with values in $\R^n$ is a pair $\mathbf{W}=(W,\mathbb{W})$ such that \begin{enumerate}
    \item[a)] $W:[0,T] \to \R^n$ with $|\delta W|_\alpha:= \sup_{s \ne t} \frac{|W_t-W_s|_{\R^n}}{|t-s|^\alpha}< +\infty$ ;
    \item[b)] $\mathbb{W}:[0,T]^2 \to \R^n \otimes \R^n$ with $|\mathbb{W}|_{2\alpha} := \sup_{s\ne t}\frac{|\mathbb{W}_{s,t}|_{\R^n \otimes \R^n}}{|t-s|^{2\alpha}}< +\infty$;
    \item[c)] (Chen's relation) for any $s,u,t \in [0,T]$, it holds \begin{equation*}
        \mathbb{W}_{s,t}-\mathbb{W}_{s,u}-\mathbb{W}_{u,t} = \delta W_{s,u}\otimes\delta W_{u,t}
    \end{equation*} or, equivalently, $\mathbb{W}_{s,t}^{ij}-\mathbb{W}_{s,u}^{ij}-\mathbb{W}_{u,t}^{ij} = \delta W_{s,u}^i \delta W_{u,t}^j$ for any $i,j=1,\dots,n$.
\end{enumerate}
We write $\mathbf{W}\in \mathscr{C}^\alpha([0,T];\R^n)$, and we define $|\mathbf{W}|_\alpha := |\delta W|_\alpha + |\mathbb{W}|_{2\alpha}$.
Given any pair of rough paths $\mathbf{W}=(W,\mathbb{W}), \ \bar{\mathbf{W}}=(\bar{W},\bar{\mathbb{W}}) \in \mathscr{C}^\alpha([0,T];\R^n)$,  we define their distance as \begin{equation*} \begin{aligned}
    \rho_{\alpha}(\mathbf{W},\bar{\mathbf{W}}) &:= |\delta (W -\bar{W})|_\alpha + |\mathbb{W}-\bar{\mathbb{W}}|_{2\alpha}. 
\end{aligned}
\end{equation*} 
An $\alpha$-rough path $\mathbf{W}=(W,\mathbb{W})$ is said to be \textit{weakly geometric} if, for any $s,t \in [0,T]$ \begin{equation*}
Sym(\mathbb{W}_{s,t}):=\frac{\mathbb{W}_{s,t}+\mathbb{W}_{s,t}^\top}{2} = \frac{1}{2}\delta W_{s,t} \otimes \delta W_{s,t}.
\end{equation*}
or, equivalently, $\frac{1}{2}(\mathbb{W}^{ij}+\mathbb{W}^{ji}) = \frac{1}{2}\delta W^i \delta W^j$ for any $i,j=1,\dots,n$. 
We write $\mathbf{W}\in \mathscr{C}^\alpha_g([0,T];\R^n)$.
\end{defn}

\begin{prop}\label{prop:roughstochastiintegral_compendium}
(see \cite[Theorem 3.5]{FHL21})
Let $\mathbf{W}=(W,\mathbb{W}) \in \mathscr{C}^{\alpha}([0,T];\mathbb{R}^n)$ be an $\alpha$-rough path with $\alpha \in (\frac{1}{3},\frac{1}{2}]$ and let $(\phi,\phi')\in \mathbf{D}^{\gamma,\gamma'}_W L_{p,q}(\mathcal{L}(\R^n;V))$ be a stochastic controlled rough path with \begin{equation*}
    \alpha + \gamma > \frac{1}{2}  \qquad \text{and} \qquad \alpha + \min\{\alpha,\gamma\} + \gamma' >1.
    \end{equation*} 
    Then there exists a unique continuous and adapted $V$-valued stochastic process $\int_0^\cdot (\phi_r,\phi'_r) d\mathbf{W}_r$, called the {rough stochastic integral} of $(\phi,\phi')$ against $\mathbf{W}$, such that \begin{equation*}
        \|\|\int_s^t (\phi_r,\phi'_r) d\mathbf{W}_r -\phi_s\delta W_{s,t} - \phi'_s \mathbb{W}_{s,t}|\mathcal{F}_s\|_p \|_q \lesssim |t-s|^{\frac{1}{2}+\varepsilon}
    \end{equation*} and \begin{equation*} \|\E_s(\int_s^t (\phi_r,\phi'_r) d\mathbf{W}_r -\phi_s\delta W_{s,t} - \phi'_s \mathbb{W}_{s,t}) \|_q \lesssim |t-s|^{1+\varepsilon}
    \end{equation*}
    for any $(s,t) \in \Delta_{[0,T]}$ and for some $\varepsilon>0$. Here $\int_s^t (\phi_r,\phi'_r) d\mathbf{W}_r = \int_0^t (\phi_r,\phi'_r) d\mathbf{W}_r - \int_0^s (\phi_r,\phi'_r) d\mathbf{W}_r $. 
    Moreover, \begin{equation*}
        \Big(\int_0^\cdot(X_r,X_r') d\mathbf{W}_r, X \Big) \in \mathbf{D}^{\alpha,\min\{\alpha,\gamma\}}_W L_{p,q}(V).
    \end{equation*}
\end{prop}

Let $B=(B_t)_{t\in [0,T]}$ be an $m$-dimensional Brownian motion on $(\Omega,\mathcal{F},\{\mathcal{F}_t\}_{t\in [0,T]},\mathbb{P})$ and let $\mathbf{W}=(W,\mathbb{W})\in \mathscr{C}^\alpha([0,T];\R^n)$ with $\alpha\in (\frac{1}{3},\frac{1}{2}]$. 
Let $b:\Omega \times [0,T] \times \R^d \to \R^d$, $\sigma:\Omega \times [0,T] \times \R^d \to \mathcal{L}(\R^m;\R^d)$, $\beta:\Omega \times [0,T] \times \R^d \to \mathcal{L}(\R^n;\R^d)$, $\beta':\Omega \times [0,T] \times \R^d \to \mathscr{L}(\R^n \otimes \R^n;\R^d)$, $F:\Omega \times [0,T] \times \R^d \to \R^d$ and $F':\Omega \times [0,T] \times \R^d \to \mathscr{L}(\R^n;\R^d)$  be progressively measurable. Assume $x \mapsto \beta(\omega,t,x)$ to be (Fréchet) differentiable. 
We define solutions to
\begin{equation} \label{eq:RSDE_compendium}
    dX_t = dF_t + b_t(X_t) \, dt+\sigma_t(X_t) \, dB_t+(\beta_t,\beta'_t)(X_t) \, d\mathbf{W}_t
\end{equation} 
according to the following (cfr.\ \cite[Definition 4.2]{FHL21} and \cite[Definition 3.6]{BCN24})

\begin{defn} \label{def:solutionRSDEs_compendium}
    An $L_{p,q}$-integrable solution of \eqref{eq:RSDE_compendium} over $[0,T]$ is a continuous and adapted $\R^d$-valued stochastic process $X$ such that the following conditions are satisfied:
		\begin{enumerate}\renewcommand{\labelenumi}{\alph{enumi})}
			\item $\int_0^T |b_r(X_r)|dr$ and $\int_0^T |(\sigma \sigma^\top)_r(X_r)|dr$ are finite $\mathbb{P}$-a.s.;
			\item $(\beta(X),D_x\beta(X)(\beta(X)+F')+\beta'(X))$ belongs to $\mathbf{D}_W^{\bar\alpha,\bar\alpha'}L_{p,q}(\mathcal{L}(\R^n,\R^d))$, 
			for some $\bar\alpha,\bar\alpha'\in(0,1]$ such that $\alpha+\min\{\alpha,\bar\alpha\}>\frac12$ and $\alpha+\min\{\alpha,\bar\alpha\}+\bar\alpha'>1$;
			\item $X$ satisfies the following stochastic Davie-type expansion for every $(s,t)\in \Delta_{[0,T]}$: \begin{multline*} 
			    \delta X_{s,t} = \delta F_{s,t} + \int_s^t b_r(X_r)dr + \int_s^t \sigma_r(X_r)dB_r + \\
                                 + \beta_s(X_s) \delta W_{s,t} + (D_x\beta_s(X_s)(\beta_s(X_s) + F'_s)+\beta'_s(X_s))\mathbb{W}_{s,t} + X^\natural_{s,t},
			\end{multline*} where $\|\|X^\natural_{s,t}|\mathcal{F}_s\|_p\|_q=o(|t-s|^\frac{1}{2})$ and $\|\E_s(X^\natural_{s,t})\|_q=o(|t-s|)$ as $|t-s| \to 0$. 
		\end{enumerate}
		When the starting position $X_0=\xi$ is specified, we say that $X$ is a solution starting from $\xi$. 
        Equivalently, c) can be replaced by \begin{itemize}
            \item[c')] $\mathbb{P}$-almost surely and for any $t \in [0,T]$, \begin{equation*}
                X_t = \xi + F_t + \int_0^t b_r(X_r) dr + \int_0^t \sigma_r(X_r) dB_r + \int_0^t (\beta_r(X_r),D_x\beta_r(X_r)(\beta_r(X_r) + F'_r) + \beta'_r(X_r)) d\mathbf{W}_r
            \end{equation*} 
        \end{itemize}
    and one has that $\|\|X^\natural_{s,t}|\mathcal{F}_s\|_p\|_q \lesssim |t-s|^{\alpha + \min\{\alpha, \bar \alpha\}}$ and $\|\E_s(X^\natural_{s,t})\|_q\lesssim |t-s|^{\alpha + \min\{\alpha, \bar \alpha\} + \bar \alpha'}$. We often denote by $(\beta_t,\beta_t')(X_t) := (\beta_t(X_t),D_x\beta_t(X_t)\beta_t(X_t) + \beta'_t(X_t))$.
    Notice that, if a solution exists then necessarily $(X,\beta(X) + F') \in \mathbf{D}_W^{\alpha,\min\{\alpha,\bar{\alpha}\}} L_{p,q}(\R^d)$. 
    We are interested in two different situations: \begin{enumerate}
    \item \textit{rough SDEs with bounded coefficients} (or simply \textit{rough SDEs}), as introduced in \cite{FHL21}, in which case $F=F'=0$, with solution theory that crucially requires $p < q =\infty$.
    \item \textit{rough SDEs with linear coefficients} (or \textit{linear rough SDEs}), as studied in \cite[Section 3]{BCN24}. In this case, we allow for general $F,F'$ and we assume $b,\sigma,\beta$ and $\beta'$ to be linear in $x$. It follows that $D_x\beta=\beta$, and b) can be directly deduced from c). In this case, the solution theory is possible for $p=q < \infty$. 
\end{enumerate}
\end{defn}

\begin{thm}\label{thm:wellposednessandaprioriestimatesforRSDEs}
    (see \cite[Theorem 4.6 and Proposition 4.5]{FHL21})
    Assume for simplicity $\delta \in [0,\alpha]$ and let $\kappa \in (\frac{1}{\alpha},3]$ such that $\alpha + (\kappa-1)\delta > 1$.  
    Let $b$ and $\sigma$ be two random bounded Lipschtiz vector fields from $\R^d$ to $\R^d$ and $\mathcal{L}(\R^m;\R^d)$, respectively, and let $(\beta, \beta') \in \mathbf{D}_W^{2\delta} L_{p,\infty} C_b^\kappa$ be a stochastic controlled vector field (in the sense of Definition \ref{def:stochasticcontrolledvectorfield}).
    Then, for any $\xi \in L^p(\Omega,\mathcal{F}_0;\R^d)$, there exists a unique solution $X$ to \begin{equation*}
        X_t = \xi  + \int_0^t b_r(X_r) dr + \int_0^t \sigma_r(X_r) dB_r + \int_0^t (\beta_r(X_r),D_x\beta_r(X_r)\beta_r(X_r) + \beta'_r(X_r)) d\mathbf{W}_r \qquad t \in [0,T]. 
    \end{equation*}
    Moreover, for any constant $M>0$ such that $\llbracket (\beta,\beta')  \rrbracket_{\kappa-1;\infty}+\llbracket(\beta,\beta')\rrbracket_{W;2\delta;p,\infty}\le M$, \begin{equation*}
        \|(X,X')\|_{\mathbf{D}_W^{\alpha,\delta}L_{p,\infty}} \lesssim_{T,p,\alpha,\delta,\kappa} (1+\|b\|_\infty+\|\sigma\|_\infty+M |\mathbf{W}|_\alpha)^\frac{1}{(\kappa-2)\delta}
    \end{equation*}
\end{thm}

\begin{thm} \label{thm:stabilityforRSDEs_compendium}
    (see \cite[Theorem 4.9]{FHL21}) 
    In the setting of \cref{thm:wellposednessandaprioriestimatesforRSDEs}, let $\bar{\mathbf{W}}=(\bar{W},\bar{\mathbb{W}}) \in \mathscr{C}^\alpha([0,T];\R^n)$, let $\bar{b}$ and $\bar{\sigma}$ be two random bounded Lipschtiz vector fields from $\R^d$ to $\R^d$ and $\mathcal{L}(\R^m;\R^d)$, respectively, and let $(\bar{\beta}, \bar{\beta'}) \in \mathbf{D}_{\bar{W}}^{2\delta} L_{p,\infty} C_b^\kappa$.
    Denote by $\bar{X}=(\bar{X}_t)_{t \in [0,T]}$ the solution of the rough SDE starting from $\bar{\xi} \in L^p(\Omega,\mathcal{F}_0;\R^d)$ and driven by $\bar{\mathbf{W}}$, with coefficients $\bar{b}, \bar{\sigma}$ and $(\bar{\beta}, \bar{\beta'})$. Then \begin{multline*} 
        \|X,\beta(X);\bar{X},\bar{\beta}(\bar{X})\|_{W,\bar{W};\alpha,\delta;p} \lesssim_{M,\alpha,\delta,p,T} \|\xi-\bar{\xi}\|_p + \rho_\alpha(\mathbf{W},\bar{\mathbf{W}}) + \sup_{t \in [0,T]} \| |b_t(\cdot)-\bar{b}_t(\cdot)|_\infty \|_p + \\
         + \sup_{t \in [0,T]} \| |\sigma_t(\cdot)-\bar{\sigma}_t(\cdot)|_\infty \|_p + \llbracket(\beta-\bar{\beta},\beta'-\bar{\beta}') \rrbracket_{\kappa-1;p} + \llbracket \beta,\beta';\bar{\beta},\bar{\beta}' \rrbracket_{W,\bar{W};2\delta;p;\kappa-1},
    \end{multline*}
    for any constant $M>0$ such that $|\mathbf{W}|_\alpha + |\bar{\mathbf{W}}|_\alpha + \|b\|_{Lip} + \|\sigma\|_{Lip} + \llbracket(\beta,\beta') \rrbracket_{\kappa;\infty} + \llbracket (\beta,\beta') \rrbracket_{W;2\delta;p,\infty;\kappa} \le M$. 
\end{thm}

\begin{thm} \label{thm:wellposednessandaprioriestimatesforlinearRSDEs}
    (see \cite[Theorem 3.11]{BCN24} and \cite[Theorem 3.10]{BCN24})
    Assume, for simplicity, $\gamma \in [0,\alpha]$, and let  $\alpha + \gamma > \frac{1}{2}$ and $\alpha + \gamma + \min\{\gamma,\gamma'\} >1$.
    Let $G$ and $S$ be two stochastic linear vector fields from $V$ to $V$ and $\mathcal{L}(\R^m;V)$, respectively, and let $(f,f') \in \mathbf{D}^{\gamma,\gamma'}_W L_{p,\infty}\mathscr{L}(V,\mathcal{L}(\R^n,V))$ be a stochastic controlled linear vector field. 
    We also consider a stochastic controlled rough path $(F,F') \in \mathbf{D}_W^{\alpha,\gamma} L_p(V)$ such that $F$ is $\mathbb{P}$-a.s.\ continuous. Such a path is often called the forcing term. 
    Then, for any $\xi \in L^p(\Omega,\mathcal{F}_{0};V)$, there exists a unique $L_p$-integrable solution $Y$ to the following linear rough SDE: \begin{equation*}
        Y_t = \xi + F_t + \int_0^t G_r Y_r dr + \int_0^tS_r Y_r dB_r + \int_0^t (f_rY_r, f_r(f_r Y_r + F'_r) + f'_r Y_r ) d\mathbf{W}_r \qquad t \in [0,T]. 
    \end{equation*}
    Moreover, \begin{equation*}
        \sup_{t \in [0,T]}\|Y_t\|_p + \|(Y,fY+F')\|_{\mathbf{D}_W^{\alpha,\gamma}L_p} \lesssim_{M,\alpha,\gamma,\gamma',p,T} \|\xi\|_p + \|(F,F')\|_{\mathbf{D}^{\alpha,\gamma}_W L_{p}}
    \end{equation*} where $M>0$ is any constant satisfying $\llbracket G \rrbracket_{\mathcal{L};\infty} + \llbracket S \rrbracket_{\mathcal{L};\infty} + \llbracket (f,f') \rrbracket_{\mathcal{L};\infty} + \llbracket (f,f') \rrbracket_{W;\gamma,\gamma';p,\infty} + |\mathbf{W}|_\alpha  \le M.$ 
\end{thm}

\begin{thm} \label{thm:stabilityforlinearRSDEs_compendium}
    In the setting of \cref{thm:wellposednessandaprioriestimatesforlinearRSDEs}, assume $\alpha \ne \frac{1}{2}$ and $p> \frac{1}{1-2\alpha}$. 
    Assume also $(f,f') \in \mathbf{D}_W ^{\gamma,\gamma'}L_{2p,\infty}\mathscr{L}(V,\mathscr{L}(\R^n,V))$ and $(F,F')\in \mathbf{D}^{\alpha,\gamma}_W L_{2p}(V)$. 
    Moreover, let $\bar{\mathbf{W}}=(\bar{W},\bar{\mathbb{W}}) \in \mathscr{C}^\alpha([0,T];\R^n)$, let $\bar{G}$, $\bar{S}$ be two stochastic linear vector fields from $V$ to $V$ and $\mathscr{L}(\R^m,V)$, respectively. 
    Let us also consider $(\bar{f},\bar{f}') \in \mathbf{D}^{\gamma,\gamma'}_{\bar{W}} L_{2p,\infty}\mathscr{L}(V,\mathscr{L}(\R^n,V))$, and  $(\bar{F},\bar{F}') \in \mathbf{D}^{\alpha,\gamma}_{\bar{W}} L_{2p}(V)$ such that $\bar{F}$ is $\mathbb{P}$-a.s.\ continuous. Let $\bar{\xi}\in L^{2p}(\Omega,\mathcal{F}_0;V)$, and let us denote by $\bar{Y}=(\bar{Y}_t)_{t \in [0,T]}$ the $L_{2p}$-integrable solution to the linear rough SDE driven by $\bar{\mathbf{W}}$, with coefficients $\bar{G},\bar{S},(\bar{f},\bar{f}')$ and forcing term $\bar{F}$, starting from $\bar{\xi}$. 
    Then \begin{align*} 
    &\|Y,fY+F';\bar{Y},\bar{f}\bar{Y}+\bar{F}'\|_{W,\bar{W};\alpha,\gamma;p} \lesssim_{M_p,\alpha,\gamma,\gamma',p,T} \|\xi - \bar{\xi}\|_p + \rho_{\alpha}(\mathbf{W},\bar{\mathbf{W}}) + \E \Big(\int_0^T |G_r-\Bar{G}_r|_\mathcal{L}^{2p} dr \Big)^\frac{1}{2p}+ \\
    & \quad + \E \Big(\int_0^T |S_r-\Bar{S}_r|_\mathcal{L}^{2p} dr \Big)^\frac{1}{2p} +  \llbracket (f-\bar{f},f'-\bar{f}') \rrbracket_{\mathcal{L};2p} + \llbracket f,f'; \bar{f},\bar{f}' \rrbracket_{W,\bar{W};\gamma,\gamma';2p} + \|F,F';\bar{F},\bar{F}'\|_{W,\bar{W};\alpha,\gamma;p}.
\end{align*}
 where $M_p$ is any positive constant such that \begin{equation*} 
    \begin{aligned}
     &\|G\|_\infty + \|S\|_\infty + \llbracket(f,f')\rrbracket_{\mathcal{L};\infty} + \llbracket(f,f')\rrbracket_{W;\gamma,\gamma';2p,\infty} + \|(F,F')\|_{\mathbf{D}^{\alpha,\gamma}_W L_{2p}}  +\|\bar{G}\|_\infty + \|\bar{S}\|_\infty +  \llbracket(\bar{f},\bar{f}')\rrbracket_{\mathcal{L};\infty} + \\
     & \quad + \llbracket(\bar{f},\bar{f}')\rrbracket_{W;\gamma,\gamma';2p,\infty} + \|(\bar{F},\bar{F}')\|_{\mathbf{D}^{\alpha,\gamma}_W L_{2p}} + |\mathbf{W}|_\alpha + |\bar{\mathbf{W}}|_\alpha + \|\xi\|_{2p}  + \|\bar{\xi}\|_{2p} \le M_p.
    \end{aligned}
\end{equation*} 
\end{thm} 

\begin{proof}
    By slightly rearranging the proof of \cite[Theorem 3.13]{BCN24}, as in \cref{prop:continuityanddifferentiability_forcingterm}.
\end{proof}

We recall a classical interpolation result for H\"older spaces. This kind of statements can be applied to stochastic controlled rough paths and to stochastic controlled vector fields, possibly linear. For definiteness, we explicitly state the result in the former case in \cref{rmk:interpolation}; analogous statements can be deduced for the latter.  

\begin{prop} \label{prop:interpolation}
    \textnormal{(Interpolation)}
    Let $Z:[0,T] \to V$ be an $\alpha$-H\"older continuous path, for some $\alpha \in (0,1]$, and let us denote by $[Z]_\alpha$ its H\"older seminorm. 
    Then, for any $\gamma \in (0,\alpha)$, \begin{equation*}
        [Z]_\gamma \le [Z]_\alpha^\frac{\gamma}{\alpha} \left(\sup_{0 \le s \le t \le T} |\delta Z_{s,t}|\right)^{1-\frac{\gamma}{\alpha}}.
    \end{equation*}
    In particular, if $(Z^n)_n$ is a sequence of $\alpha$-H\"older continuous paths such that $\sup_{n} [Z^n]_\alpha < +\infty$ and $Z^n \to Z^\infty$ uniformly on $[0,T]$, then, for any $\gamma \in (0,\alpha)$, $Z^\infty$ is $\gamma$-H\"older continuous and $[Z^n - Z^\infty]_\gamma \to 0$ as $n \to +\infty$.
\end{prop}

\begin{coroll} \label{rmk:interpolation}
    Let $(W^n)_{n \ge 0}$ be a sequence of $\alpha$-H\"older continuous paths, for some $\alpha \in (0,1]$. For any $n \ge 0$, let $(X^n,(X^n)') \in \mathbf{D}_{W^n}^{\gamma,\gamma'} L_{p,q}$ and assume $\|X^n, (X^n)'; X^0, (X^0)'\|_{W^n,W^0;\gamma,\gamma';p,q} \to 0$ as $n \to +\infty, n \ge 1$. 
    Assume there exist $\lambda \in (\gamma,1]$ and $\lambda' \in (\gamma',1]$ such that $$\sup_{n\ge1} \|(X^n,(X^n)')\|_{\mathbf{D}_{W^n}^{\lambda,\lambda'}L_{p,q}} < +\infty. $$
    Then $(X^0,(X^0)') \in \mathbf{D}_{W^0}^{\eta,\eta'} L_{p,q}$ and $\|X^n, (X^n)'; X^0, (X^0)'\|_{W^n,W^0;\eta,\eta';p,q} \to 0$ as $n \to +\infty, n \ge 1$, for any $\eta \in (\gamma,\lambda)$, $\eta' \in (\gamma',\lambda')$. 
\end{coroll}

\section{Parameter-dependent rough SDEs} \label{section:parameterdependence}

Let $T>0$ and let $(\Omega,\mathcal{F},\{\mathcal{F}_t\}_{t \in [0,T]},\mathbb{P})$ be a complete filtered probability space and let $d \in \mathbb{N}_{\ge 1}$.
In \cite{KRYLOV} one can find the following notions of $\mathscr{L}$-\textit{continuity} and $\mathscr{L}$-\textit{differentiability} for a family of stochastic processes $\{X^\zeta,\ \zeta \in U\}$ taking values in $\R^d$ and indexed by a parameter.

\begin{defn} \label{def:stochasticprocessesinL}
    (cf.\ \cite[Definition 1.7.1]{KRYLOV})
     We denote by $\mathscr{L}(\R^d)$ the set of $\R^d$-valued progressively measurable stochastic processes belonging to $L^p(\Omega\times [0,T];\R^d)$ for any $p \in [1,\infty)$. For definiteness, $X \in \mathscr{L}(\R^d)$ if and only if it is progressively measurable and  \begin{equation*}
         \E\left(\int_0^T |X_r|^p dr\right) < +\infty \qquad \text{for any $p \in [1,\infty)$}.
     \end{equation*}
     If clear from the context, we simply write $\mathscr{L}$ instead of $\mathscr{L}(\R^d)$.
    With a little abuse of notation, $\mathscr{L}(\R^d)$ can also denote the set of $\R^d$-valued random variables belonging to $L^p(\Omega;\R^d)$ for any $p \in [1,\infty)$.
\end{defn}

\begin{defn} \label{def:continuityofstochasticprocesses}
    (cf.\ \cite[Definition 1.7.2]{KRYLOV})
    Let $(\mathcal{U},\rho)$ be a metric space and let $U \subseteq \mathcal{U}$ be non-empty. 
    Let $\{X^\zeta\}_{\zeta \in U}$ be a family of stochastic processes in $\mathscr{L}(\R^d)$ and let $\zeta^0 \in U$ be an accumulation point.
    We say that $\{X^\zeta\}_{\zeta \in U}$ admits limit in $\mathscr{L}(\R^d)$ as $\zeta \to \zeta^0$ if there is $Z \in \mathscr{L}(\R^d)$ such that, for any $(\zeta^n)_{n\ge1} \subseteq U \setminus \{\zeta^0\}$ with $\rho(\zeta^n,\zeta^0) \to_{n \to +\infty} 0$ and for any $p \in [1,\infty)$, there exists \begin{equation*}
        \lim_{n\to +\infty} \E \left( \int_0^T |X^{\zeta^n}_r - Z_r|^p dr \right)^\frac{1}{p} =0.
    \end{equation*}
    In such a case, we write $Z = \mathscr{L}\text{-}\lim_{\zeta \to \zeta^0} X^\zeta$.
    We say that $\{X^\zeta\}_{\zeta \in U}$ is $\mathscr{L}$-continuous at $\zeta^0 \in U$ if $\zeta^0$ is an isolated point of $U$, or if it is an accumulation point and $X^{\zeta^0} = \mathscr{L}\text{-}\lim_{\zeta \to \zeta^0} X^\zeta$. 
    We say that $\{X^\zeta\}_{\zeta \in U}$ is $\mathscr{L}$-continuous if it is $\mathscr{L}$-continuous at every $\zeta^0 \in U$.
\end{defn}

\begin{defn} \label{def:differentiabilityofstochasticprocesses}
    (cf.\ \cite[Definition 1.7.2]{KRYLOV})
    Let $(\mathcal{U},|\cdot|)$ be a normed vector space, let $U \subseteq \mathcal{U}$ be a non-empty open subset and let $\{X^\zeta\}_{\zeta \in U} \subseteq \mathscr{L}(\R^d)$. 
    We say that $\{X^\zeta\}_{\zeta \in U}$ is $\mathscr{L}$-differentiable at $\zeta^0 \in U$ in the direction $l \in \mathcal{U}, |l|=1$, if \begin{equation*}
            \left\{ \frac{1}{\varepsilon} (X^{\zeta^0 + \varepsilon l} - X^{\zeta^0}) \right\}_{\varepsilon \in \R\setminus \{0\}}
        \end{equation*}
        admits limit in $\mathscr{L}$ as $\varepsilon \to 0$. We denote such a limit by $\mathscr{L}\text{-}\frac{\partial}{\partial l} X^{\zeta^0}$.
        We say that $\{X^\zeta\}_{\zeta \in U}$ is ($\mathscr{L}$-continuously) $\mathscr{L}$-differentiable if it is $\mathscr{L}$-differentiable at any $\zeta^0 \in U$ in every direction $l \in \mathcal{U}, |l|=1$ (and, for any $l$, $\{ \mathscr{L}\text{-}\frac{\partial}{\partial l} X^{\zeta}\}_{\zeta \in U}$ is $\mathscr{L}$-continuous).
        Recursively on $k \in \mathbb{N}_{\ge 1}$, we say that $\{X^\zeta\}_{\zeta \in U}$ is $(k+1)$-times ($\mathscr{L}$-continuously) $\mathscr{L}$-differentiable if, for any $l \in \mathcal{U}, |l|=1$, $\{\mathscr{L}\text{-}\frac{\partial}{\partial l}X^{\zeta}\}_{\zeta \in U}$ is $k$-times ($\mathscr{L}$-continuously) $\mathscr{L}$-differentiable.
\end{defn}

\begin{rmk}
    In practice, the limits in \cref{def:continuityofstochasticprocesses} and \cref{def:differentiabilityofstochasticprocesses} can be required to hold for any $p \in [p_0,\infty)$ for some $1\le p_0 < \infty$, rather than all $p \in [1,\infty)$. 
    Since the time horizon $T$ is finite, the previous definitions also apply to random variables (i.e.\ constant-in-time stochastic processes). 
\end{rmk}

The previous framework accommodates where $\{X^\zeta\}_{\zeta \in U}$ is the family of solutions to a parameter-dependent SDE of the form $dX^\zeta_t = b_t(\zeta,X^\zeta_t) dt + \sigma_t(\zeta,X^\zeta_t) dB_t$. 
Indeed, every $X^\zeta$ can be seen as an element of $L^p(\Omega \times [0,T];\R^d)$. 
See \cite[Theorem 1.8.1]{KRYLOV} and \cite[Theorem 1.8.4]{KRYLOV} for relevant results. 
To convince the reader that the previous are good notions to understand continuity or the differentiability in mean for a family of stochastic processes, we immediately present the following result. 

\begin{prop} \label{prop:convergenceundermean}
    \begin{enumerate}
        \item Let $(\mathcal{U},\rho)$ be a metric space and let $U \subseteq \mathcal{U}$ be non-empty. Let $\{X^\zeta\}_{\zeta \in U}$ be a $\mathscr{L}(\R)$-continuous family of random variables and let $\phi:U \to \R$ be the map defined by $\phi(\zeta) := \E( X^\zeta)$.
        Then $\phi$ is continuous.
        \item Let $k \in \mathbb{N}_{\ge 1}$ and let $\{X^\zeta\}_{\zeta \in \R^d}$ be a family of $k$-times $\mathscr{L}(\R)$-differentiable random variables.
        Let us define $\phi(\zeta) := \E(X^\zeta)$ for $\zeta \in \R^d$.
        Then, for any $i_1,\dots,i_k \in \{1,\dots,d\}$, there exists $$\frac{\partial^k \phi}{\partial x^{i_k} \dots \partial x^{i_1}} (\zeta) = \E \left(\mathscr{L}\text{-}\frac{\partial}{\partial e_{i_k}} \dots \mathscr{L}\text{-}\frac{\partial}{\partial e_{i_1}} X^\zeta\right).$$ 
    \end{enumerate}
\end{prop}

\begin{proof}
        \textit{1. } Let $\zeta^0 \in U$ be an accumulation point, and let $(\zeta^n)_{n \ge 1} \subseteq U$ be an arbitrary sequence such that $\rho(\zeta^n,\zeta^0) \to 0$ as $n \to +\infty$. 
        Then, for any $n \ge 1$, \begin{equation*} \begin{aligned}
             |\phi(\zeta^n) - \phi(\zeta^0)| &= |\E(X^{\zeta^n}) - \E(X^{\zeta^0})| \le \E(|X^{\zeta^n}-X^{\zeta^0}|) 
        \end{aligned}
        \end{equation*}
        and the right-hand side tends to 0 as $n \to +\infty$, by assumption. \\
        
        \textit{2. } We argue recursively on $k$, and we first consider the case $k=1$.
        Let $i \in \{1,\dots,d\}$. For any $\zeta \in \R^d$ and for any $\varepsilon \in \R \setminus \{0\}$ we can write \begin{equation*}
            \frac{\phi(\zeta+\varepsilon e_i) - \phi(\zeta)}{\varepsilon} = \E \left( \frac{1}{\varepsilon} (X^{\zeta + \varepsilon e_i} - X^\zeta) \right),
        \end{equation*}
        and it tends to $\E(\mathscr{L}\text{-}\frac{\partial}{\partial e_i} X^\zeta)$ as $\varepsilon \to 0$, by assumption. \\
        Assume now the conclusion holds for $k=k_0$ and that the assumptions are satisfied for $k=k_0+1$. Let $i_1,\dots,i_{k_0},i_{k_0+1} \in \{1,\dots,d\}$.
        For any $\zeta \in \R^d$ and for any $\varepsilon \in \R \setminus \{0\}$ we have \begin{equation*}
            \begin{aligned}
                &\frac{1}{\varepsilon} \left( \frac{\partial^{k_0} \phi}{\partial x^{k_0} \dots \partial x^1} (\zeta + \varepsilon e_{i_{k_0+1}}) - \frac{\partial^{k_0} \phi}{\partial x^{k_0} \dots \partial x^1} (\zeta) \right) = \\
                &= \E \left( \frac{1}{\varepsilon} \left( \mathscr{L}\text{-}\frac{\partial}{\partial e_{i_{k_0}}} \dots \mathscr{L}\text{-}\frac{\partial}{\partial e_{i_1}} X^{\zeta+\varepsilon e_{k_0+1}} - \mathscr{L}\text{-}\frac{\partial}{\partial e_{i_{k_0}}} \dots \mathscr{L}\text{-}\frac{\partial}{\partial e_{i_1}} X^\zeta \right)  \right)
            \end{aligned}
        \end{equation*}
        and it tends to $\E \left(\mathscr{L}\text{-}\frac{\partial}{\partial e_{i_{k_0+1}}}\left(\mathscr{L}\text{-}\frac{\partial}{\partial e_{i_{k_0}}} \dots \mathscr{L}\text{-}\frac{\partial}{\partial e_{i_1}} X^\zeta \right)\right)$ as $\varepsilon \to 0$, by assumption. 
\end{proof}

As explained in \cite{FHL21}, a solution $X^\zeta$ to a parameter-dependent rough SDE of the form $dX^\zeta_t = b_t(\zeta,X^\zeta_t) dt + \sigma_t(\zeta,X^\zeta_t) dB_t + \beta_t(\zeta,X^\zeta_t) d\mathbf{W}_t$ must be considered in a subset of $L^p(\Omega \times [0,T];\R^d)$, namely the space of stochastic controlled rough paths introduced in \cref{def:stochasticcontrolledroughpaths}. 
The norm in \cref{def:stochasticprocessesinL} is too weak to obtain meaningful stability estimates (see instead \cref{thm:stabilityforRSDEs_compendium}).
Moreover, the solution space itself depends on the rough path driving the equation. 
For this reason, continuity with respect to $\zeta$ is formulated below allowing $\mathbf{W}$ to vary with the parameter, while differentiability requires $\mathbf{W}$ to be fixed so that linear combinations of solutions remain well defined. 
In the following paragraphs we refine the notions of $\mathscr{L}$-continuity and $\mathscr{L}$-differentiability to extend Krylov's theory to rough SDEs. 

\subsection{Continuity for rough SDEs} \label{section:continuityforRSDEs}

Let $(\mathcal{U},\rho)$ be a metric space and let $U \subseteq \mathcal{U}$ be non-empty.

\begin{defn} \label{def:L-stochasticcontrolledroughpaths}
    Let $W:[0,T]\to \R^n$ be $\alpha$-H\"older continuous, with $\alpha \in (0,1]$, and let $\gamma,\gamma' \in [0,1]$.
    We denote by $\mathbf{D}_W^{\gamma,\gamma'}\mathscr{L}(\R^d)$ the set of $\R^d$-valued stochastic controlled rough paths belonging to $\mathbf{D}_W^{\gamma,\gamma'}L_p(\R^d)$ for any $p \in [2,\infty)$ (see \cref{def:stochasticcontrolledroughpaths}). 
    If clear from the context, we simply denote $\mathbf{D}_W^{\gamma,\gamma'} \mathscr{L}(\R^d)$ by $\mathbf{D}_W^{\gamma,\gamma'} \mathscr{L}$.  
\end{defn}

\begin{defn} \label{def:convergenceandcontinuityofstochastiscontrolledroughpaths}
    Let $\alpha,\gamma,\gamma' \in [0,1]$, with $\alpha \ne 0$.
    Let $\{W^\zeta\}_{\zeta \in \mathcal{U}}$ be a family of $\alpha$-H\"older continuous paths depending on a parameter and, for any $\zeta \in U$, let $(X^\zeta,(X^\zeta)') \in \mathbf{D}_{W^{\zeta}}^{\gamma,\gamma'}\mathscr{L}(\R^d)$. 
    Let $\zeta^0 \in U$ be an accumulation point.
    We say that $\{(X^\zeta,(X^\zeta)')\}_{\zeta \in U}$ admits limit in $\{\mathbf{D}_{W^\zeta}^{\gamma,\gamma'} \mathscr{L}\}_{\zeta \in U}$ as $\zeta \to \zeta^0$ if there is $(Z,Z') \in \mathbf{D}_{W^{\zeta^0}}^{\gamma,\gamma'}\mathscr{L}$ such that, for any $(\zeta^n)_{n\ge1} \subseteq U \setminus \{\zeta^0\}$ converging to $\zeta^0$ in $\mathcal{U}$ and for any $p \in [2,\infty)$, there exists \begin{equation*}
            \lim_{n \to +\infty} \|X^{\zeta^n},(X^{\zeta^n})';Z,Z'\|_{W^{\zeta^n},W^{\zeta^0};\gamma,\gamma';p} =0.
        \end{equation*}
    In such a case we write $$(Z,Z') = \{\mathbf{D}_{W^\zeta}^{\gamma,\gamma'}\mathscr{L}\}_\zeta  \ \text{-}\lim_{\zeta \to \zeta^0} (X^{\zeta},(X^{\zeta})').$$
    We say that $\{(X^\zeta,(X^\zeta)')\}_{\zeta \in U}$ is $\{\mathbf{D}_{W^\zeta}^{\gamma,\gamma'} \mathscr{L}\}_{\zeta}$-continuous at $\zeta^0 \in U$ if $\zeta^0$ is an isolated point of $U$, or if it is an accumulation point and $(X^{\zeta^0},(X^{\zeta^0})') = \{\mathbf{D}_{W^\zeta}^{\gamma,\gamma'}\mathscr{L}\}_\zeta \ \text{-}\lim_{\zeta \to \zeta^0} (X^{\zeta},(X^{\zeta})')$. 
    We say that $\{(X^\zeta,(X^\zeta)')\}_{\zeta \in U}$ is $
        \{\mathbf{D}_{W^\zeta}^{\gamma,\gamma'} \mathscr{L}\}_{\zeta}$-continuous if it is $\{\mathbf{D}_{W^\zeta}^{\gamma,\gamma'} \mathscr{L}\}_{\zeta}$-continuous at every $\zeta^0 \in U$. 
    If $W^\zeta =W$ for any $\zeta \in U$, we simply write $\mathbf{D}_{W}^{\gamma,\gamma'} \mathscr{L}$ instead of $\{\mathbf{D}_{W^\zeta}^{\gamma,\gamma'} \mathscr{L}\}_{\zeta}$.
\end{defn}

Let $(V,|\cdot|_V)$ an auxiliary Banach space, typically $V=\R^{d\times d_2}$. If clear from the context, $|\cdot|_V$ is simply denoted by $|\cdot|$.

\begin{defn} \label{def:convergenceforstochasticboundedcontinuousvectorfields}
    For any $\zeta \in U$, let $g^\zeta$ be a random bounded vector field from $\R^d$ to $V$ in the sense of \cref{def:stochasticboundedandLipschitzvectorfield}. 
    Let $\zeta^0 \in U$ be an accumulation point.
    We say that $\{g^\zeta\}_{\zeta \in U}$ admits limit in $\mathscr{L}C^0_b(\R^d;V)$ as $\zeta \to \zeta^0$ if there is a random bounded vector field $g$ from $\R^d$ to $V$ such that, for any $(\zeta^n)_{n\ge1} \subseteq U \setminus \{\zeta^0\}$ with $\rho(\zeta^n,\zeta^0) \to_{n \to +\infty} 0$ and for any $p \in [1,\infty)$, there exists \begin{equation*}
     \lim_{n \to +\infty} \sup_{t \in [0,T]} \||g^{\zeta^n}_t(\cdot) - g_t(\cdot)|_\infty\|_p =0. 
     \end{equation*}
    In such a case we write  $g = \mathscr{L}C^0_b(\R^d;V) \text{-} \lim_{\zeta \to \zeta^0} g^\zeta$.
    We say that $\{g^\zeta\}_{\zeta \in U}$ is $\mathscr{L}C^0_b$-continuous at $\zeta^0$ if $\zeta^0$ is an isolated point of $U$, or if it is an accumulation point and  $g^{\zeta^0}=\mathscr{L}C^0_b\text{-}\lim_{\zeta \to \zeta^0} g^{\zeta}$.
    We say that $\{g^\zeta\}_{\zeta \in U}$ is $\mathscr{L}C^0_b$-continuous if it is $\mathscr{L}C^0_b$-continuous at every $\zeta^0 \in U$.
    If clear from the context we simply write $\mathscr{L}C^0_b$ instead of $\mathscr{L}C^0_b(\R^d;V)$.
\end{defn}

\begin{defn}  \label{def:convergenceforstochasticcontrolledvetorfields} 
Let $\alpha, \delta \in [0,1]$, with $\alpha \ne 0$, and let $\kappa >1$. 

    \begin{itemize}
        \item Given an $\alpha$-H\"older continuous path $W:[0,T]\to \R^n$, we denote by $\mathbf{D}_W^{2\delta}\mathscr{L}C^\kappa_b(\R^d;V)$ the set of stochastic controlled vector fields from $\R^d$ to $V$ belonging to $\mathbf{D}_W^{2\delta}L_p C^\kappa_b(\R^d;V)$ for any $p \in [2,\infty)$ (see \cref{def:stochasticcontrolledvectorfield}). 
        If clear from the context, we simply write $\mathbf{D}_W^{2\delta}\mathscr{L} C^\kappa_b$ instead of $\mathbf{D}_W^{2\delta}\mathscr{L} C^\kappa_b(\R^d;V)$. 
        \item Let $\{W^\zeta\}_{\zeta \in \mathcal{U}} \subseteq {C}^{\alpha}([0,T];\R^n)$ be a family of $\alpha$-H\"older continuous paths depending on a parameter.
        For any $\zeta \in U$, consider $(\beta^\zeta,(\beta^\zeta)') \in \mathbf{D}_{W^{\zeta}}^{2\delta} \mathscr{L} C_b^\kappa(\R^d;V)$ and let $\zeta^0 \in U$ be an accumulation point.
        We say that $\{(\beta^\zeta,(\beta^\zeta)')\}_{\zeta \in U}$ admits limit in $\{\mathbf{D}_{W^\zeta}^{2\delta}\mathscr{L}C_b^\kappa\}_\zeta$ as $\zeta \to \zeta^0$ if there is $(\beta,\beta') \in \mathbf{D}_{W^{\zeta^0}}^{2\delta}\mathscr{L}C^\kappa_b$ such that, for any $(\zeta^n)_{n\ge1} \subseteq U \setminus \{\zeta^0\}$ with $\rho(\zeta^n,\zeta^0) \to_{n \to +\infty} 0$ and for any $p \in [2,\infty)$, there exists \begin{equation*}
        \lim_{n \to +\infty} \llbracket (\beta^{\zeta^n}-\beta, (\beta^{\zeta^n})' - \beta') \rrbracket_{\kappa;p} +  \llbracket \beta^{\zeta^n},(\beta^{\zeta^n})';\beta,\beta' \rrbracket_{W^{\zeta^n},W^{\zeta^0};2\delta;p;\kappa} =0.
    \end{equation*}
    In such a case we write $$(\beta,\beta') = \{\mathbf{D}_{W^\zeta}^{2\delta}\mathscr{L}C_b^\kappa\}_\zeta \text{-} \lim_{\zeta \to \zeta^0} (\beta^\zeta,(\beta^\zeta)').$$
    We say that $\{(\beta^\zeta,(\beta^\zeta)')\}_{\zeta \in U}$ is $\{\mathbf{D}_{W^\zeta}^{2\delta}\mathscr{L}C_b^\kappa\}_\zeta$-continuous at $\zeta^0 \in U$ if $\zeta^0$ is an isolated point of $U$, or if it is an accumulation point and $(\beta^{\zeta^0},(\beta^{\zeta^0})') = \{\mathbf{D}_{W^\zeta}^{2\delta}\mathscr{L}C_b^\kappa\}_\zeta \text{-} \lim_{\zeta \to \zeta^0} (\beta^\zeta,(\beta^\zeta)')$.
    We say that $\{(\beta^\zeta,(\beta^\zeta)')\}_{\zeta \in U}$ is $\{\mathbf{D}_{W^\zeta}^{2\delta}\mathscr{L}C_b^\kappa\}_\zeta$-continuous if it is $\{\mathbf{D}_{W^\zeta}^{2\delta}\mathscr{L}C_b^\kappa\}_\zeta$-continuous at any $\zeta^0 \in U$.
    If $W^\zeta = W$ for any $\zeta \in U$, we simply write $\mathbf{D}_{W}^{2\delta}\mathscr{L}C_b^\kappa$ instead of $\{\mathbf{D}_{W^\zeta}^{2\delta}\mathscr{L}C_b^\kappa\}_\zeta$.
    \end{itemize}
\end{defn}

\begin{defn} \label{def:convergenceforroughpaths}
    Let $\alpha \in (\frac{1}{3},\frac{1}{2}]$ and let $\{\mathbf{W}^\zeta\}_{\zeta \in U} \subseteq \mathscr{C}^{\alpha}([0,T];\R^n)$ be a family of $\alpha$-rough paths depending on a parameter.
    Let $\zeta^0 \in U$ be an accumulation point. 
    We say that $\{\mathbf{W}^\zeta\}_{\zeta \in U}$ admits limit in $\mathscr{C}^{\alpha}([0,T];\R^n)$ if there is $\mathbf{W} \in \mathscr{C}^{\alpha}([0,T];\R^n)$ such that, for any arbitrary sequence $(\zeta^n)_{n\ge1} \subseteq U \setminus \{\zeta^0\}$ converging to $\zeta^0$, there exists \begin{equation*}
        \lim_{n \to +\infty} \rho_{\alpha} (\mathbf{W}^{\zeta^n},\mathbf{W}) =0. 
    \end{equation*}
    In such a case we write $\mathbf{W}=\mathscr{C}^{\alpha}\text{-}\lim_{\zeta \to \zeta^0} \mathbf{W}^\zeta$. 
    We say that $\{\mathbf{W}^\zeta\}_{\zeta \in U}$ is $\mathscr{C}^\alpha$-continuous at $\zeta^0 \in U$ if $\zeta^0$ is an isolated point of $U$, or if it is an accumulation point and $\mathbf{W}^{\zeta^0}=\mathscr{C}^{\alpha}\text{-}\lim_{\zeta \to \zeta^0} \mathbf{W}^\zeta$.
    We say that $\{\mathbf{W}^\zeta\}_{\zeta \in U}$ is $\mathscr{C}^\alpha$-continuous if it is $\mathscr{C}^\alpha$-continuous at every $\zeta^0 \in U$.    
\end{defn}


\begin{thm} \label{thm:continuity_RSDEs}
    Let $\{\mathbf{W}^\zeta\}_{\zeta \in {U}}$ be a $\mathscr{C}^{\alpha}$-continuous family of $\alpha$-rough paths in $\R^n$, with $\alpha \in (\frac{1}{3},\frac{1}{2}]$, and let $\kappa \in (2,3]$ and $\delta \in [0,\alpha]$ such that $\alpha + (\kappa-1)\delta >1$.
    Let $\{\xi^{\zeta}\}_{\zeta \in U}$ be a $\mathscr{L}(\R^d)$-continuous family of $\mathcal{F}_0$-measurable random variables. 
    Let $\{b^\zeta\}_{\zeta \in U}$ and $\{\sigma^\zeta\}_{\zeta \in U}$ be two families of random bounded Lipschitz vector fields which are $\mathscr{L}C^0_b(\R^d;\R^d)$-continuous and $\mathscr{L}C^0_b (\R^d;\R^{d \times m})$-continuous, respectively.
    Let $\{(\beta^\zeta,(\beta^\zeta)')\}_{\zeta \in U}$ be a $\{\mathbf{D}_{W^\zeta}^{2\delta}\mathscr{L}C^{\kappa-1}_b\}_\zeta$-continuous family of stochastic controlled vector fields such that, for any $\zeta \in U$, \begin{equation*}
        (\beta^\zeta, (\beta^\zeta)') \in \bigcap_{p \in [2,\infty)} \mathbf{D}_{W^\zeta}^{2\delta}L_{p,\infty} C_b^\kappa(\R^d;\R^{d\times n}).
    \end{equation*} 
    For any $\zeta \in U$, let $X^\zeta$ be an $L_{p,\infty}$-integrable solution of the following rough SDE, for any $p \in [2,\infty)$: \begin{equation*}
        X^\zeta_t = \xi^\zeta + \int_0^t b_r^\zeta(X^\zeta_r)dr + \int_0^t \sigma_r(X^\zeta_r)dB_r + \int_0^t (\beta^\zeta_r(X^\zeta_r), D_x\beta^\zeta_r(X^\zeta_r)\beta^\zeta_r(X^\zeta_r) + (\beta^\zeta)'_r(X^\zeta_r)) d\mathbf{W}^\zeta_r \quad t \in [0,T].
    \end{equation*} 
    Then $\{X^\zeta,\beta^\zeta(X^\zeta)\}_{\zeta \in U}$ is $\{\mathbf{D}_{W^\zeta}^{\alpha,\delta}\mathscr{L}(\R^d)\}_\zeta$-continuous. 
\end{thm}

\begin{proof}
    Let $\zeta^0 \in U$ be an accumulation point and let $(\zeta^n)_{n \ge 1} \subseteq U \setminus \{\zeta^0\}$ be an arbitrary sequence such that $\rho(\zeta^n,\zeta^0) \to 0$ as $n \to +\infty$.
    Let $p \in [2,\infty)$ and let $M>0$ be any constant such that \begin{equation*}
        \begin{aligned}
            |\mathbf{W}^{\zeta^0}|_\alpha + \sup_{n\ge 1} |\mathbf{W}^{\zeta^n}|_\alpha +\|\beta^{\zeta^0}\|_{Lip} + \|\sigma^{\zeta^0}\|_{Lip} + \llbracket (\beta^{\zeta^0}, (\beta^{\zeta^0})') \rrbracket_{\kappa;\infty} + \llbracket (\beta^{\zeta^0}, (\beta^{\zeta^0})') \rrbracket_{W^{\zeta^0};2\delta;p,\infty;\kappa} \le M.
        \end{aligned}
    \end{equation*}
    From \cref{thm:stabilityforRSDEs_compendium} we conclude that, for any $n \ge 1$, \begin{equation*}
        \begin{aligned}
            &\|X^{\zeta^n},\beta^{\zeta^n}(X^{\zeta^n});X^{\zeta^0},\beta^{\zeta^0}(X^{\zeta^0})\|_{W^{\zeta^n},W^{\zeta^0};\alpha,\delta;p} \\
            &\quad \lesssim_{M,\alpha,\delta,\kappa,p,T} \|X^{\zeta^n}_0 - X^{\zeta^0}_0\|_p + \rho_{\alpha}(\mathbf{W}^{\zeta^n},\mathbf{W}^{\zeta^0}) + \sup_{t \in [0,T]} \| |b^{\zeta^n}_t-b^{\zeta^0}_t|_\infty\|_p + \sup_{t \in [0,T]} \| |\sigma^{\zeta^n}_t-\sigma^{\zeta^0}_t|_\infty\|_p  + \\
            & \quad \quad + \llbracket (\beta^{\zeta^n}-\beta^{\zeta^0},(\beta^{\zeta^n})'-(\beta^{\zeta^0})') \rrbracket_{\kappa-1;p} + \llbracket \beta^{\zeta^n},(\beta^{\zeta^n})'; \beta^{\zeta^0},(\beta^{\zeta^0})' \rrbracket_{W^{\zeta^n},W^{\zeta^0};2\alpha;p;\kappa-1}.
        \end{aligned}
    \end{equation*}
    By assumption, the right-hand side is going to zero as $n \to +\infty$, and the assertion is therefore proved.
\end{proof}


\subsection{Continuity for linear rough SDEs} \label{section:continuityforlinearRSDEs}

Let $(\mathcal{U},\rho)$ be a metric space and let $U \subseteq \mathcal{U}$ be non-empty.

\begin{defn}  \label{def:continuityforstochasticlinearvectorfields}
    Let $\{Q^\zeta\}_{\zeta \in U}$ be a parameter-dependent family of stochastic linear vector fields from $\R^d$ to $V$, in the sense of \cref{def:stochasticlinearvectorfields}. 
    Let $\zeta^0 \in U$ be an accumulation point.
    We say that $\{Q^\zeta\}_{\zeta \in U}$ admits limit in  $\mathscr{L}\mathcal{L}(\R^d;V)$ as $\zeta \to \zeta^0$ if there is a stochastic linear vector field $Q$ from $\R^d$ to $V$ such that, for any $(\zeta^n)_{n\ge1} \subseteq U \setminus \{\zeta^0\}$ with $\rho(\zeta^n,\zeta^0) \to_{n \to +\infty} 0$ and for any $p \in [1,\infty)$, there exists \begin{equation*}
       \lim_{n \to +\infty} \E \left(\int_0^T |Q^{\zeta^n}_r-Q_r|_{\mathcal{L}(\R^d;V)}^p dr \right) =0.
  \end{equation*}
  In such a case we write $$Q = \mathscr{L}\mathcal{L}(\R^d;V)\text{-}\lim_{\zeta \to \zeta^0} Q^\zeta. $$
  If clear from the context, we simply write $\mathscr{L}\mathcal{L}$ instead of $\mathscr{L}\mathcal{L}(\R^d;V)$.
  We say that $\{Q^\zeta\}_{\zeta \in U}$ is $\mathscr{L}\mathcal{L}$-continuous at $\zeta^0 \in U$ if $\zeta^0$ is an isolated point of $U$, or if it is an accumulation point and $Q^{\zeta^0}=\mathscr{L}\mathcal{L}\text{-}\lim_{\zeta \to \zeta^0} Q^\zeta$. 
  We say that $\{Q^\zeta\}_{\zeta \in U}$ is $\mathscr{L}\mathcal{L}$-continuous if it is $\mathscr{L}\mathcal{L}$-continuous at every $\zeta^0 \in U$. 
\end{defn}

\begin{defn} \label{def:convergenceandcontinuitystochasticcontrolledlinearvectorfields}
    Let $\alpha,\delta,\delta' \in [0,1]$, with $\alpha \ne 0$. \begin{itemize}
        \item Given any $\alpha$-H\"older continuous path $W:[0,T] \to \R^n$, we denote by $\mathbf{D}_W^{\delta,\delta'}\mathscr{L}\mathcal{L}(\R^d;V)$ the set of stochastic controlled linear vector fields belonging to $\mathbf{D}_W^{\delta,\delta'}L_p\mathcal{L}(\R^d;V)$ for any $p \in [2,\infty)$. 
        If clear from the context, we simply write $\mathbf{D}_W^{\delta,\delta'}\mathscr{L}\mathcal{L}$ instead of $\mathbf{D}_W^{\delta,\delta'}\mathscr{L}\mathcal{L}(\R^d;V)$.
        \item  Let $\{W^\zeta\}_{\zeta \in \mathcal{U}} \subseteq C^\alpha([0,T];\R^n)$ and assume that, for every $\zeta \in U$, there is $(f^\zeta,(f^\zeta)') \in \mathbf{D}_{W^\zeta}^{\delta,\delta'}\mathscr{L}\mathcal{L}$.
        Let $\zeta^0 \in U$ be an accumulation point.
        We say that $\{(f^\zeta,(f^\zeta)')\}_{\zeta \in U}$ admits limit in $\{\mathbf{D}_{W^\zeta}^{\delta,\delta'}\mathscr{L}\mathcal{L}\}_\zeta$ as $\zeta \to \zeta^0$ if there is $(f,f') \in \mathbf{D}_{W^{\zeta^0}}^{\delta,\delta'}\mathscr{L}\mathcal{L}$ such that, for any $(\zeta^n)_{n\ge1} \subseteq U \setminus \{\zeta^0\}$ with $\rho(\zeta^n,\zeta^0) \to 0$ and for any $p \in [2,\infty)$, there exists
        \begin{equation*}
        \lim_{n \to +\infty} \llbracket (f^{\zeta^n}-f^{\zeta^0}, (f^{\zeta^n})'-(f^{\zeta^0})') \rrbracket_{\mathcal{L};p} + \llbracket f^{\zeta^n},(f^{\zeta^n})';f, f' \rrbracket_{W^{\zeta^n},W^{\zeta^0};\delta,\delta';p} =0 \qquad \text{for any $p \in [2,\infty)$}. \end{equation*}
    In such a case we write $$(f,f') = \{\mathbf{D}_{W^\zeta}^{\delta,\delta'}\mathscr{L}\mathcal{L}\}_\zeta \text{-} \lim_{\zeta \to \zeta^0} (f^\zeta,(f^\zeta)'). $$ 
    We say that $\{(f^\zeta,(f^\zeta)')\}_{\zeta \in U}$ is $\{\mathbf{D}_{W^\zeta}^{\delta,\delta'}\mathscr{L}\mathcal{L}\}_\zeta$-continuous at $\zeta^0 \in U$ if $\zeta^0$ is an isolated point of $U$, or if it is an accumulation point and $(f^{\zeta^0},(f^{\zeta^0})') = \{\mathbf{D}_{W^\zeta}^{\delta,\delta'}\mathscr{L}\mathcal{L}\}_\zeta \text{-} \lim_{\zeta \to \zeta^0} (f^\zeta,(f^\zeta)')$.
    We say that $\{(f^\zeta,(f^\zeta)')\}_{\zeta \in U}$ is $\{\mathbf{D}_{W^\zeta}^{\delta,\delta'}\mathscr{L}\mathcal{L}\}_\zeta$-continuous if it is $\{\mathbf{D}_{W^\zeta}^{\delta,\delta'}\mathscr{L}\mathcal{L}\}_\zeta$-continuous at every $\zeta^0 \in U$. 
    If $W^\zeta = W$ for every $\zeta \in U$, we simply write $\mathbf{D}_{W^\zeta}^{\delta,\delta'}\mathscr{L}\mathcal{L}$ instead of $\{\mathbf{D}_{W^\zeta}^{\delta,\delta'}\mathscr{L}\mathcal{L}\}_\zeta$.
    \end{itemize}
\end{defn}


\begin{thm} \label{prop:continuity_linearRSDEs} 
    Let $\{\mathbf{W}^\zeta\}_{\zeta \in U}$ be a $\mathscr{C}^{\alpha}$-continuous family of rough paths, with $\alpha \in (\frac{1}{3},\frac{1}{2})$, and let $\delta \in [0,\alpha]$ and $\delta' \in [0,1]$ with $\alpha + \delta > \frac{1}{2}$ and $\alpha + \delta + \min\{\delta,\delta'\} >1$. 
    Let $\{\xi^{\zeta}\}_{\zeta \in U}$ be an $\mathscr{L}(V)$-continuous family of $\mathcal{F}_0$-measurable random variables.
    Let $\{G^\zeta\}_{\zeta \in U}$ be a $\mathscr{L}\mathcal{L}(V;V)$-continuous family of stochastic linear vector fields and let $\{S^\zeta\}_{\zeta \in U}$ be $\mathscr{L}\mathcal{L}(V;\mathcal{L}(\R^m;V))$-continuous. 
    Let  $\{f^\zeta,(f^\zeta)'\}_{\zeta \in U}$ be $\{\mathbf{D}_{W^\zeta
    }^{\delta,\delta'}\mathscr{L}\mathcal{L}(V;\mathcal{L}(\R^n;V))\}_\zeta$-continuous and let $\{(F^\zeta,(F^\zeta)')\}_{\zeta \in U}$ be $\{\mathbf{D}_{W^\zeta}^{\alpha,\delta}\mathscr{L}(V)\}_\zeta$-continuous. 
    Moreover, assume that for any $R>0$ and for any $p \in [2,\infty)$,
    \begin{equation*}
      \sup_{|\zeta|\le R} \left(\llbracket G^\zeta \rrbracket_\infty + \llbracket S^\zeta \rrbracket_\infty + \llbracket (f^\zeta,(f^\zeta)') \rrbracket_{\mathcal{L};\infty} + \llbracket (f^\zeta,(f^\zeta)') \rrbracket_{W^\zeta;\delta,\delta';p,\infty;\mathcal{L}} \right) <+\infty.
    \end{equation*}
    For any $\zeta \in U$, let $Y^\zeta$ be a $L_{p}$-integrable solution of the following linear RSDE, for any $p \in [2,\infty)$: \begin{equation*}
        Y_t^{\zeta} = \xi^{\zeta} + F_t^{\zeta} + \int_0^t G_r^\zeta Y_r^{\zeta}d r + \int_0^t S_r^\zeta Y_r^{\zeta} d B_r + \int_0^t (f^\zeta_r,(f^\zeta)'_r) Y^\zeta_r d\mathbf{W}^\zeta_r \qquad t \in [0,T].
    \end{equation*}
    where $(f^\zeta_t,(f^\zeta)'_t) Y^\zeta_t = (f_t^\zeta Y_t^\zeta,(f^\zeta)'_t Y_t^\zeta + f_t^\zeta (f_t^\zeta Y_t^\zeta + (F^\zeta)'_t ))$. 
    Then $\{(Y^{\zeta},f^\zeta Y^\zeta + (F^\zeta)')\}_{\zeta \in U}$ is $\{\mathbf{D}_{W^\zeta}^{\alpha,\delta}\mathscr{L}(V)\}_\zeta$-continuous. 
\end{thm}

\begin{proof}
The assertion follows from the local Lipschitz estimates for linear rough SDEs in \cref{thm:stabilityforlinearRSDEs_compendium}, with a similar argument as in the proof of \cref{thm:continuity_RSDEs}. 
\end{proof}

\subsection{Differentiability for linear rough SDEs} \label{section:differentiabilityforlinearRSDEs}

Let $(\mathcal{U},|\cdot|)$ be a normed vector space and let $U \subseteq \mathcal{U}$ be a non-empty open subset.
We denote by $\mathbb{S}(\mathcal{U}) :=\{l \in \mathcal{U} \mid |l|=1 \}$ the set of all directions (along which we are interested to take derivatives).
Let $W:[0,T]\to \R^n$ be an $\alpha$-H\"older continuous path, with $\alpha \in (0,1]$, and let $(V,|\cdot|_V)$ be a finite-dimensional Banach space (think of $V=\R^{d_2}$ or $V=\R^{d \times d_2}$ for some $d_2 \in \mathbb{N}_{\ge 1}$).

\begin{defn} \label{def:differentiabilitystochasticcontrolledroughpaths}
    Let $\gamma,\gamma' \in [0,1]$.  
    Let $\{(X^\zeta,(X^\zeta)')\}_{\zeta \in U} \subseteq \mathbf{D}_W^{\gamma,\gamma'} \mathscr{L}$ be a family of stochastic controlled rough paths depending on a parameter. 
    We say that it is $\mathbf{D}_W^{\gamma,\gamma'} \mathscr{L}$-differentiable at $\zeta^0 \in U$ in the direction $l \in \mathbb{S}(\mathcal{U})$ if \begin{equation*}
            \left\{ \left(\frac{1}{\varepsilon}(X^{\zeta^0 + \varepsilon l}-X^{\zeta^0}),\frac{1}{\varepsilon}\big((X^{\zeta^0 + \varepsilon l})'-(X^{\zeta^0})'\big) \right) \right\}_{\varepsilon \in \R \setminus \{0\}}
    \end{equation*}
    admits limit in $\mathbf{D}_W^{\gamma,\gamma'} \mathscr{L}$ as $\varepsilon \to 0$, in the sense of \cref{def:convergenceandcontinuityofstochastiscontrolledroughpaths}.
    We denote such a limit as \begin{equation*}
            \left(\mathbf{D}_W^{\gamma,\gamma'} \mathscr{L} \text{-} \frac{\partial}{\partial l} X^{\zeta^0}, \mathbf{D}_W^{\gamma,\gamma'} \mathscr{L} \text{-} \frac{\partial}{\partial l} (X^{\zeta^0})' \right).
    \end{equation*}
    We say that $\{(X^\zeta,(X^\zeta)')\}_{\zeta \in U}$ is ($\mathbf{D}_W^{\gamma,\gamma'} \mathscr{L}$-continuously) $\mathbf{D}_W^{\gamma,\gamma'} \mathscr{L}$-differentiable if it is $\mathbf{D}_W^{\gamma,\gamma'} \mathscr{L}$-differentiable at any $\zeta^0 \in U$ in every direction $l \in \mathbb{S}(\mathcal{U})$ (and, for any $l \in \mathbb{S}(\mathcal{U})$, $\{(\mathbf{D}_W^{\gamma,\gamma'} \mathscr{L} \text{-} \frac{\partial}{\partial l} X^{\zeta}, \mathbf{D}_W^{\gamma,\gamma'} \mathscr{L} \text{-} \frac{\partial}{\partial l} (X^{\zeta})')\}_{\zeta \in U}$ is $\mathbf{D}_W^{\gamma,\gamma'} \mathscr{L}$-continuous).
    Recursively on $k \in \mathbb{N}_{\ge 1}$, we say that $\{(X^\zeta,(X^\zeta)')\}_{\zeta \in U}$ is $(k+1)$-times ($\mathbf{D}_W^{\gamma,\gamma'} \mathscr{L}$-continuously) $\mathbf{D}_W^{\gamma,\gamma'} \mathscr{L}$-differentiable if, for any $l \in \mathbb{S}(\mathcal{U})$, $\{(\mathbf{D}_W^{\gamma,\gamma'} \mathscr{L} \text{-} \frac{\partial}{\partial l} X^{\zeta}, \mathbf{D}_W^{\gamma,\gamma'} \mathscr{L} \text{-} \frac{\partial}{\partial l} (X^{\zeta})')\}_{\zeta \in U}$ is $k$-times ($\mathbf{D}_W^{\gamma,\gamma'} \mathscr{L}$-continuously) $\mathbf{D}_W^{\gamma,\gamma'} \mathscr{L}$-differentiable.
\end{defn}

\begin{defn} \label{def:differentiabilitystochasticlinearvectorfields}
  We say that the family of stochastic linear vector fields $\{Q^\zeta\}_{\zeta \in U}$ is $\mathscr{L}\mathcal{L}$-differentiable at $\zeta^0 \in U$ in the direction $l \in \mathbb{S}(\mathcal{U})$, if
  \begin{equation*}
      \left\{ \frac{1}{\varepsilon} \left(Q^{\zeta^0 + \varepsilon l} - Q^{\zeta^0} \right) \right\}_{\varepsilon \in \R \setminus \{0\}}
  \end{equation*}
  admits limit in $\mathscr{L}\mathcal{L}$ as $\varepsilon \to 0$, in the sense of \cref{def:continuityforstochasticlinearvectorfields}. 
  We denote such a limit as $\mathscr{L}\mathcal{L}\text{-}\frac{\partial}{\partial l}Q^{\zeta^0}$. We say that $\{Q^\zeta\}_{\zeta \in U}$ is $\mathscr{L}\mathcal{L}$-differentiable if it is $\mathscr{L}\mathcal{L}$-differentiable at every $\zeta^0 \in U$ and in every direction $l \in \mathbb{S}(\mathcal{U})$. 
  We say that it is $\mathscr{L}\mathcal{L}$-continuously $\mathscr{L}\mathcal{L}$-differentiable if, in addition, 
  $\{\mathscr{L}\mathcal{L}\text{-}\frac{\partial}{\partial l}Q^{\zeta}\}_{\zeta \in U}$ is $\mathscr{L}\mathcal{L}$-continuous for any $l \in \mathbb{S}(\mathcal{U})$.
  Recursively on $k \in \mathbb{N}_{\ge 1}$, we say that $\{Q^\zeta\}_{\zeta \in U}$ is $(k+1)$-times ($\mathscr{L}\mathcal{L}$-continuously) $\mathscr{L}\mathcal{L}$-differentiable if, for any $l \in \mathbb{S}(\mathcal{U})$, $\{\mathscr{L}\mathcal{L}\text{-}\frac{\partial}{\partial l}Q^{\zeta}\}_{\zeta \in U}$ is $k$-times ($\mathscr{L}\mathcal{L}$-continuously) $\mathscr{L}\mathcal{L}$-differentiable.
\end{defn}

\begin{defn}
    Let $\delta,\delta' \in [0,1]$.
    Let $\{(f^\zeta,(f^\zeta)')\}_{\zeta \in U} \subseteq \mathbf{D}_{W}^{\delta,\delta'}\mathscr{L}\mathcal{L}$ be a family of stochastic controlled linear vector fields. 
    We say that $\{(f^\zeta,(f^\zeta)')\}_{\zeta \in U}$ is $\mathbf{D}_W^{\delta,\delta'}\mathscr{L}\mathcal{L}$-differentiable at $\zeta^0 \in U$ in the direction $l \in \mathbb{S}(\mathcal{U})$ if \begin{equation*}
        \left\{  \left( \frac{1}{\varepsilon}(f^{\zeta^0 + \varepsilon l}-f^{\zeta^0}),\frac{1}{\varepsilon}\big((f^{\zeta^0 + \varepsilon l})'-(f^{\zeta^0})'\big) \right) \right\}_{\varepsilon \in \R \setminus \{0\}}
    \end{equation*} 
    admits limit in $\mathbf{D}_W^{\delta,\delta'}\mathscr{L}\mathcal{L}$ as $\varepsilon \to 0$, in the sense of \cref{def:convergenceandcontinuitystochasticcontrolledlinearvectorfields}.
    We denote such a limit as \begin{equation*}
        \left(\mathbf{D}_W^{\delta,\delta'}\mathscr{L}\mathcal{L}\text{-}\frac{\partial}{\partial l}f^{\zeta^0},\mathbf{D}_W^{\delta,\delta'}\mathscr{L}\mathcal{L}\text{-}\frac{\partial}{\partial l}(f^{\zeta^0})' \right).
    \end{equation*}
    We say that $\{(f^\zeta,(f^\zeta)')\}_{\zeta \in U}$ is $\mathbf{D}_W^{\delta,\delta'}\mathscr{L}\mathcal{L}$-differentiable if it is $\mathbf{D}_W^{\delta,\delta'}\mathscr{L}\mathcal{L}$-differentiable at every $\zeta^0 \in U$ in every direction $l \in \mathbb{S}(\mathcal{U})$. 
    We say that it is $\mathbf{D}_W^{\delta,\delta'}\mathscr{L}\mathcal{L}$-continuously $\mathbf{D}_W^{\delta,\delta'}\mathscr{L}\mathcal{L}$-differentiable if, in addition, $\{(\mathbf{D}_W^{\delta,\delta'}\mathscr{L}\mathcal{L}\text{-}\frac{\partial}{\partial l}f^\zeta,\mathbf{D}_W^{\delta,\delta'}\mathscr{L}\mathcal{L}\text{-}\frac{\partial}{\partial l}(f^\zeta)')\}_{\zeta \in U}$ is     
    $\mathbf{D}_W^{\delta,\delta'}\mathscr{L}\mathcal{L}$-continuous, for any $l \in \mathbb{S}(\mathcal{U})$.
    Recursively on $k \in \mathbb{N}_{\ge 1}$, $\{(f^\zeta,(f^\zeta)')\}_{\zeta \in U}$ is said to be $(k+1)$-times ($\mathbf{D}_W^{\delta,\delta'}\mathscr{L}\mathcal{L}$-continuously) $\mathbf{D}_W^{\delta,\delta'}\mathscr{L}\mathcal{L}$-differentiable if, for any $l \in \mathbb{S}(\mathcal{U})$,  $\{(\mathbf{D}_W^{2\alpha'}\mathscr{L}\mathcal{L}\text{-}\frac{\partial}{\partial l}f^\zeta,\mathbf{D}_W^{\delta,\delta'}\mathscr{L}\mathcal{L}\text{-}\frac{\partial}{\partial l}(f^\zeta)')\}_{\zeta \in U}$ is $k$-times ($\mathbf{D}_W^{\delta,\delta'}\mathscr{L}\mathcal{L}$-continuously) $\mathbf{D}_W^{\delta,\delta'}\mathscr{L}\mathcal{L}$-differentiable.
\end{defn}


\begin{thm} \label{thm:differentiability_linearRSDEs}
    Let $\mathbf{W}=(W,\mathbb{W}) \in \mathscr{C}^\alpha([0,T];\R^n)$ be an $\alpha$-rough path, with $\alpha \in (\frac{1}{3},\frac{1}{2})$, and let $\delta \in [0,\alpha]$, $\delta' \in [0,1]$ with $\alpha + \delta > \frac{1}{2}$ and $\alpha + \delta + \min\{\delta,\delta'\} >1$.     
    Let $k \in \mathbb{N}_{\ge 1}$. 
    Assume that $\{\xi^\zeta\}_{\zeta \in U}$ is a family of $k$-times ($\mathscr{L}(V)$-continuously) $\mathscr{L}(V)$-differentiable $\mathcal{F}_0$-measurable random variables, and that $\{(F^\zeta,(F^\zeta)')\}_{\zeta \in U}$ is $k$-times ($\mathbf{D}_W^{\alpha,\delta}\mathscr{L}(V)$-continuously) $\mathbf{D}_W^{\alpha,\delta}\mathscr{L}(V)$-differentiable.
    Assume $\{G^\zeta\}_{\zeta \in U}$ to be $k$-times ($\mathscr{L}\mathcal{L}(V;V)$-continuously) $\mathscr{L}\mathcal{L}(V;V)$-differentiable
    and $\{S^\zeta\}_{\zeta \in U}$ to be $k$-times ($\mathscr{L}\mathcal{L}(V;\mathcal{L}(\R^m;V))$-continuously) $\mathscr{L}\mathcal{L}(V;\mathcal{L}(\R^m;V))$-differentiable.
    Assume that $\{(f^\zeta,(f^\zeta)')\}_{\zeta \in U}$ is $k$-times ($\mathbf{D}_W^{\delta,\delta'}\mathscr{L}\mathcal{L}(V;\mathcal{L}(\R^n;V))$-continuously) $\mathbf{D}_W^{\delta,\delta'}\mathscr{L}\mathcal{L}(V;\mathcal{L}(\R^n;V))$-differentiable.
    For any $\zeta \in U$, let $Y^\zeta$ be a $L_{p}$-integrable solution of the following linear rough SDE, for any $p \in [2,\infty)$: \begin{equation*}
        Y_t^{\zeta} = \xi^{\zeta} + F_t^{\zeta} + \int_0^t G_r^\zeta Y_r^{\zeta}d r + \int_0^t S_r^\zeta Y_r^{\zeta} d B_r + \int_0^t (f_r^\zeta,(f^\zeta_r)')Y^\zeta_r d\mathbf{W}_r \qquad t \in [0,T],
    \end{equation*}
    and assume that, for any $R>0$ and for any $p \in [2,\infty)$,
    \begin{equation*}
      \sup_{|\zeta|\le R} \left(\|G^\zeta\|_\infty + \|S^\zeta\|_\infty + \llbracket (f^\zeta,(f^\zeta)') \rrbracket_{\mathcal{L};\infty} + \llbracket (f^\zeta,(f^\zeta)') \rrbracket_{W;\delta,\delta';p,\infty;\mathcal{L}} \right) <+\infty.
    \end{equation*}     
    Then $\{(Y^\zeta,f^\zeta Y^\zeta + (F^\zeta)')\}_{\zeta \in U}$ is $k$-times ($\mathbf{D}_W^{\alpha,\delta}\mathscr{L}(V)$-continuously) $\mathbf{D}_W^{\alpha,\delta}\mathscr{L}(V)$-differentiable. 
    In particular, given any $\zeta \in U$ and $l \in \mathbb{S}(\mathcal{U})$ - and denoting by $Z^{\zeta,l} = \mathbf{D}_W^{\alpha,\delta}\mathscr{L}\text{-}\frac{\partial}{\partial l} Y^\zeta$ - it holds that, for any $p \in [2,\infty)$, the latter is an $L_{p}$-integrable solution of the following linear RSDE on $[0,T]$ with coefficients defined as in \eqref{eq:coefficientsofZ^zeta}:
    \begin{equation*}
        Z_t^{\zeta,l} = \mathscr{L}\text{-}\frac{\partial}{\partial l} \xi^\zeta + \tilde{F}_t^\zeta + \int_0^t G_r^\zeta Z_r^{\zeta,l} dr + \int_0^t S_r^\zeta Z_r^{\zeta,l} dB_r + \int_0^t (f_r^\zeta,(f^\zeta)'_r) Z_r^{\zeta,l} d\mathbf{W}_r .
    \end{equation*}
\end{thm}

\begin{proof}
    We argue recursively on $k$. Let us first consider the case $k=1$.
    Let $\zeta^0 \in U$ and $l \in \mathbb{S}(\mathcal{U})$ be arbitrary.
    For any $\varepsilon \in \R \setminus \{0\}$, we define the following $V$-valued stochastic process: \begin{equation*}
        Z^\varepsilon := \frac{1}{\varepsilon} \left(Y^{\zeta^0 + \varepsilon l}-Y^{\zeta^0}\right).
    \end{equation*}
    The idea is to show that $(Z^\varepsilon_t)_t$ solves a linear rough SDE, and use the results on linear rough SDEs to prove its convergence as $\varepsilon \to 0$.
    Let us therefore show that all the requirements in \cref{def:solutionRSDEs_compendium} are satisfied. 
    For any $(s,t)\in \Delta_{[0,T]}$ and by considering the Davie-type expansions of both $Y^{\zeta^0+\varepsilon l}$ and $Y^{\zeta^0}$ as solutions to some linear rough SDEs, we have \begin{equation*}
        \begin{aligned}
            \delta Z_{s,t}^\varepsilon &= \frac{1}{\varepsilon  } (\delta Y_{s,t}^{\zeta^0 + \varepsilon  l}-Y_{s,t}^{\zeta^0}) = \notag \\
            &= \frac{1}{\varepsilon  } (\delta F^{\zeta^0 + \varepsilon  l}_{s,t}-\delta F_{s,t}^{\zeta^0}) + \frac{1}{\varepsilon  } \int_s^t G_r^{\zeta^0+\varepsilon   l} Y_r^{\zeta^0+\varepsilon  l}- G_r^{\zeta^0}Y_r^{\zeta^0} dr + \frac{1}{\varepsilon  } \int_s^t S_r^{\zeta^0+\varepsilon   l} Y_r^{\zeta^0+\varepsilon  l}- S_r^{\zeta^0} Y_r^{\zeta^0} dB_r +  \\
            &\quad + \frac{1}{\varepsilon  } \big(f_s^{\zeta^0+\varepsilon   l} Y_r^{\zeta^0+\varepsilon  l}- f_s^{\zeta^0}Y_r^{\zeta^0}\big) \delta W_{s,t} + \frac{1}{\varepsilon  } \big( f_s^{\zeta^0+\varepsilon  l}(f_s^{\zeta^0+\varepsilon  l} Y_s^{\zeta^0+\varepsilon  l}+(F^{\zeta^0+\varepsilon  l})'_s) - f_s^{\zeta^0}(f_s^{\zeta^0}Y_s^{\zeta^0}+(F^{\zeta^0})'_s \big) \mathbb{W}_{s,t} \\
            &\quad + \frac{1}{\varepsilon  } \big( (f^{\zeta^0+\varepsilon  l})'_s Y_s^{\zeta^0+\varepsilon  l} - (f^{\zeta^0})'_s Y_s^{\zeta^0} \big) \mathbb{W}_{s,t} + Z^{\varepsilon,\natural}_{s,t} 
        \end{aligned}
    \end{equation*} 
    where $Z^{\varepsilon,\natural}_{s,t}:= \frac{1}{\varepsilon  } (Y^{\zeta^0+\varepsilon  l,\natural}_{s,t} - Y^{\zeta^0,\natural}_{s,t})$ is such that $\|Z^{\varepsilon,\natural}_{s,t}\|_p \lesssim |t-s|^{\alpha+\delta}$ and $\|\E_s(Z^{\varepsilon,\natural}_{s,t})\|_p \lesssim |t-s|^{\alpha+\delta+\delta'}$ for any $p \in [2,\infty)$.
    By the linearity of the coefficients, we can rewrite all the terms on the right-hand side in a more convenient way. 
    For instance,
    \begin{equation*} \begin{aligned}
        \frac{1}{\varepsilon  } (G_r^{\zeta^0+\varepsilon   l} Y_r^{\zeta^0+\varepsilon  l}- G_r^{\zeta^0}Y_r^{\zeta^0}) 
        &= \frac{1}{\varepsilon  } (G_r^{\zeta^0+\varepsilon   l} Y_r^{\zeta^0+\varepsilon  l}- G_r^{\zeta^0+\varepsilon  l} Y_r^{\zeta^0}) + \frac{1}{\varepsilon  } (G_r^{\zeta^0+\varepsilon   l} Y_r^{\zeta^0}- G_r^{\zeta^0} Y_r^{\zeta^0}) = \\
        &= G_r^{\zeta^0+\varepsilon   l} Z_r^\varepsilon  +  \left( \frac{1}{\varepsilon  } (G_r^{\zeta^0+\varepsilon   l}- G_r^{\zeta^0}) \right) Y_r^{\zeta^0} =: G_r^{\zeta^0+\varepsilon   l} Z_r^\varepsilon + \mu_r^\varepsilon, 
    \end{aligned} 
    \end{equation*}
    and it is possible to prove - by a simple application of H\"older's inequality - that $\mu^\varepsilon  \in \mathscr{L}(V)$.
    By assumption we have that $\mathscr{L}\mathcal{L}(V;V)\text{-}\lim_{\varepsilon \to 0, \varepsilon \ne 0} G^{\zeta^0 + \varepsilon   l}=G^{\zeta^0}$ and we also deduce that $\mathscr{L}(V)\text{-}\lim_{\varepsilon\to 0, \varepsilon \ne 0}\mu^\varepsilon = (\mathscr{L}\mathcal{L}\text{-}\frac{\partial}{\partial l} G^{\zeta^0}) Y^{\zeta^0}$.
    Similarly, \begin{align*}
        \frac{1}{\varepsilon  } (S_r^{\zeta^0+\varepsilon   l} Y_r^{\zeta^0+\varepsilon  l}- S_r^{\zeta^0}Y_r^{\zeta^0}) &= S_r^{\zeta^0+\varepsilon   l} Z_r^\varepsilon  + \left(\frac{1}{\varepsilon  } (S_r^{\zeta^0+\varepsilon   l} - S_r^{\zeta^0})\right)Y_r^{\zeta^0} =: S_r^{\zeta^0+\varepsilon   l} Z_r^\varepsilon  + \nu_r^\varepsilon,
    \end{align*} 
    and we have that $\nu^\varepsilon  \in \mathscr{L}(\mathcal{L}(\R^m;V))$ and $S^{\zeta^0}=\mathscr{L}\mathcal{L}(V;\mathcal{L}(\R^m;V))\text{-}\lim_{\varepsilon \to 0, \varepsilon \ne 0} S^{\zeta^0 + \varepsilon l}$. 
    Again, a simple application of H\"older's inequality shows that $\mathscr{L}(\mathcal{L}(\R^m;V))\text{-}\lim_{\varepsilon\to 0,\varepsilon\ne 0}\nu^\varepsilon =  (\mathscr{L}\mathcal{L}\text{-}\frac{\partial}{\partial l} S^{\zeta^0})Y^{\zeta^0}$.
    The remaining terms can be written as        
    \begin{align*}
        \frac{1}{\varepsilon  } (f_s^{\zeta^0+\varepsilon   l} Y_s^{\zeta^0+\varepsilon  l}- f_s^{\zeta^0}Y_s^{\zeta^0}) &= f_s^{\zeta^0+\varepsilon   l} Z_s^\varepsilon  + \left(\frac{1}{\varepsilon  } (f_s^{\zeta^0+\varepsilon   l}- f_s^{\zeta^0})\right) Y_s^{\zeta^0} =: f_s^{\zeta^0+\varepsilon   l} Z_s^\varepsilon  + \phi_s^\varepsilon,  \\
        \frac{1}{\varepsilon  } ((f^{\zeta^0+\varepsilon   l})'_s Y_s^{\zeta^0+\varepsilon  l}- (f^{\zeta^0}Y_s^{\zeta^0})'_s) &= (f^{\zeta^0+\varepsilon   l})'_s Z_s^\varepsilon  + \left( \frac{1}{\varepsilon  } ((f^{\zeta^0+\varepsilon   l})'_s - (f^{\zeta^0})'_s)\right) Y_s^{\zeta^0} =: (f^{\zeta^0+\varepsilon   l})'_s Z_s^\varepsilon  + (\phi^\varepsilon_2)'_s, 
    \end{align*}
    and
    \begin{align*}
        &\frac{1}{\varepsilon  } ( f_s^{\zeta^0+\varepsilon  l} (f_s^{\zeta^0+\varepsilon  l} Y_s^{\zeta^0+\varepsilon  l}+(F^{\zeta^0+\varepsilon  l})'_s) - f_s^{\zeta^0}(f_s^{\zeta^0}Y_s^{\zeta^0}+(F^{\zeta^0})'_s ) = \\
        &= f_s^{\zeta^0+\varepsilon  l} f_s^{\zeta^0+\varepsilon  l} Z_s^\varepsilon  + f_s^{\zeta^0+\varepsilon  l} \left(\frac{1}{\varepsilon  } ( f_s^{\zeta^0+\varepsilon  l} - f_s^{\zeta^0}) \right) Y_s^{\zeta^0} + \left(\frac{1}{\varepsilon  } ( f_s^{\zeta^0+\varepsilon  l}  - f_s^{\zeta^0}) \right) f_s^{\zeta^0} Y_s^{\zeta^0} + \\
        &\quad + f_s^{\zeta^0+\varepsilon  l} \frac{1}{\varepsilon  } ((F^{\zeta^0+\varepsilon  l})'_s-(F^{\zeta^0})'_s) + \left( \frac{1}{\varepsilon  } (f_s^{\zeta^0+\varepsilon  l} - f_s^{\zeta^0})\right) (F^{\zeta^0})'_s \\
        &=: f_s^{\zeta^0+\varepsilon  l} \left(f_s^{\zeta^0+\varepsilon  l}Z_s^\varepsilon  + \frac{1}{\varepsilon  } ((F^{\zeta^0+\varepsilon  l})'_s-(F^{\zeta^0})'_s) \right)+ (\phi_1^\varepsilon )'_s.
    \end{align*}
    From \cite[Proposition 3.3]{BCN24} we see that $(\phi^\varepsilon ,(\phi^\varepsilon )':=(\phi_1^\varepsilon )' + (\phi_2^\varepsilon )') \in \mathbf{D}_W^{\delta,\delta'}\mathscr{L}(\mathcal{L}(\R^n;V))$, being the space of stochastic controlled linear vector fields a linear space.
    Moreover, by assumption we have that $$\mathbf{D}_W^{\delta,\delta'}\mathscr{L}\mathcal{L}(V;\mathcal{L}(\R^n;V))\text{-}\lim_{\varepsilon \to 0, \varepsilon \ne 0} (f^{\zeta^0 + \varepsilon   l},(f^{\zeta^0 + \varepsilon   l})')  = (f^{\zeta^0},(f^{\zeta^0})'),$$
    and from \cref{prop:stambilityundercompositionlinearvectorfields} below it is straightforward to deduce that $\{\phi^\varepsilon ,(\phi^\varepsilon )'\}_{\varepsilon \in \R \setminus \{0\}}$ tends to \begin{equation*}
         \left( (\mathbf{D}_W^{\delta,\delta'}\mathscr{L}\mathcal{L}\text{-}\frac{\partial}{\partial l} f^{\zeta^0})Y^{\zeta^0}, (\mathbf{D}_W^{\delta,\delta'}\mathscr{L}\mathcal{L}\text{-}\frac{\partial}{\partial l} (f^{\zeta^0})') Y^{\zeta^0} + \mathbf{D}_W^{\delta,\delta'}\mathscr{L}\mathcal{L}\text{-}\frac{\partial}{\partial l} f^{\zeta^0} \left((\mathbf{D}_W^{\delta,\delta'}\mathscr{L}\mathcal{L}\text{-}\frac{\partial}{\partial l} f^{\zeta^0})Y^{\zeta^0} + (F^{\zeta^0})'\right) \right) 
    \end{equation*} 
    in $\mathbf{D}_W^{\delta,\lambda}\mathscr{L}(\mathcal{L}(\R^n;V))$ as $\varepsilon \to 0$, where $\lambda=\min\{\delta,\delta'\}$. 
    Putting everything together, we have shown that $Z^\varepsilon $ solves the following linear rough SDE: \begin{equation*}
        Z_t^\varepsilon  = Z^\varepsilon _0 + \tilde{F}^\varepsilon _t + \int_0^t G_r^{\zeta^0+\varepsilon  l} Z_r^\varepsilon  dr + \int_0^t S_r^{\zeta^0+\varepsilon  l}Z_r^\varepsilon  dB_r + \int_0^t (f_r^{\zeta^0+\varepsilon  l},(f^{\zeta^0+\varepsilon  l})'_r) Z_r^\varepsilon  d\mathbf{W}_r \quad t \in [0,T]
    \end{equation*}
    with \begin{equation*} \begin{aligned}
        (\tilde{F}^\varepsilon , (\tilde{F}^\varepsilon )') &= \Big( \frac{1}{\varepsilon  } (F^{\zeta^0+\varepsilon  l}-F^{\zeta^0}) + \int_0^\cdot \mu_r^\varepsilon  dr + \int_0^\cdot \nu_r^\varepsilon  dB_r + \int_0^\cdot (\phi_r^\varepsilon ,(\phi^\varepsilon )'_r) d\mathbf{W}_r, \frac{1}{\varepsilon  } \big( (F^{\zeta^0+\varepsilon  l})'-(F^{\zeta^0})' \big) +\phi^\varepsilon  \Big).
    \end{aligned}
    \end{equation*}
    Let us now consider $(Z^\zeta)_{\zeta \in U}$ to be family of $L_p$-integrable solutions to the following (well-posed) linear rough SDE, for any $p \in [2,\infty)$: 
    \begin{equation} \label{eq:definitionofZ^zeta}
        \begin{aligned}
             Z_t^\zeta = \mathscr{L}\text{-}\frac{\partial}{\partial l} \xi^\zeta + \tilde{F}_t^\zeta + \int_0^t G_r^\zeta Z_r^\zeta dr + \int_0^t S_r^\zeta Z_r^\zeta dB_r + \int_0^t (f_r^\zeta,(f^\zeta)'_r) Z_r^\zeta d\mathbf{W}_r \qquad t \in [0,T]
        \end{aligned}
    \end{equation}
    with \begin{equation} \label{eq:coefficientsofZ^zeta}
    \begin{aligned}
            (\tilde{F}^{\zeta},(\tilde{F}^{\zeta})') &= \left(\mathbf{D}_W^{\alpha,\delta}\mathscr{L}\text{-}\frac{\partial}{\partial l} F^{\zeta} + \int_0^\cdot \mu_r^{\zeta} dr + \int_0^\cdot \nu^\zeta dB_r + \int_0^\cdot (\phi^\zeta,(\phi^\zeta)') d\mathbf{W}_r, \mathbf{D}_W^{\alpha,\delta}\mathscr{L}\text{-} \frac{\partial}{\partial l} (F^\zeta)' + \phi^\zeta \right) \\
            \mu^\zeta &= (\mathscr{L}\mathcal{L}\text{-}\frac{\partial}{\partial l} G^\zeta) Y^\zeta  \\
            \nu^\zeta &= (\mathscr{L}\mathcal{L}\text{-}\frac{\partial}{\partial l} S^{\zeta}) Y^{\zeta} \\
            \phi^\zeta &=  (\mathbf{D}_W^{\delta,\delta'}\mathscr{L}\mathcal{L}\text{-}\frac{\partial}{\partial l} f^{\zeta})Y^{\zeta} \\
            (\phi^\zeta)' &=  (\mathbf{D}_W^{\delta,\delta'}\mathscr{L}\mathcal{L}\text{-}\frac{\partial}{\partial l} (f^{\zeta})') Y^{\zeta} + \mathbf{D}_W^{\delta,\delta'}\mathscr{L}\mathcal{L}\text{-}\frac{\partial}{\partial l} f^{\zeta} \left((\mathbf{D}_W^{\delta,\delta'}\mathscr{L}\mathcal{L}\text{-}\frac{\partial}{\partial l} f^{\zeta})Y^{\zeta} + (F^{\zeta})'\right).
    \end{aligned}
    \end{equation}
    From \cref{prop:continuity_linearRSDEs} and  \cref{prop:continuityanddifferentiability_forcingterm} below, due to the previous considerations we have that \begin{equation*}
         (Z^{\zeta^0},f^{\zeta^0}Z^{\zeta^0} + (\tilde{F}^{\zeta^0})') = \mathbf{D}_W^{\alpha,\delta}\mathscr{L}(V) \text{-} \lim_{\varepsilon \to 0, \varepsilon \ne 0} (Z^\varepsilon ,f^{\zeta^0+\varepsilon^\varepsilon  l} Z^\varepsilon  + (\tilde{F}^\varepsilon )').
    \end{equation*}
    In other words, $\{(Z^\zeta,f^\zeta + (\tilde{F}^\zeta)')\}_{\zeta \in U}$ is the $\mathbf{D}_W^{\alpha,\delta}\mathscr{L}$-derivative of $\{(Y^\zeta,f^\zeta Y^\zeta + (F^\zeta)')\}_{\zeta \in U}$ in the direction $l$. 
    The $\mathbf{D}_W^{\alpha,\delta}\mathscr{L}$-continuity of $Z^\zeta$ easily follows from the assumptions, considering again \cref{prop:continuity_linearRSDEs}. \\

    Let us now assume the assertion is proved for $k=k_0$ and that the assumptions of the theorem are satisfied for $k=k_0+1$. It is sufficient to show that, for any arbitrary direction $l \in \mathbb{S}(\mathcal{U})$, the family $\{(Z^\zeta,f^\zeta Z^\zeta + (\tilde{F}^\zeta)')\}_{\zeta\in U}$ defined as in \eqref{eq:definitionofZ^zeta} is $k_0$-times ($\mathbf{D}_W^{\alpha,\delta}\mathscr{L}$-continuously) $\mathbf{D}_W^{\alpha,\delta}\mathscr{L}$-differentiable. 
    Indeed, $Z^{\zeta}$ itself solves a linear RSDE whose coefficients and initial condition are $k_0$-times differentiable.
    Moreover, from inductive assumptions and \cref{prop:continuityanddifferentiability_forcingterm}, it holds that $\{(\tilde{F}^{\zeta},(\tilde{F}^{\zeta})')\}_{\zeta \in U}$ is $k_0$-times ($\mathbf{D}_W^{\alpha,\delta}\mathscr{L}$-continuously) $\mathbf{D}_W^{\alpha,\delta}\mathscr{L}$-differentiable. 
    Being the assertion valid for $k=k_0$, the proof is concluded. 
\end{proof}

\subsection{Differentiability for rough SDEs} \label{section:differentiabilityforRSDEs}

Let $(\mathcal{U},|\cdot|)$ be a Banach space and let $U \subseteq \mathcal{U}$ be a non-empty open subset.
Let $\mathbf{W}=(W,\mathbb{W}) \in \mathscr{C}^\alpha([0,T];\R^n)$ be a rough path taking values in $\R^n$, with $\alpha \in (\frac{1}{3},\frac{1}{2})$ \footnote{The case $\alpha = \frac{1}{2}$ is not included because in the proof of \cref{thm:differentiability_RSDEs} we are applying \cref{thm:differentiability_linearRSDEs} which in turn is based on \cref{thm:stabilityforlinearRSDEs_compendium} and \cref{prop:continuityanddifferentiability_forcingterm}}.
We consider the following progressively measurable mappings: 
\begin{align*}
    b&: \Omega \times [0,T] \times U \times \R^d \to \mathbb{R}^d, \quad (\omega,t,\zeta,x) \mapsto b_t(\zeta,x)  \\
    \sigma&: \Omega \times [0,T] \times U \times \mathbb{R}^d \to \mathbb{R}^{d \times m}, \quad (\omega,t,\zeta,x) \mapsto \sigma_t(\zeta,x) \\
    \beta&: \Omega \times [0,T] \times U \times \R^d \to \R^{d\times n}, \quad (\omega,t,\zeta,x) \mapsto \beta_t(\zeta,x) \\
    \beta'&: \Omega \times [0,T] \times U \times \R^d \to \mathcal{L}(\R^n,\R^{d\times n}), \quad (\omega,t,\zeta,x) \mapsto \beta'_t(\zeta,x).
\end{align*} 
We assume that, for any $\zeta \in U$, there is a $L_{p,\infty}$-integrable solution $X^\zeta$ to the following rough SDE with $\mathcal{F}_0$-measurable initial condition $\xi^\zeta$, for any $p \in [2,\infty)$ (cf. \cref{def:solutionRSDEs_compendium}):
\begin{equation} \label{eq:RSDEwithparameter}
     X_t^{\zeta} = \xi^{\zeta} + \int_0^t b_r (\zeta, X_r^{\zeta}) dr + \int_0^t \sigma_r (\zeta, X_r^{\zeta}) dB_r + \int_0^t (\beta_r,\beta'_r)(\zeta,X_r^\zeta) d \mathbf{W}_r \qquad t \in [0, T],
\end{equation}
where $(\beta_t,\beta'_t)(\zeta,X^\zeta_t) := (\beta_t(\zeta,X^\zeta_t),D_x\beta(\zeta,X^\zeta_t)\beta_t(\zeta,X^\zeta_t) + \beta'_t(\zeta,X^\zeta_t))$.

\begin{thm} \label{thm:differentiability_RSDEs}
     Let $\delta \in [0,\alpha]$ such that $\alpha + \delta > \frac{1}{2}$ and $\alpha + 2\delta >1$, and let $k \in \mathbb{N}_{\ge 1}$.
     Assume that $\{\xi^{\zeta}\}_{\zeta \in U}$ is $k$-times ($\mathscr{L}(\R^d)$-continuously) $\mathscr{L}(\R^d)$-differentiable. Assume that, for $\mathbb{P}$-almost every $\omega$ and for any $t \in [0,T]$, \begin{equation*}
         b_t (\cdot, \cdot) \in C^k  (U \times \mathbb{R}^d ; \mathbb{R}^d) \quad \text{and} \quad \sigma_t (\cdot, \cdot) \in C^k  (U \times \mathbb{R}^d ; \mathbb{R}^{d \times m})
     \end{equation*} in Fréchet sense and the functions and all their (partial) derivatives up to order $k$ are uniformly bounded in $(\omega, t, \zeta, x)$.
     Moreover, assume \begin{equation} \label{eq:stochasticcontrolledvectorfieldintheRSDE}
         (\beta,\beta') \in \bigcap_{p \in [2,\infty)} \mathbf{D}_W^{2\delta} L_{p,\infty} C_b^{\kappa}(U \times \R^d; \R^{d\times n})
     \end{equation}
     \textcolor{black}{for some $\kappa \in (2+k,3+k]$,} in the sense of \cref{def:stochasticcontrolledvectorfield}.
     Then $\{(X^{\zeta},\beta(\zeta,X^\zeta)\}_{\zeta \in U}$ is $k$-times ($\mathbf{D}_W^{\alpha,\delta}\mathscr{L}(\R^d)$-continuously) $\mathbf{D}_W^{\alpha,\delta}\mathscr{L}(\R^d)$-differentiable.
     In particular, given $\zeta \in U$ and $l \in \mathbb{S}(\mathcal{U})$ and denoting by $Y^{\zeta,l}= \mathbf{D}_W^{\alpha,\delta}\mathscr{L}\text{-} \frac{\partial}{\partial l} X^\zeta$, it holds that, for any $p \in [2,\infty)$, the latter is an $L_{p}$-integrable solution of the following linear rough SDE on $[0,T]$ whose coefficients are defined as in \eqref{eq:coefficientsforBLderivative} : \begin{equation*}
         Y_t^{\zeta,l} = \mathscr{L}\text{-}\frac{\partial}{\partial l} X_0^\zeta + F^{\zeta}_t + \int_0^t G_r^\zeta Y_r^{\zeta,l} dr + \int_0^t S_r^\zeta Y_r^{\zeta,l} dB_r + \int_0^t (f_r^\zeta ,(f^\zeta)'_r) Y_r^{\zeta,l} d\mathbf{W}_r .
     \end{equation*}
\end{thm}

\begin{proof} 
Notice that, for any $\zeta \in U$, $b (\zeta,\cdot)$ is a random bounded Lipschitz vector fields with $\sup_{\zeta \in U} \|b(\zeta,\cdot)\|_\infty <+\infty$.
The same holds replacing $b$ with $\sigma$.
From \cref{rmk:propertiesofstochasticcontrolledvectorfields}, we deduce $$(\beta(\zeta,\cdot),\beta'(\zeta,\cdot)) \in \bigcap_{p \in [2,\infty)} \mathbf{D}_W^{2\delta} L_{p,\infty} C_b^{\kappa}(\R^d; \R^{d\times n})$$ with \begin{equation*}
    \sup_{\zeta \in U} \left( \llbracket (\beta(\zeta,\cdot),\beta'(\zeta,\cdot)) \rrbracket_{\kappa;\infty} +  \llbracket (\beta(\zeta,\cdot),\beta'(\zeta,\cdot))\rrbracket_{W;2\delta;p,\infty;\kappa} \right) < +\infty, \quad \text{for any $p \in [2,\infty)$}.
\end{equation*}
From \cref{thm:wellposednessandaprioriestimatesforRSDEs} it follows that \begin{equation*}
    \sup_{\zeta \in U} \|(X^\zeta,\beta(\zeta,X^\zeta))\|_{\mathbf{D}_W^{\alpha,\delta}L_{p,\infty}} <+\infty. 
\end{equation*}
Moreover, the following facts hold true (recall that, by construction, $\kappa-1 \ge 2$):
\begin{itemize}
    \item[-] $(D_x\beta,D_x\beta') \in \bigcap_{p \in [2,\infty)} \mathbf{D}_W^{2\delta} L_{p,\infty} C_b^{\kappa-1}(U \times \R^d; \mathcal{L}(\R^d;\R^{d\times n}))$;
    \item[-] $(D_\zeta \beta,D_\zeta \beta') \in \bigcap_{p \in [2,\infty)} \mathbf{D}_W^{2\delta} L_{p,\infty} C_b^{\kappa-1}(U \times \R^d; \mathcal{L}(U;\R^{d\times n}))$.
\end{itemize}

It is straight-forward to see that, due to the assumptions on $b$, $\{b(\zeta,\cdot)\}_{\zeta \in U}$ is $\mathscr{L}C^0_b(\R^d;\R^d)$-continuous in the sense of \cref{def:convergenceforstochasticboundedcontinuousvectorfields}.
Similarly, we also have that $\{\sigma(\zeta,\cdot)\}_{\zeta \in U}$ is $\mathscr{L}C^0_b(\R^d;\R^{d \times m})$-continuous. 
One can also prove that $\{(\beta(\zeta,\cdot),\beta'(\zeta,\cdot))\}_{\zeta \in U}$ is $\textbf{D}_W^{2\delta} \mathscr{L} C_b^3(\R^d;\R^{d \times n})$-continuous in the sense of \cref{def:convergenceforstochasticcontrolledvetorfields} (notice that, even with $k=1$ we have $\kappa \in (3,4]$). That is, for any $\zeta^0 \in U$ accumulation point and for any $(\zeta^n)_{n\ge1} \subseteq U \setminus \{\zeta^0\}$ such that $|\zeta^n-\zeta^0| \to_{n \to +\infty} 0$,  \begin{equation} \label{eq:normsbeta_differentiability}
    \llbracket (\beta(\zeta^n,\cdot)-\beta(\zeta^0,\cdot),\beta'(\zeta^n,\cdot)-\beta'(\zeta^0,\cdot)) \rrbracket_{3;p} + \llbracket \beta(\zeta^n,\cdot),\beta'(\zeta^n,\cdot) ;  \beta(\zeta^0,\cdot),\beta'(\zeta^0,\cdot) \rrbracket_{W;2\delta;p;3} \to 0
\end{equation}
as $n \to +\infty$. For example, one has that uniformly in $x \in \R^d$ and $\mathbb{P}$-almost surely \begin{equation*}
        \begin{aligned}
            &|\E_s(\beta_t(\zeta^n,x) - \beta_s(\zeta^n,x) - \beta'_s(\zeta^n,x) \delta W_{s,t} - \beta_t(\zeta^0,x) - \beta_s(\zeta^0,x) - \beta'_s(\zeta^0,x) \delta W_{s,t} ) |  \\
            &= \Big| \int_0^1 \E_s \big( D_\zeta \beta_t((1-\theta)\zeta^0+ \theta \zeta^n,x) - D_\zeta \beta_s ((1-\theta)\zeta^0+ \theta \zeta^n,x) + \\
            & \qquad  - D_\zeta \beta_s'((1-\theta)\zeta^0+ \theta \zeta^n,x) \delta W_{s,t} \big) d\theta (\zeta^n-\zeta^0) \Big|  \le  \|\E_s(R_{s,t}^{D_\zeta \beta})\|_\infty |\zeta^n - \zeta^0|.
        \end{aligned}
    \end{equation*}
Similar computations apply for any other term in \eqref{eq:normsbeta_differentiability}. 
Being $\{\xi^\zeta\}_{\zeta \in U}$ $\mathscr{L}(\R^d)$-continuous by assumption and applying \cref{thm:continuity_RSDEs} with $\mathbf{W}^\zeta \equiv \mathbf{W}$, we deduce that $\{(X^\zeta,\beta(\zeta,X^\zeta))\}_{\zeta \in U}$ is $\mathbf{D}_W^{\alpha,\delta}\mathscr{L}(\R^d)$-continuous. \\

  We can now prove the statement of the theorem, and we argue recursively on $k$. 
  First, let $k = 1$. In this case $\kappa \in (3,4]$. 
  Let $\zeta^0 \in U$ and $l \in \mathbb{S}(\mathcal{U})$ be arbitrary.
  For any $\varepsilon \in \R \setminus \{0\}$, we define the following stochastic process \begin{equation*}
      Y^\varepsilon := \frac{1}{\varepsilon} \left(X^{\zeta^0 + \varepsilon l } - X^{\zeta^0} \right).
  \end{equation*}
  It is sufficient to show that $\{(Y^\varepsilon,\frac{1}{\varepsilon}(\beta(\zeta^0+\varepsilon l,X^{\zeta^0 + \varepsilon l}) - \beta(\zeta^0,X^{\zeta^0})))\}_{\varepsilon \in \R \setminus \{0\}}$ admits limit in $\mathbf{D}_W^{\alpha,\delta}\mathscr{L}$ (and the limit is $\mathbf{D}_W^{\alpha,\delta}\mathscr{L}$-continuous). 
  Let us first observe that $Y^\varepsilon$ is the solution of a linear RSDE. Indeed, for any $(s,t) \in \Delta_{[0,T]}$ we can write \begin{equation*}
      \begin{aligned}
          \delta Y_{s,t}^{\varepsilon} &= \frac{1}{\varepsilon} \left(\delta X_{s,t}^{\zeta^0 + \varepsilon l } - \delta X_{s,t}^{\zeta^0} \right) \\
          &= \frac{1}{\varepsilon} \int_s^t b_r(\zeta^0+\varepsilon l,X^{\zeta^0+\varepsilon l}_r) - b_r(\zeta^0,X_r^{\zeta^0}) dr  + \frac{1}{\varepsilon} \int_s^t \sigma_r(\zeta^0+\varepsilon l,X^{\zeta^0+\varepsilon l}_r) - \sigma_r(\zeta^0,X_r^{\zeta^0}) dB_r + \\
          & \quad + \frac{1}{\varepsilon} \left( \beta_s(\zeta^0+\varepsilon l,X^{\zeta^0+\varepsilon l}_s) - \beta_s(\zeta^0,X_s^{\zeta^0}) \right)\delta W_{s,t} + \frac{1}{\varepsilon} \Big(D_x\beta_s(\zeta^0+\varepsilon l,X^{\zeta^0+\varepsilon l}_s)\beta_s(\zeta^0+\varepsilon l,X^{\zeta^0+\varepsilon l}_s) + \\
          & \quad - D_x\beta_s(\zeta^0,X_s^{\zeta^0})\beta_s(\zeta^0,X_s^{\zeta^0}) + \beta_s'(\zeta^0+\varepsilon l,X^{\zeta^0+\varepsilon l}_s) - \beta_s'(\zeta^0,X_s^{\zeta^0}) \Big)\mathbb{W}_{s,t} + Y^{\varepsilon,\natural}_{s,t},
      \end{aligned}
  \end{equation*} 
  where $Y^{\varepsilon,\natural}_{s,t} = \frac{1}{\varepsilon}(X^{\zeta^0+\varepsilon l,\natural}_{s,t}-X^{\zeta^0,\natural}_{s,t})$ is such that $\|Y^{\varepsilon,\natural}_{s,t}\|_p \lesssim |t-s|^{\alpha+\delta}$ and $\|\E_s(Y^{\varepsilon,\natural}_{s,t})\|_p \lesssim |t-s|^{\alpha+\delta+\delta'}$ for any $p \in [2,\infty)$. 
  Every term in the right-hand side of the previous identity can be re-written due to the differentiability assumptions on the coefficients, by applying the mean value theorem; by doing so the linear structure of the equation is made visible. 
  For simplicity, for $\theta \in [0,1]$ we denote by $X^\varepsilon(\theta) = (1-\theta) X^{\zeta^0} + \theta X^{\zeta^0 +\varepsilon l}$. For instance, 
  \begin{equation*}
      \begin{aligned}
          &\frac{1}{\varepsilon} \int_s^t b_r(\zeta^0+\varepsilon l,X^{\zeta^0+\varepsilon l}_r) - b_r(\zeta^0,X_r^{\zeta^0}) dr = \\
          &= \int_s^t \int_0^1 D_{\zeta} b_r (\zeta^0 + \theta\varepsilon l, X^\varepsilon_r (\theta))l d\theta dr + \int_s^t \left(\int_0^1 D_x b_r (\zeta^0 + \theta \varepsilon l, X^\varepsilon_r(\theta)) d\theta \right) Y^{n}_r dr =: \int_s^t \mu_r^\varepsilon dr + \int_s^t G_r^\varepsilon Y_r^\varepsilon dr.
      \end{aligned}
  \end{equation*}
  Applying \cref{lemma:stochasticlinearvectorfields_Riemannintegral} to the $\mathscr{L}(\R^{d_1+d})$-continuous process $\{(\zeta^0+\varepsilon l,X^{\zeta^0 + \varepsilon l})\}_{\varepsilon \ne 0}$, we have that $G^\varepsilon$ is a stochastic linear vector field, and it approaches $G := D_xb (\zeta^0,X^{\zeta^0})$ in $\mathscr{L}\mathcal{L}(\R^d;\R^d)$, as $\varepsilon \to 0$. 
  Moreover, it trivially follows that there exists $\mathscr{L}(\R^d)\text{-}\lim_{\varepsilon\to 0}\mu^\varepsilon = \mu$, where $\mu := D_\zeta b(\zeta^0,X^{\zeta^0})l$. Similarly,
   \begin{multline*}
       \frac{1}{\varepsilon} \int_s^t \sigma_r(\zeta^0+\varepsilon l,X^{\zeta^0+\varepsilon l}_r) - \sigma_r(\zeta^0,X_r^{\zeta^0}) dB_r = \int_s^t \int_0^1 D_{\zeta} \sigma_r (\zeta^0 + \theta\varepsilon l, X^\varepsilon_r (\theta))l d\theta dB_r + \\ + \int_s^t \left(\int_0^1 D_x \sigma_r (\zeta^0 + \theta \varepsilon l, X^\varepsilon_r(\theta)) d\theta \right) Y^\varepsilon_r dB_r =: \int_s^t \nu_r^\varepsilon dB_r + \int_s^t S_r^\varepsilon Y_r^\varepsilon dB_r,
   \end{multline*}
   where $S^\varepsilon$ is a random bounded linear vector field approaching $S := D_x\sigma (\zeta^0,X^{\zeta^0})$ in $\mathscr{L}\mathcal{L}(\R^d;\R^{d\times m})$ as $\varepsilon \to 0$, and $\mathscr{L}(\R^{d \times m}) \text{-} \lim_{\varepsilon \to 0}\nu^\varepsilon =\nu$, with $\nu:=D_\zeta \sigma(\zeta^0,X^{\zeta^0})l$. The remaining terms can be written as
   \begin{align*}
       \frac{1}{\varepsilon} \left(\beta_s(\zeta^0+\varepsilon l,X^{\zeta^0+\varepsilon l}_s) - \beta_s(\zeta^0,X_s^{\zeta^0}) \right) &= \int_0^1 D_\zeta \beta_s(\zeta^0+\theta\varepsilon l,X^\varepsilon_r(\theta)) l d\theta + \left(\int_0^1 D_x\beta_s(\zeta^0+\theta\varepsilon l,X^\varepsilon_r(\theta)) d\theta \right) Y_s^\varepsilon \\
       &=: \phi^\varepsilon_s + f_s^\varepsilon Y_s^\varepsilon, \\        
       \frac{1}{\varepsilon} \left(\beta'_s(\zeta^0+\varepsilon l,X^{\zeta^0+\varepsilon l}_s) - \beta'_s(\zeta^0,X_s^{\zeta^0})\right) &= \int_0^1 D_\zeta \beta'_s(\zeta^0+\theta\varepsilon l,X^\varepsilon_s(\theta)) l d\theta + \left(\int_0^1 D_x\beta'_s(\zeta^0+\theta\varepsilon l,X^\varepsilon_s(\theta)) d\theta \right) Y_s^\varepsilon \\
       &=: (\phi^\varepsilon_2)'_s + (f^\varepsilon_2)'_s Y_s^\varepsilon
    \end{align*}
    and, applying the mean value theorem to the function $(x,y,z)\mapsto D_y\beta(x,y)z$,  \begin{equation*}
        \begin{aligned}
            &\frac{1}{\varepsilon} \left(D_x\beta_s(\zeta^0+\varepsilon l,X^{\zeta^0+\varepsilon l}_s)\beta_s(\zeta^0+\varepsilon l,X^{\zeta^0+\varepsilon l}_s) - D_x\beta_s(\zeta^0,X_s^{\zeta^0})\beta_s(\zeta^0,X_s^{\zeta^0})\right) = \\
            &= \int_0^1 D^2_{\zeta x} \beta_s(\zeta^0+\theta\varepsilon l,X_s^\varepsilon(\theta) ((1-\theta)\beta_s(\zeta^0,X_s^{\zeta^0}) + \theta \beta_s(\zeta^0+\varepsilon l,X_s^{\zeta^0+\varepsilon l}),l) d\theta + \\
            & \quad + \left(\int_0^1 D^2_{xx} \beta_s(\zeta^0+\theta\varepsilon l,X_s^\varepsilon(\theta) ((1-\theta)\beta_s(\zeta^0,X_s^{\zeta^0}) + \theta \beta_s(\zeta^0+\varepsilon l,X_s^{\zeta^0+\varepsilon l}),\cdot) d\theta \right) Y_s^\varepsilon + \\
            & \quad + \frac{1}{\varepsilon} \int_0^1 D_x\beta_s(\zeta^0+\theta\varepsilon l,X_s^\varepsilon(\theta))d\theta) (\beta_s(\zeta^0+\varepsilon l,X_s^{\zeta^0+\varepsilon l})-\beta_s(\zeta^0,X_s^{\zeta^0})) d\theta \\
            &=: (\phi_1^\varepsilon)'_s + (f_1^\varepsilon)'_s Y_s^\varepsilon + f_s^\varepsilon(\phi^\varepsilon_s + f_s^\varepsilon Y_s^\varepsilon).
        \end{aligned}
    \end{equation*}
    Notice that, as a general fact, we can define $\bar{X}^\zeta := (\zeta, X^\zeta) \in U \times \R^d$ and we get that $\delta \bar{X}^\zeta = (0,\delta X^\zeta)$.
    This implies that, from the point of view of stochastic sewing, we can identify $\bar{X}^\zeta$ as a $\R^d$-valued stochastic controlled rough path with $(\bar{X}^\zeta)' = (0,\beta(\zeta,X^\zeta))$. 
    With a little abuse of notation, we keep writing $((\zeta,X^\zeta),\beta(\zeta,X^\zeta))$ instead of $((\zeta,X^\zeta),(0,\beta(\zeta,X^\zeta)))$.
    Applying \cref{lemma:stochasticcontrolledlinearvectorfields_Riemannintegral} to the $\mathbf{D}_W^{\alpha,\delta}\mathscr{L}$-continuous family of stochastic controlled rough path $\{(\zeta^0 + \varepsilon l, X^{\zeta^0 + \varepsilon l}), \beta(\zeta^0 + \varepsilon l, X^{\zeta^0 + \varepsilon l})\}_{\varepsilon \ne 0}$, it follows that \begin{equation*}
        (f^\varepsilon, (f^\varepsilon)' := (f^\varepsilon_1)'+(f^\varepsilon_2)') \in \bigcap_{p \in [2,\infty)} \mathbf{D}_W^{2\delta}L_{p,\infty}\mathcal{L}(\R^d;\R^{d\times n}),
    \end{equation*} for any $n \in \mathbb{N}$, and it approaches $(f,f'):= (D_x\beta(\zeta^0,X^{\zeta^0}),D^2_{xx}\beta(\zeta^0,X^{\zeta^0})\beta(\zeta^0,X^{\zeta^0})+D_x\beta'(\zeta^0,X^{\zeta^0}))$ in $\mathbf{D}_W^{\delta,\lambda}\mathscr{L}\mathcal{L}$ as $\varepsilon \to 0$, for any $\lambda \in [0,\delta)$.
    In particular, being $\alpha + 2\delta > 1$, we can choose such $\lambda$ so that it satisfies $\alpha + \delta + \lambda > 1$.
    Moreover, it also follows that $(\phi^\varepsilon, (\phi^\varepsilon)' := (\phi_1^\varepsilon)' + (\phi^\varepsilon_2)') \in \mathbf{D}_W^{\delta,\delta'}\mathscr{L}(\R^{d \times n})$ and it tends to $(\phi,\phi')$ in $\mathbf{D}_W^{\delta,\lambda}\mathscr{L}$, where \begin{equation*}
        (\phi, \phi') := \left(D_\zeta \beta(\zeta^0,X^{\zeta^0})l, D^2_{x\zeta} \beta(\zeta^0,X^{\zeta^0}) (\beta(\zeta^0,X^{\zeta^0}),l) + D_\zeta \beta'(\zeta^0,X^{\zeta^0})l \right).
    \end{equation*}
    Therefore we have shown that $Y^\varepsilon$ solves \begin{equation*}
        Y^\varepsilon_t = Y^\varepsilon_0 + F^\varepsilon_t + \int_0^t G_r^\varepsilon Y_r^\varepsilon dr + \int_0^t S_r^\varepsilon Y_r^\varepsilon dB_r + \int_0^t (f^\varepsilon_r Y_r^\varepsilon,(f^\varepsilon)'_r Y_r^\varepsilon + f_r^\varepsilon (f_r^\varepsilon Y_r^\varepsilon + (F^\varepsilon)'_r) ) d\mathbf{W}_r
    \end{equation*}
    with $Y^\varepsilon_0 = \frac{1}{\varepsilon}(X_0^{\zeta^0 + \varepsilon l}-X_0^{\zeta^0})$ and $(F^\varepsilon,(F^\varepsilon)') = \left( \int_0^\cdot \mu_r^\varepsilon dr + \int_0^\cdot \nu_r^\varepsilon dB_r + \int_0^\cdot (\phi^\varepsilon_r, (\phi^\varepsilon)'_r) d\mathbf{W}_r,\phi^\varepsilon \right).$ 
    Let us now define, for any $\zeta \in U$, $Y^\zeta = (Y^\zeta_t)_{t \in [0, T]}$ to be the solution of the following (well-posed) linear SDE \begin{equation} \label{eq:equationfortheBLderivative}
      Y_t^\zeta = \mathscr{L}\text{-}\frac{\partial}{\partial l} X_0^\zeta + F^{\zeta}_t + \int_0^t G_r^\zeta Y_r^\zeta dr + \int_0^t S_r^\zeta Y_r^\zeta dB_r + \int_0^t (f_r^\zeta ,(f^\zeta)'_r) Y_r^\zeta d\mathbf{W}_r,
  \end{equation}
  with \begin{equation} \label{eq:coefficientsforBLderivative}
      \begin{aligned}
        G_r^\zeta  &= D_x b_r (\zeta, X_r^{\zeta}) \qquad \qquad  
       S_r^\zeta  = D_x \sigma_r (\zeta, X_r^{\zeta}) \\
       (f_r^\zeta,(f^\zeta)'_r) &= (D_x\beta_r(\zeta, X_r^{\zeta}), D^2_{xx}\beta_r(\zeta, X_r^{\zeta})\beta_r(\zeta, X_r^{\zeta}) + D_x\beta'_r(\zeta, X_r^{\zeta})) \\
       (F^{\zeta},(F^{\zeta})')  &= \left( \int_0^\cdot \mu^\zeta_r dr + \int_0^\cdot \nu^\zeta_r dB_t + \int_0^\cdot (\phi^\zeta_r, (\phi^\zeta)'_r) d\mathbf{W}_r, \phi^\zeta \right) \\
       \mu^\zeta_r &= D_{\zeta} b_r (\zeta, X_r^{\zeta}) l \qquad \qquad
       \nu^\zeta_r = D_{\zeta} \sigma_r (\zeta, X_r^{\zeta})  l \\
       (\phi^\zeta_r, (\phi^\zeta)'_r) &= (D_\zeta\beta_r(\zeta,X_r^{\zeta})l,D^2_{x\zeta}\beta_s(\zeta,X_r^{\zeta})(\beta_s(\zeta,X_s^{\zeta}),l) + D_\zeta \beta'_s(\zeta,X_s^{\zeta})l).
  \end{aligned}
  \end{equation} 
  By construction, we have that \begin{equation*}
      \sup_{\zeta \in U} \left(\|G^\zeta\|_\infty + \|S^\zeta\|_\infty + \llbracket (f^\zeta,(f^\zeta)') \rrbracket_{\mathcal{L};\infty} + \llbracket (f^\zeta,(f^\zeta)') \rrbracket_{W;\delta,\delta';p,\infty;\mathcal{L}} \right) <+\infty.
  \end{equation*} 
  The previous considerations and \cref{prop:continuity_linearRSDEs} allow us to conclude that 
  \begin{equation*}
         (Y^{\zeta^0},f^{\zeta^0}Y^{\zeta^0} + (F^{\zeta^0})') = \mathbf{D}_W^{\alpha,\delta}\mathscr{L}(\R^d) \text{-} \lim_{\varepsilon \to 0} (Y^\varepsilon,f^\varepsilon Y^\varepsilon + (F^\varepsilon)').
        \end{equation*} 
  In other words, due to the arbitrariety of $\zeta^0$, we have that $\{(Y^{\zeta},f^\zeta Y^\zeta + (F^{\zeta})')\}_{\zeta \in U}$ is the (family of) $\mathbf{D}_W^{\alpha,\delta}\mathscr{L}$-derivatives of $\{(X^\zeta,\beta(\zeta,X^\zeta)\}_{\zeta \in U}$ in the direction $l$. \\  
  Assume now that all the assumptions about the continuity of the derivatives hold. Being $\{(\zeta,X^\zeta)\}_{\zeta \in U}$ $\mathbf{D}_W^{\alpha,\delta}\mathscr{L}$-continuous by contruction, we apply again \cref{prop:continuity_linearRSDEs} to have that $\{(Y^\zeta,f^\zeta Y^\zeta + (F^\zeta)')\}_{\zeta \in U}$ is $\mathbf{D}_W^{\alpha,\delta}\mathscr{L}(\R^d)$-continuous. 
  Indeed, from \cref{prop:stochasticlinearvectorfields_properties} it follows that $\{G^\zeta\}_{\zeta \in U}, \{S^\zeta\}_{\zeta \in U}$ are $\mathscr{L}\mathcal{L}$-continuous, and $\{\mu^\zeta\}_{\zeta \in U},\{\nu^\zeta\}_{\zeta \in U}$ are $\mathscr{L}$-continuous.
  Moreover, from \cref{prop:stochasticcontrolledlinearvectorfields_properties} it follows that $\{(f^\zeta,(f^\zeta)')\}_{\zeta \in U}$ is $\mathbf{D}_W^{\delta,\lambda}\mathscr{L}\mathcal{L}$-continuous, and $\{(\phi^\zeta,(\phi^\zeta)')\}_{\zeta \in U}$ is $\mathbf{D}_W^{\delta,\lambda}\mathscr{L}$-continuous. 
  This in particular implies - taking into account \cref{prop:continuityanddifferentiability_forcingterm} - that $\{(F^\zeta,(F^\zeta)')\}_{\zeta \in U}$ is $\mathbf{D}_W^{\alpha,\delta}\mathscr{L}$-continuous. \\
  
  Suppose now the theorem has been proved for $k=k_0$ and that, in addition, all the assumptions are satisfied for $k=k_0+1$. In this case $\kappa-1 \in (2+k_0,3+k_0]$.
  It is enough to show that, for any $l \in \mathbb{S}(\mathcal{U})$, the family $\{(Y^\zeta,f^\zeta Y^\zeta + (F^\zeta)')\}_{\zeta \in U}$ of stochastic controlled rough paths, defined according to \eqref{eq:equationfortheBLderivative}, is $k_0$-times ($\mathbf{D}_W^{\alpha,\delta}\mathscr{L}$-continuously) $\mathbf{D}_W^{\alpha,\delta}\mathscr{L}$-differentiable. 
  By inductive assumption, $\{\mathscr{L}\text{-}\frac{\partial}{\partial l} X_0^\zeta \}_{\zeta \in U}$ is $k_0$-times ($\mathscr{L}$-continuously) $\mathscr{L}$-differentiable.
  Moreover, $\{((\zeta,X^\zeta),\beta(\zeta,X^\zeta))\}_{\zeta \in U}$ is $k_0$-times ($\mathbf{D}_W^{\alpha,\delta}\mathscr{L}(\R^d)$-continuously) $\mathbf{D}_W^{\alpha,\delta}\mathscr{L}(\R^d)$-differentiable as well.  
  Due to the assumptions on $b,\sigma$, applying \cref{prop:stochasticlinearvectorfields_properties} we have that $\{G^\zeta\}_{\zeta \in U}$ and $\{S^\zeta\}_{\zeta \in U}$ are $k_0$-times ($\mathscr{L}$-continuously) $\mathscr{L}$-differentiable.
  Thanks to the assumptions on $(\beta,\beta')$, from \cref{prop:stochasticcontrolledlinearvectorfields_properties} we can deduce that $\{(f^\zeta,(f^\zeta)')\}_{\zeta \in U}$ is $k_0$-times ($\mathbf{D}_W^{\delta,\lambda}\mathscr{L}\mathcal{L}$-continuously) $\mathbf{D}_W^{\delta,\lambda}\mathscr{L}\mathcal{L}$-differentiable, for any $\lambda\in [0,\lambda)$.
  A similar thing also holds for the $k_0$-times ($\mathbf{D}_W^{\alpha,\delta}\mathscr{L}\mathcal{L}$-continuous) $\mathbf{D}_W^{\alpha,\delta}\mathscr{L}\mathcal{L}$-differentiability of  $\{(F^\zeta,(F^\zeta)')\}_{\zeta \in U}$, taking into account \cref{prop:continuityanddifferentiability_forcingterm}. 
  The recursive assertion is therefore proved, due to \cref{thm:differentiability_linearRSDEs}.
\end{proof}

  \subsection{Continuity and differentiability of composition and integration} \label{section:compositionandintegration_technicalresults}

In this section we collect a series of auxiliary result, which are extensively used in the previous sections.
Applications also appear in \cref{section:backwardroughPDEs}. 
At the start of \cref{section:parameterdependence} we introduced suitable notions of continuity and differentiability for a family $\{X^\zeta\}_{\zeta \in U}$ of $\R^d$-valued stochastic processes depending on a parameter. 
These notions were later refined to accommodate the framework of rough stochastic analysis developed in \cite{FHL21}.
Before applying these concepts to solutions of parameter-dependent rough SDEs, we verify their compatibility with key operations—namely composition and integration—by addressing the following questions: \begin{itemize}
    \item[-]  If $\{X^\zeta\}_{\zeta \in U}$ suitably continuous (differentiable) with respect to $\zeta$, is $\{g(X^\zeta)\}_{\zeta \in U}$ still continuous (differentiable) for a given function $g$? Here $g$ represents a (possibly random) vector field. 
    If $\{X^\zeta\}_{\zeta} \subseteq \mathscr{L}$ in the sense of \cref{def:stochasticprocessesinL}, this is addressed in \cref{prop:stochasticlinearvectorfields_properties} and suitably generalized in \cref{lemma:stochasticlinearvectorfields_Riemannintegral} and \cref{prop:stochasticcontinuousvectorfields_properties}.  
    If instead $X^\zeta$ is a stochastic controlled rough path (cf.\ \cref{def:L-stochasticcontrolledroughpaths}) the relevant result is \cref{prop:stochasticcontrolledlinearvectorfields_properties}, where $g$ is assumed to be a stochastic controlled vector field. 
    Continuity and differentiability are also shown to be preserved under multiplication; see\cref{lemma:compositionwithfunctionsofpolynomialgrowth}. 
    \item[-] If $\{X^\zeta\}_{\zeta \in U}$ is suitably continuous (differentiable) with respect to $\zeta$, do the Lebesgue integral $\int_0^\cdot X^\zeta_r dr$, the It\^o integral $\int_0^\cdot X^\zeta_r dB_r$ and the rough stochastic integral $\int_0^\cdot X^\zeta_r d\mathbf{W}_r$ inherit the same regularity? This is answered in \cref{prop:continuityanddifferentiability_forcingterm}
\end{itemize}
Let $V$ be a finite-dimensional vector space (think of  , $V=\R^{d_2}$ for some $d_2 \in \mathbb{N}_{\ge 2}$).

\begin{prop} \label{prop:stochasticlinearvectorfields_properties}
(see \cite[Lemma 1.7.6]{KRYLOV} and \cite[Theorem 1.9.2 (b)]{KRYLOV})
Let $g:\Omega \times [0,T] \times \R^d \to V$ be a progressively measurable map such that, for any $t \in [0,T]$ and for $\mathbb{P}$-almost every $\omega$, $\R^d \ni x \mapsto g_t(x):=g(\omega,t,x) \in V$ is continuous. 
Moreover, assume that there exist $K>0, m \ge 1$ such that \begin{equation*}
      \esssup_{\omega} \sup_{t \in [0, T]} |g_t(x)| \le K (1+|x|^m) \quad \text{for any $x \in \R^d$}.
\end{equation*}
    \begin{enumerate}
        \item Let $(\mathcal{U},\rho)$ be a metric space and let $U \subseteq \mathcal{U}$ be non-empty.
        Let $\{X^\zeta\}_{\zeta \in U}$ be a family of $\mathscr{L}(\R^d)$-continuous stochastic processes. Then $\{g(X^\zeta)\}_{\zeta \in U}$ is $\mathscr{L}(V)$-continuous
        \item Let $(\mathcal{U},|\cdot|)$ be a normed vector space and let $U \subseteq \mathcal{U}$ be a non-empty open subset. 
        Let $k \in \mathbb{N}_{\ge 1}$ and let $\{X^\zeta\}_{\zeta \in U}$ be a family of $k$-times ($\mathscr{L}(\R^d)$-continuously) $\mathscr{L}(\R^d)$-differentiable stochastic processes.
        Assume that, for $\mathbb{P}$-almost every $\omega$ and for any $t \in [0,T]$, $g(\omega,t,\cdot) \in C^k(\R^d;V)$, and all the derivatives up to order $k$ do not exceed $K(1 + |x|^m)$. 
        Then $\{g(X^\zeta)\}_{\zeta \in P}$ is $k$-times ($\mathscr{L}$-continuously) $\mathscr{L}$-differentiable. 
        In particular, for any $l \in \mathbb{S}(\mathcal{U})$, \begin{equation*}
          \mathscr{L}\text{-}\frac{\partial}{\partial l} g(X^\zeta) = D_xg(X^\zeta) \mathscr{L}\text{-}\frac{\partial}{\partial l} X^\zeta .
      \end{equation*}
    \end{enumerate}
\end{prop}

\begin{rmk} \label{rmk:identificationstochasticlinearvectorfields}
    The previous proposition provides a criterion to construct $\mathscr{L}\mathcal{L}$-continuous ($\mathscr{L}\mathcal{L}$-differentiable) stochastic linear vector fields, in the sense of \cref{def:continuityforstochasticlinearvectorfields} (\cref{def:differentiabilitystochasticlinearvectorfields}).  
    Indeed, if $g$ is a random bounded vector field from $\R^d$ to $V=\mathcal{L}(\R^{d_2};\bar{V})$, then \cref{prop:stochasticlinearvectorfields_properties} ensures that $\{g(X^\zeta)\}_{\zeta \in U}$ is $\mathscr{L}\mathcal{L}(\R^{d_2};\bar{V})$-continuous and/or $k$-times ($\mathscr{L}\mathcal{L}(\R^{d_2};\bar{V})$-continuously) $\mathscr{L}\mathcal{L}(\R^{d_2};\bar{V})$-differentiable.
\end{rmk} 

\begin{lemma} \label{lemma:stochasticlinearvectorfields_Riemannintegral}
    Let $g:\Omega \times [0,T] \times \R^d \to \mathcal{L}(\R^{d_2},V)$ be a random bounded continuous vector field. Let $(\mathcal{U},|\cdot|)$ be a normed vector space and let $U \subseteq \mathcal{U}$ be a non-empty open subset.
    Let $\{X^\zeta\}_{\zeta \in U}$ be $\mathscr{L}(\R^d)$-continuous. 
    Let $\zeta^0 \in U$ and $l \in \mathbb{S}(\mathcal{U})$ be fixed.
    For any $\varepsilon \in \R$ define \begin{align*}
        Q^\varepsilon &= \int_0^1 g(X^\varepsilon(\theta)) \, d\theta := \int_0^1 g((1-\theta) X^{\zeta^0} + \theta X^{\zeta^0 + \varepsilon l}) \, d\theta
    \end{align*}
    Then $\{Q^\varepsilon\}_{\varepsilon \in \R}$ is a family of stochastic linear vector fields from $\R^{d_2}$ to $V$, and it tends to $Q:= g(X^{\zeta^0})$ in $\mathscr{L}\mathcal{L}$ as $\varepsilon \to 0$.
\end{lemma}

\begin{proof}
    Notice that $Q^\varepsilon=\int_0^1 g(X^\varepsilon(\theta)) d\theta$ is well-defined as a stochastic linear vector field from $\R^{d_2}$ to $V$ since $(g_t(X^\varepsilon_t(\theta)))_{t \in [0,T]}$ is progressively measurable, uniformly bounded and its dependence in $\theta$ is continuous (cf.\ \cref{rmk:identificationstochasticlinearvectorfields}). 
    The same holds for $Q$. 
    For any $p \in [1,\infty)$ and for any $(\varepsilon^n)_n \subseteq \R$ converging to 0 as $n \to +\infty$,
    \begin{equation*}
        \E \left( \int_0^T |Q^{\varepsilon^n}_r - Q_r|^p \right) \le \int_0^1 \E \left( \int_0^T |g(X^{\varepsilon^n}(\theta)) - g(X^{\zeta^0})|^p \right) d\theta. 
    \end{equation*}
    By continuous mapping theorem, $g(X^{\varepsilon^n}(\theta))$ converges to $g(X^{\zeta^0})$ in measure $\mathbb{P} \otimes \lambda$, as $n \to +\infty$, and being $g$ uniformly bounded in $(\omega,t,x)$, a double application of Lebesgue's dominated convergence shows that $Q = \mathscr{L}\mathcal{L}\text{-}\lim_{\varepsilon \to 0} Q^\varepsilon$.
\end{proof}

\begin{prop} \label{prop:stochasticcontinuousvectorfields_properties}
    Let $(\mathcal{U},\rho)$ be a metric space and let $U \subseteq \mathcal{U}$ be non-empty. 
    Let $\{X^\zeta\}_{\zeta \in U}$ be a $\mathscr{L}(\R^d)$-continuous family of stochastic processes and let $\{g^\zeta\}_{\zeta \in U}$ be a $\mathscr{L}C^0_b(\R^d;V)$-continuous family of random bounded vector fields. 
    Then $\{g^\zeta(X^\zeta)\}_{\zeta \in U}$ is $\mathscr{L}(V)$-continuous. 
\end{prop}

\begin{proof}
    By construction, $(g^\zeta_t(X^\zeta_t))_{t \in [0,T]}$ is a progressively measurable stochastic process. 
    Let $\zeta^0 \in U$ be an accumulation point, and let $(\zeta^n)_{n \ge 1} \subseteq U \setminus \{\zeta^0\}$ be any arbitrary sequence such that $\rho(\zeta^n,\zeta^0) \to 0$ as $n \to +\infty$. 
    For any $n \ge 1$ and for any $p \in [1,\infty)$, it holds that \begin{equation*}
        \begin{aligned}
            &\E \left( \int_0^T |g^{\zeta^n}_r(X^{\zeta^n}_r) - g^{\zeta^0}_r(X^{\zeta^0}_r)|^p dr \right) \le \E \left( \int_0^T |g^{\zeta^n}_r(\cdot) - g^{\zeta^0}_r(\cdot)|_\infty^p dr \right) + \E \left( \int_0^T |g^{\zeta^0}_r(X^{\zeta^n}_r) - g^{\zeta^0}_r(X^{\zeta^0}_r)|^p dr \right). 
        \end{aligned}
    \end{equation*}
    The first term on the right-hand side tends to 0 as $n \to +\infty$ by assumption. The second term tends to 0, by applying \cref{prop:stochasticlinearvectorfields_properties}. 
\end{proof}

\begin{prop} \label{lemma:compositionwithfunctionsofpolynomialgrowth}
    Let $W \in C^\alpha([0,T];\R^n)$, with $\alpha \in (0,1]$, and let $h:\R^d \to V$ be three-times differentiable such that there exist $K>0$ and $m \ge 1$ with \begin{equation*}
        |D_xh(x)| + |D^2_{xx}h(x)| + |D^3_{xxx}h(x)| \le K(1+|x|^m) \quad \text{for any $x \in \R^d$}.
    \end{equation*}
    Let $\gamma,\gamma' \in [0,1]$.
    \begin{enumerate}
        \item Let $(X,X') \in \mathbf{D}_W^{\gamma,\gamma'}\mathscr{L}(\R^d)$, with $X_0 \in \mathscr{L}(\R^d)$. Then $(h(X),D_x h(X)X') \in \mathbf{D}_W^{\gamma,\gamma'}\mathscr{L}(V)$. 
        \item Let $(\mathcal{U},\rho)$ be a metric space and let $U \subseteq \mathcal{U}$ be non-empty. Let $\{(X^\zeta,(X^\zeta)')\}_{\zeta \in U}$ be $\mathbf{D}_W^{\gamma,\gamma'}\mathscr{L}(\R^d)$-continuous.
        Then $\{(h(X^\zeta),D_x h(X^\zeta)(X^\zeta)')\}_{\zeta \in U}$ is $\mathbf{D}_W^{\gamma,\gamma'}\mathscr{L}(V)$-continuous.
        \item Let $(\mathcal{U},|\cdot|)$ be a normed vector space and let $U \subseteq \mathcal{U}$ be a non-empty open subset. 
        Let $k \in \mathbb{N}_{\ge 1}$ and assume $h \in C^{2+k}_p(E,\bar{E})$. 
        Let $\{(X^\zeta,(X^\zeta)'\}_{\zeta \in U}$ be $k$-times ($\mathbf{D}_W^{\alpha,\bar{\alpha}}\mathscr{L}$-continuously) $\mathbf{D}_W^{\alpha,\bar{\alpha}}\mathscr{L}$-differentiable.  Then $\{(h(X^\zeta),D_x h(X^\zeta)(X^\zeta)')\}_{\zeta \in U}$ is $k$-times ($\mathbf{D}_W^{\alpha,\bar{\alpha}}\mathscr{L}$-continuously) $\mathbf{D}_W^{\alpha,\bar{\alpha}}\mathscr{L}$-differentiable, for any $\bar{\alpha} \in (0,\alpha)$.
    \end{enumerate}
\end{prop} 

\begin{proof}
    \textit{1. } It follows from \cite[Proposition 2.13]{BCN24}. \\
    
    \textit{2. } If $h \in C^3_b(\R^d;V)$, the result would follow from \cref{prop:compositionwithstochasticcontrolledvectorfields_compendium} because, for any $p \in [2,\infty)$, we would have that \begin{equation} \label{eq:stabilityestimateforh}
        \| h(X^{\zeta^n}), D_x h(X^{\zeta^n}) (X^{\zeta^n})' ; h(X^{\zeta^0}), D_x h(X^{\zeta^0}) (X^{\zeta^0})' \|_{W;\gamma, \gamma';p} \lesssim \| X^{\zeta^n}, (X^{\zeta^n})' ; X^{\zeta^0}, (X^{\zeta^0})' \|_{W;\gamma, \gamma';p}
    \end{equation}
    for any $\zeta^0 \in U$ accumulation point and any $(\zeta^n)_{n \ge 1} \subseteq U \setminus \{\zeta^0\}$ such that $\rho(\zeta^n,\zeta^0) \to 0$ as $n \to +\infty$, where the implicit constant is uniform in $n$. \cref{prop:compositionwithstochasticcontrolledvectorfields_compendium} can be easily generalized to vector fields of polynomial growth. One need to consider the following useful elementary estimate : given any $g \in C^2(V,\bar{V})$ such that $|D_xg(x)| + |D^2_{xx}g(x)| \le K(1+|x|)^m$ and given any arbitrary $a,b,c,d \in V$, an application of the mean value theorem shows that \begin{equation} \label{eq:elementaryestimate}
        |g(a)-g(b)-g(c)+g(d)| \le K (1+|a|+|b|+|c|+|d|)^m  \big((|a-c|+|b-d|)|c-d|+|a-b-c+d|\big).
    \end{equation}  
    For instance,  for any $(s,t) \in \Delta_{[0,T]}$ it follows that \begin{equation*}
        \|h(X^{\zeta^n}_t) - h(X^{\zeta^n}_s) - h(X^{\zeta^0}_t) + h(X^{\zeta^0}_s)\|_p \lesssim_{K,m,p,M(4mp)} (1+\|\delta X^{\zeta^0}\|_{\gamma;2p}) \|X^{\zeta^n},(X^{\zeta^n})';X^{\zeta^0},(X^{\zeta^0})'\|_{W;\gamma,\gamma';4p} |t-s|^\gamma
    \end{equation*}
    where $M=M(q) >0$ is any constant such that $\sup_{n \ge 0} \|(X^{\zeta^n},(X^{\zeta^n})')\|_{\mathbf{D}_W^{\gamma,\gamma'}L_q} \le M(q)$. 
    Estimates for the remaining quantities can be obtained by exploiting \eqref{eq:elementaryestimate} while retracing the proof of \cref{prop:compositionwithstochasticcontrolledvectorfields_compendium} for the deterministic controlled vector field $(h,0)$; it is thus possible to conclude that \eqref{eq:stabilityestimateforh} also holds when $h$ is of polynomial growth, with an implicit constant depending on $K,T,p,\gamma,\gamma'$ and $M(4mp)$. \\

    \textit{3. } We argue recursively on $k$. Let $\zeta^0 \in U$, $l \in \mathbb{S}(\mathcal{U})$ and $p \in [2,\infty)$ be arbitrary.
    For any $\varepsilon \in \R \setminus \{0\}$ we can write \begin{equation*}
        \begin{aligned}
            \left( \frac{1}{\varepsilon} (h(X^{\zeta^0 + \varepsilon l}) - h(X^{\zeta^0})) , \frac{1}{\varepsilon} \big(D_xh(X^{\zeta^0+\varepsilon l}) (X^{\zeta^0 + \varepsilon l})' - D_xh(X^{\zeta^0}) (X^{\zeta^0})'\big) \right) = (f^\varepsilon Y^\varepsilon, (f^\varepsilon)' Y^\varepsilon + f^\varepsilon (Y^\varepsilon)'),
        \end{aligned}
    \end{equation*}
    where, denoting by $(X^\varepsilon(\theta),X^\varepsilon(\theta)')$ as in \eqref{eq;definitionofX^varepsilon(theta)}, \begin{equation*} \begin{aligned}
        (f^\varepsilon,(f^\varepsilon)') &:= \left( \int_0^1 D_xh(X^\varepsilon(\theta)) d\theta, \int_0^1 D^2_{xx}h(X^\varepsilon(\theta)) (X^\varepsilon(\theta)',\cdot)  d\theta \right) \\
        (Y^\varepsilon,(Y^\varepsilon)') &:=  \left( \frac{1}{\varepsilon}(X^{\zeta^0 + \varepsilon l} - X^{\zeta^0}) , \frac{1}{\varepsilon} \big(X^{\zeta^0 + \varepsilon l})'- (X^{\zeta^0})' \big)  \right).
    \end{aligned} 
    \end{equation*}
    From the estimates in \cref{prop:stambilityundercompositionlinearvectorfields} and using a very similar argument as in \cref{lemma:stochasticcontrolledlinearvectorfields_Riemannintegral}, it is straightforward to deduce that $\{(f^\varepsilon Y^\varepsilon, (f^\varepsilon)' Y^\varepsilon + f^\varepsilon (Y^\varepsilon)')\}_{\varepsilon \in \R \setminus \{0\}}$ converges to \begin{equation*}
        \left(D_xh(X^{\zeta^0}) \mathbf{D}_W^{\gamma,\gamma'}\mathscr{L}\text{-}\frac{\partial}{\partial l} X^{\zeta^0}, D^2_{xx}h(X^{\zeta^0}) \left( \mathbf{D}_W^{\gamma,\gamma'}\mathscr{L}\text{-}\frac{\partial}{\partial l} (X^{\zeta^0})', \mathbf{D}_W^{\gamma,\gamma'}\mathscr{L}\text{-}\frac{\partial}{\partial l} X^{\zeta^0} \right) + D_xh(X^{\zeta^0}) \mathbf{D}_W^{\gamma,\gamma'}\mathscr{L}\text{-}\frac{\partial}{\partial l} (X^{\zeta^0})' \right).
    \end{equation*} 
    in $\mathbf{D}_W^{\gamma,\gamma'}\mathscr{L}(V)$ as $\varepsilon \to 0$. 
    For any $\zeta \in U$, we define \begin{equation*} \begin{aligned}
        Z^\zeta &:= D_xh(X^{\zeta}) \mathbf{D}_W^{\gamma,\gamma'}\mathscr{L}\text{-}\frac{\partial}{\partial l} X^{\zeta} \\
        (Z^\zeta)' &:= D^2_{xx}h(X^{\zeta}) \left( \mathbf{D}_W^{\gamma,\gamma'}\mathscr{L}\text{-}\frac{\partial}{\partial l} (X^{\zeta})', \mathbf{D}_W^{\gamma,\gamma'}\mathscr{L}\text{-}\frac{\partial}{\partial l} X^{\zeta} \right) + D_xh(X^{\zeta}) \mathbf{D}_W^{\gamma,\gamma'}\mathscr{L}\text{-}\frac{\partial}{\partial l} (X^{\zeta})'.
    \end{aligned}
    \end{equation*}
    It is straightformward to show that $\{(Z^\zeta,(Z^\zeta)')\}_{\zeta \in U}$ is $\mathbf{D}_W^{\gamma,\gamma'}\mathscr{L}$-continuous, and this proves the claim in the case $k=1$. \\
    Let us now assume that the assertion is proved for $k=k_0$, and that all the assumptions hold for $k=k_0+1$.
    By recursive hypothesis, it can be easily shown that $\{(Z^\zeta,(Z^\zeta)')\}_{\zeta \in U}$ is $k_0$-times ($\mathbf{D}_W^{\gamma,\gamma'}\mathscr{L}$-continuously) $\mathbf{D}_W^{\gamma,\gamma'}\mathscr{L}$-differentiable. 
    Indeed, we can always see $(Z^\zeta,(Z^\zeta)')$ as $(g(Y^\zeta),D_xg(Y^\zeta)(Y^\zeta)')$ where \begin{equation*}
        g(y) := D_xh(y^1) y^2 \qquad y=(y^1,y^2) \in \R^d \times \R^d
    \end{equation*} is still a function of polynomial growth, and $Y^\zeta = (X^{\zeta}, \mathbf{D}_W^{\gamma,\gamma'}\mathscr{L}\text{-}\frac{\partial}{\partial l} X^{\zeta})$ is the stochastic controlled rough path in $\mathbf{D}_W^{\gamma,\gamma'}\mathscr{L}(\R^d \times \R^d)$, whose Gubinelli derivative is defined through the (linear) action \begin{equation*}
        (Y^\zeta)'x := \left((X^{\zeta})'x, (\mathbf{D}_W^{\gamma,\gamma'}\mathscr{L}\text{-}\frac{\partial}{\partial l} X^{\zeta})x \right) \qquad x \in \R^n
    \end{equation*} 
\end{proof}

\begin{prop} \label{prop:stochasticcontrolledlinearvectorfields_properties}
Let by $\alpha,\gamma,\gamma',\delta \in [0,1]$, with $\alpha \ne 0$. Denote by $\lambda=\min\{\gamma,\delta\}$ and $\lambda' = \min\{\gamma,\gamma',\delta\}$.
    \begin{enumerate}
        \item \textcolor{black}{ Let $(\mathcal{U},\rho)$ be a metric space and let $o \in U \subseteq \mathcal{U}$.
        Let $\{{W}^\zeta\}_{\zeta \in U} \subseteq C^{\alpha}([0,T];\R^n)$ be a collection of $\alpha$-H\"older continuous paths in $\R^n$, and let $\{(X^\zeta,(X^\zeta)')\}_{\zeta \in U}$ be a family of $\{\mathbf{D}_{W^\zeta}^{\gamma,\gamma'}\mathscr{L}(\R^d)\}_\zeta$-continuous stochastic controlled rough paths such that \begin{equation*}
            \sup_{\zeta \in U: \ \rho(\zeta,o) \le R} \|(X^\zeta,(X^\zeta)')\|_{\mathbf{D}_W^{\gamma,\gamma'}L_{p,\infty}} <+\infty, \quad \text{for any $R>0$ and for any $p \in [2,\infty)$}.
        \end{equation*}
        Let $\kappa \in (2,3]$. For any $\zeta \in U$, let $(g^\zeta,(g^\zeta)') \in \mathbf{D}_{W^\zeta}^{2\delta}\mathscr{L}C_b^{\kappa}(\R^d;V)$ such that $\{(g^\zeta,(g^\zeta)')\}_{\zeta \in U}$ is a $\{\mathbf{D}_{W^\zeta}^{2\delta}\mathscr{L}C_b^{\kappa-1}(\R^d;V)\}_\zeta$-continuous family of stochastic controlled vector fields and  \begin{equation*}
            \sup_{\zeta \in U: \ \rho(\zeta,o) \le R} \left(\llbracket(g^\zeta,(g^\zeta)')\rrbracket_{\kappa-1;\infty} + \llbracket(g^\zeta,(g^\zeta)')\rrbracket_{W^\zeta;2\delta;p,\infty;\kappa-1} \right) <+\infty, 
        \end{equation*} 
        for any $R>0$ and for any $p \in [2,\infty)$.
        For any $\zeta \in U$, define \begin{equation} \label{eq:definitionoff^zeta}
            (Z^\zeta,(Z^\zeta)'):=(g^\zeta(X^\zeta),D_xg^\zeta(X^\zeta)(X^\zeta)' + (g^\zeta)'(X^\zeta)).
        \end{equation} 
        Then $\{(Z^\zeta,(Z^\zeta)')\}_{\zeta \in U}$ is $\{\mathbf{D}_{W^\zeta}^{\lambda,\eta'}\mathscr{L}(V)\}_\zeta$-continuous for any $\eta' \in [0, \lambda')$;
        while it is $\{\mathbf{D}_{W^\zeta}^{\lambda,\lambda'}\mathscr{L}(V)\}_\zeta$-continuous if $\kappa=3$. }
        \item Let $(\mathcal{U},|\cdot|)$ be a normed vector space and let $U \subseteq \mathcal{U}$ be a non-empty open subset. Let $W:[0,T] \to \R^n$ be an $\alpha$-H\"older continuous path. Let $k \in \mathbb{N}_{\ge 1}$. Assume that $\{(X^\zeta,(X^\zeta)')\}_{\zeta \in U}$ is $k$-times ($\mathbf{D}_W^{\gamma,\gamma'}\mathscr{L}$-continuously) $\mathbf{D}_W^{\gamma,\gamma'}\mathscr{L}$-differentiable with \begin{equation*}
            \sup_{|\zeta|\le R} \|(X^\zeta,(X^\zeta)')\|_{\mathbf{D}_W^{\gamma,\gamma'}L_{p,\infty}} <+\infty, \quad \text{for any $R>0$ and for any $p \in [2,\infty)$}.
        \end{equation*} 
        Assume \begin{equation*}
            (g,g') \in \bigcap_{p \in [2,\infty)} \mathbf{D}_W^{2\delta} L_{p,\infty} C^\kappa_b(\R^d;V),
        \end{equation*}  
        for some $\kappa \in (2+k,3+k]$.
        Then the family $\{(Z^\zeta,(Z^\zeta)')\}_{\zeta \in U}$ of stochastic controlled rough paths defined as in \eqref{eq:definitionoff^zeta} with $(g^\zeta,(g^\zeta)') = (g,g')$ is $(k-1)$-times ($\mathbf{D}_W^{\lambda,{\lambda}'}\mathscr{L}(V)$-continuously) $\mathbf{D}_W^{\lambda,{\lambda}'}\mathscr{L}(V)$-differentiable. Moreover, all its $(k-1)$-th $\mathbf{D}_W^{\lambda,{\lambda'}}\mathscr{L}$-derivatives are  ($\mathbf{D}_W^{\lambda,\eta'}\mathscr{L}(V)$-continuously) $\mathbf{D}_W^{\lambda,\eta'}\mathscr{L}(V)$-differentiable for any $\eta' \in [0,\lambda')$; such derivatives are ($\mathbf{D}_W^{\lambda,\lambda'}\mathscr{L}(V)$-continuously) $\mathbf{D}_W^{\lambda,\lambda'}\mathscr{L}(V)$-differentiable if $\kappa =3+k$.
        In particular, for any $\zeta \in U$ and for any $l \in \mathbb{S}(\mathcal{U})$, \begin{equation*} \begin{aligned}
            &\mathbf{D}_W^{\lambda,\eta'}\mathscr{L}\text{-}\frac{\partial}{\partial l} (Z^\zeta,(Z^\zeta)') = \\
            &= \left( D_xg(X^\zeta) \mathbf{D}_W^{\gamma,\gamma'}\mathscr{L}\text{-}\frac{\partial}{\partial l} X^\zeta , D^2_{xx}g(X^\zeta)\mathbf{D}_W^{\gamma,\gamma'}\mathscr{L}\text{-}\frac{\partial}{\partial l} X^\zeta\mathbf{D}_W^{\gamma,\gamma'}\mathscr{L}\text{-}\frac{\partial}{\partial l} X^\zeta + D_xg'(X^\zeta) \mathbf{D}_W^{\gamma,\gamma'}\mathscr{L}\text{-}\frac{\partial}{\partial l} (X^\zeta)' \right).
        \end{aligned} 
        \end{equation*} 
    \end{enumerate}
\end{prop}

\begin{proof}

    \textit{1. } The assertion follows from a straightforward application of \cref{prop:compositionwithstochasticcontrolledvectorfields_compendium}, together with an interpolation argument (see \cref{rmk:interpolation}) that uses the estimates contained in the first part of \cref{thm:stabilityforRSDEs_compendium} . \\

    \textit{2. }  We argue by recursion on $k$.
    Let us first assume $k=1$ and let us denote by $\eta=\min\{(\kappa-3)\gamma,\gamma',\delta\}$.
    Let $\zeta^0 \in U$ and $l \in \mathbb{S}(\mathcal{U})$ be arbitrary. 
    For any $\varepsilon \in \R \setminus \{0\}$, we define \begin{equation*} \begin{aligned}
        (Y^\varepsilon,(Y^\varepsilon)') &:= \left( \frac{1}{\varepsilon}(X^{\zeta^0 + \varepsilon l} - X^{\zeta^0}), \frac{1}{\varepsilon} \big((X^{\zeta^0 + \varepsilon l})' -(X^{\zeta^0})' \big) \right).
    \end{aligned}
    \end{equation*}
    By construction, $\{(Y^\varepsilon,(Y^\varepsilon)')\}_{\varepsilon \in \R \setminus \{0\}}$ tends to $(Y,Y') := (\mathbf{D}_W^{\gamma,\gamma'}\mathscr{L}\text{-}\frac{\partial}{\partial l} X^{\zeta^0}, \mathbf{D}_W^{\gamma,\gamma'}\mathscr{L}\text{-}\frac{\partial}{\partial l} (X^{\zeta^0})')$ in $\mathbf{D}_W^{\gamma,\gamma'}\mathscr{L}(\R^d)$ as $\varepsilon \to 0$.
    Denoting by $(X^\varepsilon(\theta),X^\varepsilon(\theta)') =  ((1-\theta) X^{\zeta^0} + \theta X^{\zeta^0 + \varepsilon l}, (1-\theta) (X^{\zeta^0})' + \theta (X^{\zeta^0 + \varepsilon l})'),$ by means of the mean value theorem we can write \begin{align*}
        \frac{1}{\varepsilon} (Z^{\zeta^0 + \varepsilon l} - Z^{\zeta^0}) &= \int_0^1 D_xg(X^\varepsilon(\theta)) d\theta \ Y^\varepsilon =: h^\varepsilon Y^\varepsilon \\
        \frac{1}{\varepsilon} ((Z^{\zeta^0 + \varepsilon l})' - (Z^{\zeta^0})' ) &= \left( \int_0^1 D^2_{xx}g(X^\varepsilon(\theta)) (\cdot, X^\varepsilon(\theta)') d\theta + \int_0^1 D_xg'(X^\varepsilon(\theta)) d\theta \right) Y^\varepsilon + \int_0^1 D_xg(X^\varepsilon(\theta)) d\theta \ (Y^\varepsilon)' \\
        & =: (h^\varepsilon)'Y^\varepsilon + h^\varepsilon (Y^\varepsilon)'
    \end{align*}
    From \cref{lemma:stochasticcontrolledlinearvectorfields_Riemannintegral} and being $(D_xg, D_xg') \in \bigcap_{p \in [2,\infty)} \mathbf{D}_W^{2\delta}L_{p,\infty}C_b^{\kappa-1}(\R^d;\mathcal{L}(\R^d;V))$ by construction, it holds that $\{(h^\varepsilon,(h^\varepsilon)')\}_{\varepsilon \in \R \setminus \{0\}} \subseteq \bigcap_{p \in [2,\infty)} \mathbf{D}_W^{\lambda,\lambda'}L_{p,\infty}\mathcal{L}(\R^d;V)$, and it tends to $$(h,h'):= (D_xg(X^{\zeta^0}), D^2_{xx}g (X^{\zeta^0}) (X^{\zeta^0})' + D_xg'(X^{\zeta^0}))$$ in $\mathbf{D}_W^{\lambda,\eta}\mathscr{L}\mathcal{L}$ as $\varepsilon \to 0$.
    Whenever $\kappa \ne 3$, by interpolation  the previous convergence also holds by replacing $\eta$ with any $\eta' \in [0,\lambda')$. 
    From \cite[Proposition 3.3]{BCN24}, we know that $\{(h^\varepsilon Y^\varepsilon, (h^\varepsilon)'Y^\varepsilon + h^\varepsilon (Y^\varepsilon)')\}_{\varepsilon \in \R \setminus \{0\}} \subseteq \mathbf{D}_W^{\lambda,\lambda'}L_{p,\infty}(V)$ for any $p \in [2,\infty)$ and from \cref{prop:stambilityundercompositionlinearvectorfields} it is straightforward to conclude that \begin{equation*}
        \begin{aligned}
            &\| h^{\varepsilon^n} Y^{\varepsilon^n}, (h^{\varepsilon^n})'Y^{\varepsilon^n} + h^{\varepsilon^n} (Y^{\varepsilon^n})' ; hY, h'Y + hY' \|_{W;\lambda,\eta;p} 
            \to 0
        \end{aligned}
    \end{equation*}
    as $n \to +\infty$. Such a convergence also holds by replacing $\eta$ with any $\eta' \in [0,\lambda')$, by interpolation and whenever $\kappa \ne 3$.
    In other words, denoting by $(Y^\zeta,(Y^\zeta)') := (\mathbf{D}_W^{\gamma,\gamma'}\mathscr{L}\text{-}\frac{\partial}{\partial l} X^\zeta,\mathbf{D}_W^{\gamma,\gamma'}\mathscr{L}\text{-}\frac{\partial}{\partial l} (X^\zeta)')$ and by $(h^\zeta,(h^\zeta)') := (D_xg(X^\zeta), D^2_{xx}g(X^\zeta)(X^\zeta)'+D_xg'(X^\zeta))$ we proved that, for any $\eta' \in [0,\lambda')$, \begin{equation} \label{eq:compactformforthefirstderivative}
        \{(h^\zeta Y^\zeta, (h^\zeta)'Y^\zeta + h^\zeta (Y^\zeta)')  \}_{\zeta \in U}
    \end{equation} 
    is the $\mathbf{D}_W^{\lambda,\eta'}\mathscr{L}$-derivative of $\{(Z^\zeta,(Z^\zeta)'\}_{\zeta \in U}$. 
    From \cref{lemma:compositionwithfunctionsofpolynomialgrowth} it easily follows that such a derivative is also $\mathbf{D}_W^{\lambda,\eta'}\mathscr{L}$-continuous, provided that $\{(Y^\zeta,(Y^\zeta)')\}_{\zeta \in U}$ is $\mathbf{D}_W^{\gamma,\eta'}\mathscr{L}$-continuous.
    Both of the previous differentiability and continuity results also hold for $\eta = \lambda'$, if $\kappa=3$. \\
    Let us now assume that the assertion is proved for $k=k_0$ and that all the assumptions hold for $k=k_0+1$.
    Notice that $\{(h^\zeta Y^\zeta, (h^\zeta)'Y^\zeta + h^\zeta (Y^\zeta)')  \}_{\zeta \in U}$ defined as in \eqref{eq:compactformforthefirstderivative} is necessarily the $\mathbf{D}_W^{\lambda,\lambda'}\mathscr{L}$-derivative of $\{(Z^\zeta,(Z^\zeta)'\}_{\zeta \in U}$ in the (arbitrary) direction $l \in \mathbb{S}(\mathcal{U})$.   
    By assumption, we have that $\kappa \in (2+(k_0+1),3+(k_0+1)]$ and $\{(Y^\zeta,(Y^\zeta)')\}_{\zeta \in U}$ is $k_0$-times ($\mathbf{D}_W^{\gamma,\gamma'}\mathscr{L}(\R^d)$-continuously) $\mathbf{D}_W^{\gamma,\gamma'}\mathscr{L}(\R^d)$-differentiable. 
    Moreover, being  and $(D_xg,D_xg') \in \bigcap_{p \in [2,\infty)} \mathbf{D}_W^{2\delta} L_{p,\infty} C^{\kappa-1}_b(\R^d;\mathcal{L}(\R^d;V))$, the recursive assumption leads to the fact that $\{D_xg(X^\zeta),D^2_{xx}g(X^\zeta) (X^\zeta)' + D_xg'(X^\zeta)\}_{\zeta \in U}$ is $(k_0-1)$-times ($\mathbf{D}_W^{\lambda,\lambda'}\mathscr{L}(V)$-continuously) $\mathbf{D}_W^{\lambda,\lambda'}\mathscr{L}(V)$-differentiable, and all its $(k_0-1)$-th order $\mathbf{D}_W^{\lambda,\lambda'}$-derivatives are ($\mathbf{D}_W^{\lambda,\eta'}\mathscr{L}(V)$-continuously) $\mathbf{D}_W^{\lambda,\eta'}\mathscr{L}(V)$-differentiable, for any $\eta' \in [0,\lambda')$;
    while the case $\eta = \lambda'$ can be considered if $\kappa-1=3+k_0$.
    By applying \cref{lemma:compositionwithfunctionsofpolynomialgrowth} to the function $h:\mathcal{L}(\R^d;V) \times \R^d \to V, \ h(x,y) := xy$, we obtain the desired conclusion. 
\end{proof}

\begin{rmk} \label{lemma:identificationlinearvectorfields-linearvaluedroughpaths}
    The previous proposition directly yields continuity and differentiability results for a class of stochastic controlled linear vector fields. 
    Indeed, for any given finite dimensional Banach space $\bar{V}$, applying \cref{prop:stochasticcontrolledlinearvectorfields_properties} in the case $V=\mathcal{L}(\R^{d_2};\bar{V})$ shows that $\{(g(X^\zeta),D_xg(X^\zeta)(X^\zeta)' + g'(X^\zeta))\}_{\zeta \in U}$ is $\mathbf{D}_W^{\lambda,\eta'}\mathscr{L}\mathcal{L}$-continuous.
    Similarly, we can deduce results about its $\mathbf{D}_W^{\lambda,\eta'}\mathscr{L}\mathcal{L}$-differentiability. 
\end{rmk}

\begin{lemma} \label{lemma:stochasticcontrolledlinearvectorfields_Riemannintegral}
    Let $(\mathcal{U},|\cdot|)$ be a normed vector space and let $U \subseteq \mathcal{U}$ be a non-empty open subset.
    Let $W \in C^\alpha([0,T];\R^n)$, with $\alpha \in (0,1]$, let $\gamma,\gamma',\delta \in [0,1]$ and denote by $\lambda= \min\{\gamma,\delta\}$, $\lambda'=\min\{\gamma,\gamma',\delta\}$. 
    Let \begin{equation*}
        (g,g') \in \bigcap_{p \in [2,\infty)} \mathbf{D}_W^{2\delta} L_{p,\infty} C_b^\kappa(\R^d;\mathcal{L}(\R^{d_2};V)),
    \end{equation*}
    \textcolor{black}{for some $\kappa \in (2,3]$,}
    and let $\{(X^\zeta,(X^\zeta)\}_{\zeta \in U}$ be a $\mathbf{D}_W^{\gamma,\gamma'}\mathscr{L}(\R^d)$-continuous family of progressively measurable stochastic controlled rough paths such that \begin{equation*}
        \sup_{|\zeta|\le R} \|(X^\zeta,(X^\zeta)')\|_{\mathbf{D}_W^{\gamma,\gamma'}L_{p,\infty}} < +\infty \qquad \text{for any $R >0$ and for any $p \in [2,\infty)$}. 
        \end{equation*}
        Let $\zeta^0 \in U$ and $l \in \mathbb{S}(\mathcal{U})$ be fixed. For any $\varepsilon \in \R$, define \begin{align}
            (X^\varepsilon(\theta), X^\varepsilon(\theta)') &:= (1-\theta)(X^{\zeta^0},(X^{\zeta^0})')+\theta (X^{\zeta^0 + \varepsilon l},(X^{\zeta^0 + \varepsilon l})') \qquad \theta \in [0,1] \label{eq;definitionofX^varepsilon(theta)}\\
            (f^\varepsilon, (f^\varepsilon)')  &:= \left( \int_0^1 g(X^\varepsilon(\theta)) d\theta,  \int_0^1 D_xg(X^\varepsilon(\theta))X^\varepsilon(\theta)' + g'(X^\varepsilon(\theta)) d\theta \right) \notag
        \end{align}
        Then $\{(f^\varepsilon,(f^\varepsilon)')\}_{\varepsilon \in \R} \subseteq \bigcap_{p \in [2,\infty)} \mathbf{D}_W^{\lambda,\lambda'} L_{p,\infty}\mathcal{L}(\R^{d_2};V)$,  and \begin{equation*}
            \sup_{|\varepsilon|\le R} \left( \llbracket (f^\varepsilon,(f^\varepsilon)') \rrbracket_{\mathcal{L};\infty} + \llbracket (f^\varepsilon,(f^\varepsilon)') \rrbracket_{W;\lambda,\lambda';p,\infty;\mathcal{L}} \right) <+\infty \quad \text{for any $R>0$ and for any $p \in [2,\infty)$}. 
        \end{equation*}
        Moreover, it tends to $(f,f') := (g(X^{\zeta^0}), D_xg(X^{\zeta^0}) (X^{\zeta^0})' + g'(X^{\zeta^0}))$ in $\mathbf{D}_W^{\lambda,\eta'}\mathscr{L}\mathcal{L}(\R^{d_2};V)$ as $\varepsilon \to 0$ for any $\eta' \in [0,\lambda')$, and the convergence in $\mathbf{D}_W^{\lambda,\lambda'}\mathscr{L}\mathcal{L}(\R^{d_2};V)$ holds if $\kappa=3$. 
\end{lemma}

\begin{proof}
    For simplicity, write $(f^\varepsilon,(f^\varepsilon)') = (\int_0^1 f^\varepsilon(\theta) d\theta, \int_0^1 f^\varepsilon(\theta)' d\theta)$.
    By construction, $(X^\varepsilon(\theta), X^\varepsilon(\theta)') \in \bigcap_{p \in [2,\infty)} \mathbf{D}_W^{\gamma,\gamma'} L_{p,\infty}(\R^d)$ for any $\varepsilon \in \R$, $\theta \in [0,1]$ and, for any $R >0$,  
    \begin{equation*} \begin{aligned}
        \sup_{|\varepsilon| \le R} \sup_{\theta \in [0,1]}  \|(X^\varepsilon(\theta), X^\varepsilon(\theta)')\|_{\mathbf{D}_W^{\gamma,\gamma'}L_{p,\infty}} 
        & \le 2 \sup_{|\zeta|\le |\zeta^0|+R} \|(X^{\zeta},(X^{\zeta})')\|_{\mathbf{D}_W^{\gamma,\gamma'}L_{p,\infty}} <+\infty.
    \end{aligned}
    \end{equation*}
    Notice that $\sup_{t \in [0,T]} \|X^\varepsilon_t(\theta)\|_\infty \le \llbracket (g,g') \rrbracket_{0;\infty} <+\infty$
    Hence, by \cref{prop:compositionwithstochasticcontrolledvectorfields_compendium} and by considering \cref{lemma:identificationlinearvectorfields-linearvaluedroughpaths} it holds that $(f^\varepsilon(\theta),f^\varepsilon(\theta)') \in \bigcap_{p \in [2,\infty)} \mathbf{D}_W^{\lambda,\lambda'}L_{p,\infty}\mathcal{L}(\R^{d_2};V)$. 
    Notice that, as a general fact, given any stochastic process $A(\theta)=(A_{s,t}(\theta))_{(s,t) \in \Delta_{[0,T]}}$ such that the dependence on $\theta$ is $\mathbb{P}$-a.s.\ continuous, \begin{equation*}
        \|\int_0^1 A_{s,t}(\theta) d\theta|\mathcal{F}_s\|_p \le \left( \int_0^1 \|A_{s,t}(\theta) d\theta|\mathcal{F}_s\|_p^p \right)^\frac{1}{p} \quad \mathbb{P}\text{-a.s}.
    \end{equation*}
    This allows us to conclude that, for any $R>0$ and for any $p \in [2,\infty)$, \begin{multline*}
        \sup_{|\varepsilon|\le R} \left( \llbracket (f^\varepsilon,(f^\varepsilon)') \rrbracket_{\mathcal{L};\infty} + \llbracket (f^\varepsilon,(f^\varepsilon)') \rrbracket_{W;\lambda,\lambda';p,\infty;\mathcal{L}}  \right) \\ \le \sup_{|\varepsilon| \le R} \sup_{\theta \in [0,1]}  \left( \llbracket (f^\varepsilon(\theta),f^\varepsilon(\theta)')
        \rrbracket_{\mathcal{L};\infty} + \llbracket (f^\varepsilon(\theta),f^\varepsilon(\theta)')
        \rrbracket_{W;\lambda,{\lambda'};p,\infty} \right) < +\infty.
    \end{multline*}
     From \cref{lemma:identificationlinearvectorfields-linearvaluedroughpaths} and \cref{prop:compositionwithstochasticcontrolledvectorfields_compendium} and being $\{(X^\zeta,(X^\zeta)')\}_{\zeta \in U}$ $\mathbf{D}_W^{\gamma,\gamma'}\mathscr{L}$-continuous, it follows that, for any $p \in [2,\infty)$ and for any arbitrary sequence $(\varepsilon^n)_n \subseteq \R \setminus \{0\}$ converging to zero, \begin{equation*}
         \llbracket f^{\varepsilon^n},(f^{\varepsilon^n})';f,f' \rrbracket_{\mathcal{L};p} + \llbracket f^{\varepsilon^n},(f^{\varepsilon^n})';f,f' \rrbracket_{W;\lambda,\eta;p} \to 0
     \end{equation*}
     as $n \to +\infty$, where $\eta = \min\{(\kappa-2)\gamma,\gamma',\delta\}$. 
    The conclusion follows by interpolation (see \cref{rmk:interpolation}), in a similar way to how it is done in the proof of \cref{prop:stochasticcontrolledlinearvectorfields_properties}.  
\end{proof}

\begin{prop} \label{prop:continuityanddifferentiability_forcingterm}
Let $\mathbf{W} \in \mathscr{C}^{\alpha}([0,T];\R^n)$ be a rough path, with $\alpha \in (\frac{1}{3},\frac{1}{2})$, and let $\gamma,\gamma' \in [0,1]$ be such that $\alpha + \gamma > \frac{1}{2}$ and $\alpha + \min \{\alpha, \gamma\} + \gamma' >1$.
    \begin{enumerate}
        \item Let $\mu \in \mathscr{L}(\R^d)$ and $\nu \in \mathscr{L}(\R^{d \times m})$ be two stochastic processes. Let $(\phi,\phi') \in \mathbf{D}_W^{\gamma,\gamma'}\mathscr{L}(\R^{d \times n})$ with $\sup_{t \in [0,T]} \|\phi_t\|_p < +\infty$ for any $p \in [2,\infty)$. 
        For any $t \in [0,T]$ define \begin{align*}
            F_t &:= \int_0^t \mu_r dr + \int_0^t \nu_r dB_r + \int_0^t (\phi_r,\phi'_r) d\mathbf{W}_r \\
            F_t' &:= \phi_t.
        \end{align*}
        Then $(F,F') \in \mathbf{D}_W^{\alpha,\gamma}\mathscr{L}(\R^d)$ and, for any $p > \frac{2}{1-2\alpha}$, \begin{equation*}
            \|(F,F')\|_{\mathbf{D}_W^{\alpha,\gamma}L_p} \lesssim_{\alpha,\gamma,\gamma',p,T,|\mathbf{W}|_\alpha} \E \left(\int_0^T|\mu_r|^p dr\right)^\frac{1}{p} + \E \left(\int_0^T|\nu_r|^p dr \right)^\frac{1}{p} + \sup_{t \in [0,T]} \|\phi_t\|_p + \|(\phi,\phi')\|_{\mathbf{D}_W^{\gamma,\gamma'}L_p}.
        \end{equation*}
        \item Let $(\mathcal{U},\rho)$ be a metric space and let $U \subseteq \mathcal{U}$ be non-empty, and let $\{\mathbf{W}^\zeta=(W^\zeta,\mathbb{W}^\zeta)\}_{\zeta \in U}$ be a $\mathscr{C}^\alpha$-continuous family of rough paths.
        Let $\{\mu^\zeta\}_{\zeta \in U}$ be a $\mathscr{L}(\R^d)$-continuous family of stochastic processes and let $\{\nu^\zeta\}_{\zeta \in U}$ be $\mathscr{L}(\R^{d \times m})$-continuous.
        Let $\{(\phi^\zeta,(\phi^\zeta)')\}_{\zeta \in U}$ be $\{\mathbf{D}_{W^\zeta}^{\gamma,\gamma'}\mathscr{L}(\R^{d\times n})\}_\zeta$-continuous and assume that, for any $\zeta \in U$ and for any $p \in [2,\infty)$, $\sup_{t \in [0,T]} \|\phi^\zeta_t\|_p <+\infty$.
        For any $\zeta \in U$, define \begin{equation*}
            (F^\zeta,(F^\zeta)') := \left(\int_0^\cdot \mu^\zeta_r dr + \int_0^\cdot \nu^\zeta_r dB_r + \int_0^\cdot (\phi^\zeta_r,(\phi^\zeta)'_r) d\mathbf{W}^\zeta_r, \phi^\zeta \right).
        \end{equation*} 
        Then $\{(F^\zeta,(F^\zeta)')\}_{\zeta \in U}$ is $\{\mathbf{D}_{W^\zeta}^{\alpha,\gamma}\mathscr{L}(\R^d)\}_\zeta$-continuous. 
        \item Let $k \in \mathbb{N}_{\ge 1}$ and assume $(\mathcal{U},|\cdot|)$ is a normed vector space.
        Assume that $\{\mu^\zeta\}_{\zeta \in U}$ and $\{\nu^\zeta\}_{\zeta \in U}$ are both $k$-times ($\mathscr{L}$-continuously) $\mathscr{L}$-differentiable. 
        Assume that $\{(\phi^\zeta,(\phi^\zeta)')\}_{\zeta \in U}$ is $k$-times ($\mathbf{D}_W^{\gamma,\gamma'}\mathscr{L}$-continuously) $\mathbf{D}_W^{\gamma,\gamma'}\mathscr{L}$-differentiable.
        Then $\{(F^\zeta,(F^\zeta)')\}_{\zeta \in U}$ is $k$-times ($\mathbf{D}_W^{\alpha,\gamma}\mathscr{L}$-continuously) $\mathbf{D}_W^{\alpha,\gamma}\mathscr{L}$-differentiable.
    In particular, for any $\zeta \in U$ and for any $l \in \mathbb{S}(\mathcal{U})$, \begin{equation*} \begin{aligned}
         &\left(\mathbf{D}_W^{\alpha,\gamma}\mathscr{L}\text{-}\frac{\partial}{\partial l}F^\zeta,\mathbf{D}_W^{\alpha,\gamma}\mathscr{L}\text{-}\frac{\partial}{\partial l}(F^\zeta)'\right) = \\
         &= \left(\int_0^\cdot \mathscr{L}\text{-}\frac{\partial}{\partial l} \mu^\zeta_r dr + \int_0^\cdot \mathscr{L}\text{-}\frac{\partial}{\partial l} \nu^\zeta_r dB_r + \int_0^\cdot \left( \mathbf{D}_W^{\gamma,\gamma'}\mathscr{L}\text{-}\frac{\partial}{\partial l} \phi^\zeta_r, \mathbf{D}_W^{\gamma,\gamma'}\mathscr{L}\text{-}\frac{\partial}{\partial l} (\phi^\zeta)'_r \right) d\mathbf{W}_r,\mathbf{D}_W^{\gamma,\gamma'}\mathscr{L}\text{-}\frac{\partial}{\partial l} \phi^\zeta \right).
    \end{aligned}
    \end{equation*}
    \end{enumerate}
\end{prop}

\begin{proof}
    \textit{1. } From the estimates on rough stochastic integrals in \cref{prop:roughstochastiintegral_compendium}, for any $p \in [2,\infty)$ it holds that that \begin{equation*}
        \left\| \left(\int_0^\cdot (\phi_r,\phi'_r) d\mathbf{W}_r, \phi \right) \right\|_{\mathbf{D}_W^{\alpha,\gamma}L_p} \lesssim_{\alpha,\gamma,\gamma',p,T,|\mathbf{W}|_\alpha} \sup_{t \in [0,T]} \|\phi_t\|_p + \|\phi,\phi'\|_{\mathbf{D}_W^{\gamma,\gamma'}L_p}. 
    \end{equation*}
    Moreover, a simple application of H\"older's and BDG inequalities shows that, for any $(s,t) \in \Delta_{[0,T]}$,  \begin{equation*}
        \begin{aligned}
            \left\|\int_s^t \mu_r dr \right\|_p &= \E \left( \left|\int_s^t \mu_r dr \right|^p \right)^\frac{1}{p} \le |t-s|^\frac{p-1}{p} \E \left(\int_0^T |\mu_r|^p dr \right)^\frac{1}{p}
        \end{aligned}
    \end{equation*} and \begin{equation*}
        \begin{aligned}
            \left\|\int_s^t \nu_r dB_r \right\|_p &= \E \left( \left|\int_s^t \nu_r dB_r \right|^p \right)^\frac{1}{p} \lesssim_p \E \left( \left(\int_s^t |\nu_r|^2 dr \right)^\frac{p}{2} \right)^\frac{1}{p} \le |t-s|^\frac{p-2}{2p} \E \left(\int_0^T |\nu_r|^p dr \right)^\frac{1}{p}.
        \end{aligned}
    \end{equation*}
    Under the assumption $p > \frac{2}{1-2\alpha}$, we have $\frac{p-1}{p} > \alpha + \gamma$ and $\frac{p-2}{2p} > \alpha$. Notice that this is possible only if $\alpha \ne \frac{1}{2}$.
    The previous three estimates are enough to prove the assertion, if one observes that $\E_s(R^F_{s,t}) = \E_s(\int_s^t \mu_r dr) + \E_s(R^\phi_{s,t})$.\\
    
    \textit{2. }  Let $\zeta^0 \in U$ be an accumulation point, and let $(\zeta^n)_{n \ge 1} \subseteq U \setminus \{\zeta^0\}$ such that $\rho(\zeta^n,\zeta^0) \to 0$ as $n \to +\infty$. 
    When taking the difference $F^{\zeta^n}-F^{\zeta^0}$, estimating the term $I^n := \int_0^\cdot (\phi^{\zeta^n}_r,(\phi^{\zeta^n})'_r) d\mathbf{W}^{\zeta^n}_r - \int_0^\cdot (\phi^{\zeta^0}_r,(\phi^{\zeta^0})'_r) d\mathbf{W}^{\zeta^0}_r$ requires some attention.
    Recalling the construction of the rough stochastic integral (see \cref{prop:roughstochastiintegral_compendium}), for any $p \in [2,\infty)$ and for any $(s,t) \in \Delta_{[0,T]}$ it holds that \begin{equation*}
        \|\delta I^n_{s,t} - \Xi^n_{s,t}\|_p \lesssim |t-s|^{\alpha+\gamma} \quad \text{and} \quad \|\E_s(\delta I^n_{s,t} - \Xi^n_{s,t})\|_p \lesssim |t-s|^{\alpha+\min\{\alpha,\gamma\}+\gamma'},
    \end{equation*} 
    where $\Xi^n_{s,t}:= \phi^{\zeta^n}_s \delta W^{\zeta^n}_{s,t} + (\phi^{\zeta^n})'_s \mathbb{W}^{\zeta^n}_{s,t} - \phi^{\zeta^0}_s \delta W^{\zeta^0}_{s,t} - (\phi^{\zeta^0})'_s \mathbb{W}^{\zeta^0}_{s,t}$.
    By the uniqueness statement in the stochastic sewing lemma (see e.g.\ \cite{LE20}), we deduce that $I^n$ is the unique processes characterized by $\Xi^n$ and we can use the estimates therein contained to have \begin{equation*} \begin{aligned}
        \|\delta I^n_{s,t}\|_p &\le \|\delta I^n_{s,t} - \Xi^n_{s,t}\|_p + \|\Xi^n_{s,t}\|_p \\
        &\lesssim_{\alpha,\gamma,\gamma',p,T,M} \left(\|\phi^{\zeta^n},(\phi^{\zeta^n})';\phi^{\zeta^0},(\phi^{\zeta^0})')\|_{W^{\zeta^n},W^{\zeta^0};\gamma,\gamma';p} + \rho_\alpha(\mathbf{W}^{\zeta^n},\mathbf{W}^{\zeta^0})\right) |t-s|^{\alpha+\gamma}
    \end{aligned}
    \end{equation*} and \begin{equation*} 
        \|\E_s(\delta I^n_{s,t})\|_p \lesssim_{\alpha,\gamma,\gamma',p,T,M} \left(\|\phi^{\zeta^n},(\phi^{\zeta^n})';\phi^{\zeta^0},(\phi^{\zeta^0})')\|_{W^{\zeta^n},W^{\zeta^0};\gamma,\gamma';p} + \rho_\alpha(\mathbf{W}^{\zeta^n},\mathbf{W}^{\zeta^0})\right) |t-s|^{\alpha+\min\{\alpha,\gamma\}+\gamma'},
    \end{equation*}
    where $M$ is any positive constant such that $\sup_{n \ge 0} ( \sup_{t \in [0,T]} \|\phi_t^{\zeta^n}\|_p + \|(\phi^{\zeta^n},(\phi^{\zeta^n})')\|_{\mathbf{D}_{W^{\zeta^n}}L_p} + |\mathbf{W}^{\zeta^n}|_\alpha ) \le M.$    
    Let now $p \in (\frac{1}{1-2\alpha},\infty)$ be arbitrary.
    Notice that, by construction, $F_0^\zeta=0$ for any $\zeta \in U$. 
    By linearity of the integrals and by the previous considerations, for any $n \ge 1$ we can write that \begin{equation*} 
        \begin{aligned}
            &\|F^{\zeta^n},(F^{\zeta^n})';F^{\zeta^0},(F^{\zeta^0})'\|_{W^{\zeta^n},W^{\zeta^0};\alpha,\gamma;p} = \left\| \int_\cdot^\cdot \mu^{\zeta^n}_r - \mu_r^{\zeta^0} dr \right\|_{\alpha+\gamma;p} + \left\| \int_\cdot^\cdot \nu^{\zeta^n}_r - \nu^{\zeta^0}_r dr \right\|_{\alpha;p} + \\
            & \quad + \left\| \int_0^\cdot (\phi^{\zeta^n}_r,(\phi^{\zeta^n})'_r) d\mathbf{W}^{\zeta^n}_r, \phi^{\zeta^n} ; \int_0^\cdot (\phi^{\zeta^0}_r,(\phi^{\zeta^0})'_r) d\mathbf{W}^{\zeta^0}_r, \phi^{\zeta^0}\right\|_{W^{\zeta^n},W^{\zeta^0};\alpha,\gamma;p} \\
            & \lesssim_{\alpha,\gamma,\gamma',p,T,|\mathbf{W}|_\alpha,M} \E \left(\int_0^T|\mu^{\zeta^n}_r-\mu^{\zeta^0}_r|^p dr\right)^\frac{1}{p} + \E \left(\int_0^T|\nu^{\zeta^n}_r-\nu^{\zeta^0}_r|^p dr \right)^\frac{1}{p} + \\
            & \qquad + \|\phi^{\zeta^n},(\phi^{\zeta^n})';\phi^{\zeta^0},(\phi^{\zeta^0})')\|_{W^{\zeta^n},W^{\zeta^0};\gamma,\gamma';p} + \rho_\alpha(\mathbf{W}^{\zeta^n},\mathbf{W}^{\zeta^0}),
        \end{aligned}
    \end{equation*}
    and the right-hand side tends to 0 as $n \to +\infty$, by assumption. \\
    
    \textit{3. }  We argue recursively on $k$. Let us first assume $k=1$. Let $\zeta^0 \in U$ and $l \in \mathbb{S}(\mathcal{U})$ be fixed. 
    For any $\varepsilon \in \R \setminus \{0\}$ we can write 
    \begin{align*}
            \Xi^\varepsilon :=& \frac{1}{\varepsilon} (F^{\zeta^0 + \varepsilon l} - F^{\zeta^0}) = \\
            =& \int_0^\cdot \frac{1}{\varepsilon}(\mu^{\zeta^0 + \varepsilon l}_r-\mu^{\zeta^0}_r) dr + \int_0^\cdot \frac{1}{\varepsilon} (\nu^{\zeta^0 + \varepsilon l}_r-\nu^{\zeta^0}_r) dr + \int_0^\cdot \frac{1}{\varepsilon} \left((\phi^{\zeta^0 + \varepsilon l},(\phi^{\zeta^0 + \varepsilon l})') - (\phi^{\zeta^0},(\phi^{\zeta^0})')\right) d\mathbf{W}_r \\
            (\Xi^\varepsilon)' :=& \frac{1}{\varepsilon} \big( (F^{\zeta^0 + \varepsilon l})'- (F^{\zeta^0})'\big) = \frac{1}{\varepsilon} (\phi^{\zeta^0 + \varepsilon l} - \phi^{\zeta^0})
    \end{align*}
    For any $\zeta \in U$, define \begin{equation*}
        (\Xi^\zeta,(\Xi^\zeta)') := \left(\int_0^\cdot \mathscr{L}\text{-}\frac{\partial}{\partial l} \mu^{\zeta}_r dr + \int_0^\cdot \mathscr{L}\text{-}\frac{\partial}{\partial l} \nu^\zeta_r dB_r + \int_0^\cdot \left( \mathbf{D}_W^{2\alpha'}\mathscr{L}\text{-}\frac{\partial}{\partial l} \phi^\zeta_r, \mathbf{D}_W^{2\alpha'}\mathscr{L}\text{-}\frac{\partial}{\partial l} (\phi^\zeta)'_r \right) d\mathbf{W}_r,\mathbf{D}_W^{2\alpha'}\mathscr{L}\text{-}\frac{\partial}{\partial l} \phi^\zeta \right). 
    \end{equation*}
    Let $(\varepsilon^n)_n \subseteq \R \setminus \{0\}$ be any arbitrary sequence converging to 0. 
    Then by reasoning as before, it follows that, for any $p \in (\frac{1}{1-2\alpha},\infty)$ and for any $n \in \mathbb{N}$, \begin{equation*}
        \begin{aligned}
            &\|\Xi^{\varepsilon^n},(\Xi^{\varepsilon^n})';\Xi^{\zeta^0},(\Xi^{\zeta^0})'\|_{W;\alpha,\gamma;p} \\
            &\lesssim_{\alpha,\gamma,\gamma',p,T,|\mathbf{W}|_\alpha} \E \left( \int_0^T \left| \frac{1}{\varepsilon}(\mu^{\zeta^0 + \varepsilon l}_r-\mu^{\zeta^0}_r) - \mathscr{L}\text{-}\frac{\partial}{\partial l} \mu^{\zeta^0}_r \right|^p dr \right)^\frac{1}{p} + \E \left( \int_0^T \left| \frac{1}{\varepsilon}(\nu^{\zeta^0 + \varepsilon l}_r-\nu^{\zeta^0}_r) - \mathscr{L}\text{-}\frac{\partial}{\partial l} \nu^{\zeta^0}_r \right|^p dr \right)^\frac{1}{p} + \\
            &+ \left\|\frac{1}{\varepsilon} (\phi^{\zeta^0 + \varepsilon l} - \phi^{\zeta^0}), \frac{1}{\varepsilon} ((\phi^{\zeta^0 + \varepsilon l})' - (\phi^{\zeta^0})' ) ; \mathbf{D}_W^{\gamma,\gamma'}\mathscr{L}\text{-}\frac{\partial}{\partial l} \phi^{\zeta^0}_r, \mathbf{D}_W^{\gamma,\gamma'}\mathscr{L}\text{-}\frac{\partial}{\partial l} (\phi^{\zeta^0})'_r \right\|_{W;\gamma,\gamma';p}.
        \end{aligned}
    \end{equation*}
    The right-hand side is going to 0 as $n \to +\infty$, by assumption. 
    Moreover, from what we proved in 1.\, it is easy to deduce that $\{(\Xi^\zeta,(\Xi^\zeta)')\}_{\zeta \in U}$ is $\mathbf{D}_W^{\alpha,\gamma}\mathscr{L}$-continuous, provided that the continuity assumptions on the integrands processes are satisfied. \\
    Let us now assume that the assertion is proved for $k=k_0$ and that all the assumptions hold for $k=k_0+1$. 
    By the recursive assumption, it is immediate to conclude that, for any $l \in \mathbb{S}(\mathcal{U})$, $\{(\Xi^\zeta,(\Xi^\zeta)')\}_{\zeta \in U}$ is ($\mathbf{D}_W^{\alpha,\gamma}\mathscr{L}$-continuously) $\mathbf{D}_W^{\alpha,\gamma}\mathscr{L}$-differentiable.
\end{proof}

\section{Well-posedness of rough PDEs} \label{section:backwardroughPDEs}
Let $T>0$ and let $\mathbf{W} = (W,\mathbb{W}) \in \mathscr{C}_g^{\alpha} ([0,T];\R^n)$ be a weakly geometric rough path, with $\alpha \in (\frac{1}{3},\frac{1}{2}]$.
Given $g : \R^d \to \mathbb{R}$, we aim at proving an existence-and-uniqueness result for the following terminal value problem:
\begin{equation} \label{eq:roughPDE}
    \begin{cases}
        -du_t = L_t u_t \ dt + \Gamma_t u_t \ d\mathbf{W}_t \qquad t \in [0,T] \\
        u_T = g
    \end{cases}.
\end{equation}
The solution is a map $u:[0,T]\times \R^d \to \R$.
The (respectively, second and first order)  time-dependent differential operators $L_t$ and $\Gamma_t = ((\Gamma_t)_1, \dots , (\Gamma_t)_n)$ are defined, given $f:\R^d \to \R$ and $x \in \R^d$, by
\begin{align*}
  L_t f  (x) & = \frac{1}{2}  \sum_{i,j=1}^d \sum_{a=1}^m \sigma_a^i(t,x)\sigma_a^j(t,x) \frac{\partial^2 f}{\partial x^j \partial x^i}(x) + \sum_{i = 1}^d{b^i}(t,x)  \frac{\partial f}{\partial x^i} (x) +c(t,x)f(x) \\
  (\Gamma_t)_\mu f (x) 
  &= \sum_{i=1}^d \beta_\mu^i(t,x) \frac{\partial f}{\partial
  x^i}(x) + \gamma_\mu(t,x)f(x) \qquad \mu=1,\dots,n  
\end{align*}
for some functions \begin{align*}
    b, \sigma_1, \dots, \sigma_m, \beta_1, \ldots, \beta_n : [0,T] \times \mathbb{R}^d \to \mathbb{R}^{d} \qquad \text{and} \qquad 
    c , \gamma_1,\dots,\gamma_n, : [0,T]\times\R^d \to \R .
\end{align*} 
We denote by $b^i$ the $i$-th component of $b$ (similarly for all the other maps).
From now on we use Einstein's repeated indices summation. Let \begin{equation*}
    \beta'_{11},\dots, \beta'_{nn}:[0,T] \times \R^d \to \R^d \qquad \text{and} \qquad \gamma'_{11},\dots,\gamma'_{nn}:[0,T] \times \R^d \to \R . 
\end{equation*}
The formal identities $\beta_\mu(t,x) - \beta_\mu(s,x) \approx \beta'_{\mu \nu}(s,x) \delta W^\nu_{s,t}$ and $\gamma_\mu(t,x) - \gamma_\mu(s,x) \approx \gamma'_{\mu \nu}(s,x) \delta W^\nu_{s,t}$
justify the definition of 
\begin{equation*}
    (\Gamma_t')_{\mu \nu} f (x) = \sum_{i=1}^d (\beta_{\mu \nu}')^i(t,x) \frac{\partial f}{\partial
  x^i}(x) + \gamma'_{\mu \nu}(t,x)f(x) \qquad \mu, \nu =1,\dots,n .
\end{equation*}

\begin{defn} \label{def:solutiontobacwardroughPDEs} 
  A regular solution to \eqref{eq:roughPDE} is a map $u \in C^{0,2} ([0,T] \times \R^d;\R)$ \footnote{For $\kappa \in \mathbb{N}_{\ge 1}$, we denote by $C^{0,\kappa}([0,T] \times \R^d: \R)$ ($C_b^{0,\kappa}([0,T] \times \R^d; \R)$) the space of functions $f=f(t,x)$ such that $f(t,\cdot) \in C^\kappa(\R^d;\R)$ for any $t \in [0,T]$, and $f,D_xf,\dots,D^\kappa_{x\dots x}f \in C^0([0,T] \times \R^d;\R) \ (C_b^0([0,T] \times \R^d;\R)) $} such that 
  \begin{itemize}
    \item[i)] $u (T, x) = g (x)$ for any $x \in \R^d$;
    \item[ii)] there exist $\alpha',\alpha'' \in [0,1]$ with $\alpha + \min\{\alpha,\alpha'\} + \alpha'' >1$ and a constant $C>0$ such that,
    for any $(s,t)\in \Delta_{[0,T]}$ and for any $\mu\in \{1,\dots,n\}$ \footnote{The reason for the presence of the minus sign in front of $\Gamma'$ is due to the backward-in-time nature of the equation. It is essentially due to the fact that in \cref{assumptions} we consider $\beta$ and $\gamma$ time-controlled by $\beta'$ and $\gamma'$, respectively, in a forward sense. A clear understanding of this phenomenon comes by looking at the proof of \cref{thm:existenceRPDE}. For an overview on backward rough integration and backward rough PDEs we refer to \cite[Section 5.4]{FRIZHAIRER} and \cite[Remark 12.4]{FRIZHAIRER}.} , \begin{equation*} \begin{aligned}
        &\sup_{x \in \R^d} |\big((\Gamma_s)_\mu (\Gamma_s)_\nu - (\Gamma_s')_{\mu \nu}\big) u_s(x) - \big((\Gamma_t)_\mu (\Gamma_t)_\nu - (\Gamma_t')_{\mu \nu} \big) u_t(x)| \le C|t-s|^{\alpha'} \qquad \nu=1,\dots,n \\
        &\sup_{x \in \R^d}|(\Gamma_t)_\mu u_t (x) - (\Gamma_s)_\mu u_s (x) + \big((\Gamma_t)_\mu (\Gamma_t)_\nu - (\Gamma'_t)_{\mu \nu} \big) u_t(x) \ \delta W_{s,t}^\nu| \le C |t-s|^{\alpha'+\alpha''}; 
    \end{aligned}
    \end{equation*}
    \item[iii)] for any $(s,t)\in \Delta_{[0,T]}$ and for any $x \in \R^d$, \begin{equation} \label{eq:Davieexpansion_solutionRPDE}
         u_s(x) - u_t(x) = \int_s^t L_r u_r (x) dr + (\Gamma_t)_\mu u_t (x) \delta W_{s,t}^\mu + \big((\Gamma_t)_\mu (\Gamma_t)_\nu - (\Gamma_t')_{\mu \nu}  \big) u_t(x)\mathbb{W}_{s,t}^{\mu \nu} + u^\natural_{s,t}(x)
    \end{equation}
    with $\sup_{x \in \R^d} |u^\natural_{s,t}(x)| = o(|t-s|)$ as $|t-s| \to 0$. 
  \end{itemize}
\end{defn}

It is convenient to explicitly write $(\Gamma_t)_\mu (\Gamma_t)_\nu f (x) = (\Gamma_t)_\mu ((\Gamma_t)_\nu f (\cdot) )(x) $ in coordinates as follows: \begin{equation} \label{eq:explicitGammaGamma}
    \begin{aligned}
        (\Gamma_t)_\mu (\Gamma_t)_\nu f (x) 
        &= \beta_\mu^j(t,x) \frac{\partial \beta_\nu^i}{\partial x^j}(t,x) \frac{\partial f}{\partial x^i}(x) + \beta_\mu^j(t,x) \beta_\nu^i(t,x) \frac{\partial^2 f}{\partial x^j \partial x^i}(x) + \beta^j_\mu(t,x) \frac{\partial \gamma_\nu}{\partial x^j}(t,x) f(x) + \\
        & \quad + \beta^j_\mu(t,x) \gamma_\nu (t,x) \frac{\partial f}{\partial x^j}(x) + \gamma_\mu (t,x) \beta^i_\nu(t,x)  \frac{\partial f}{\partial x^i}(x) + \gamma_\mu (t,x) \gamma_\nu(t,x) f(x). 
    \end{aligned}
\end{equation}

\begin{prop}
    Let $u:[0,T] \times \R^d \to \R$ be any function satisfying condition ii) of \cref{def:solutiontobacwardroughPDEs}. Then, for any $(t,x) \in [0,T] \in \R^d$ and along any sequence of partitions $\pi$ of $[t,T]$ whose mesh size tends to zero, there exists \begin{equation*}
        \int_t^T (\Gamma_r u_r(x), (\Gamma_r \Gamma_r - \Gamma_r') u_r(x)) d\mathbf{W}_r := \lim_{|\pi|\to 0} \sum_{[s,\tau] \in \pi} (\Gamma_\tau)_\mu u_\tau
        (x) \delta W^\mu_{s,\tau} + ((\Gamma_\tau)_\mu (\Gamma_\tau)_\nu - (\Gamma'_\tau)_{\mu \nu}) u_\tau(x) \mathbb{W}^{\mu \nu}_{s,\tau}.
    \end{equation*}
    In particular condition iii) of the previous definition can be equivalently replaced by \begin{itemize}
        \item[iii')] for any $t \in [0,T]$ and for any $x \in \R^d$ \begin{equation} \label{eq:RPDEinintegralform}
        u_t(x) = g(x) + \int_t^T L_r u_r (x) dr + \int_t^T (\Gamma_r u_r (x), (\Gamma_r \Gamma_r - \Gamma'_r) u_r(x)) d\mathbf{W}_r
    \end{equation} 
    \end{itemize}
    and the quantity $u^\natural_{s,t}(x)$ in fact satisfies $\sup_{x \in \R^d}|u^\natural_{s,t}(x)| \lesssim |t-s|^{\alpha + \min\{\alpha,\alpha'\}+\alpha''}$ for any $(s,t) \in \Delta_{[0,T]}$. 
\end{prop}

\begin{proof}
    The proof is essentially a consequence of the sewing lemma (see \cite[Lemma 4.2]{FRIZHAIRER}). Let $t \in [0, T]$ and, for any $x \in \R^d$ and for any $s, \tau \in [0, t]$, define
  \begin{equation*}
        A_{s,\tau}(x) := (\Gamma_\tau)_\mu u_\tau(x) \delta W_{s,\tau}^\mu + ((\Gamma_\tau)_\mu \Gamma_\nu + (\Gamma'_\tau)_{\mu\nu}) u_\tau(x) \mathbb{W}_{s,\tau}^{\mu,\nu} \in \R. 
    \end{equation*}
  Recalling Chen's relation for rough paths, for any $s, r, \tau \in [0, t]$ we have
  \begin{align*}
    A_{s, \tau}(x) - A_{s, r}(x) - A_{r,\tau}(x) & =
     ((\Gamma_\tau)_\mu u_\tau(x) - 
    (\Gamma_r)_\mu u_r(x))  \delta W^\mu_{r,\tau} +\\
    & \quad + ( (\Gamma_\tau)_\mu (\Gamma_\tau)_\nu u_\tau (x) - (\Gamma'_\tau)_{\nu \mu} u_\tau(x) ) (\mathbb{W}^{\mu \nu}_{s, r} + \delta W^\mu_{s,r} \delta W^\nu_{r,\tau}) + \\
    & \quad - ( (\Gamma_s)_\mu (\Gamma_s)_\nu u_s(x) - (\Gamma'_s)_{\nu \mu} u_s(x)) \mathbb{W}^{\mu \nu}_{s, r} .
  \end{align*}
  From condition (ii) of we deduce that $ \sup_{x \in \R^d}| A_{s, v}(x) - A_{s,u} (x) - A_{u, v}(x) | \le C  | \mathbf{W}
  |_{\alpha} |t - s|^{\alpha + \min\{\alpha,\alpha'\}+\alpha''}$.
\end{proof}

In what follows, we are making use of a sort of time-shift in the rough driver of the equation, as illustrated in the following
\begin{prop} \label{prop:shiftingaroughpath}
    Let $\mathbf{W}=(W,\mathbb{W})$ be an $\alpha$-H\"older (weakly geometric) rough path, with $\alpha \in (\frac{1}{3},\frac{1}{2}]$, and let $s \in  [0,T]$ be arbitrary. Let us define \begin{equation*}
        W^s:[0,T] \to \R^n, \quad u \mapsto W^s_u := \begin{cases}
            W_{s+u} & \text{if $u \le T-s$} \\
            W_T & \text{otherwise}
        \end{cases}
    \end{equation*} and \begin{equation*}
        \mathbb{W}^s:[0,T]^2 \to \R^n \otimes \R^n, \quad (u,v) \mapsto \mathbb{W}^s_{u,v} := \begin{cases}
            \mathbb{W}_{s+u,s+v} & \text{if $u,v \le T-s$} \\
            \mathbb{W}_{s+u,T} & \text{if $u \le T-s < v$} \\
            \mathbb{W}_{T,s+v} & \text{if $v \le T-s < u$} \\
            \mathbb{W}_{T,T}=0 & \text{otherwise}
        \end{cases}.
    \end{equation*}
    Then $\mathbf{W}^s := (W^s,\mathbb{W}^s)$ is an $\alpha$-H\"older (weakly geometric) rough path with $|\mathbf{W}^s|_\alpha \le |\mathbf{W}|_\alpha$. Moreover, for any $\alpha' \in (0,\alpha)$ the family $\{\mathbf{W}^s\}_{s \in [0,T]}$ is $\mathscr{C}^{\alpha'}$-continuous in the sense of \cref{def:convergenceforroughpaths}. 
\end{prop}

\begin{proof}
   It is straightforward to see that, for any $s \in [0,T]$, $|\delta W^s|_\alpha \le |W|_\alpha$, $|\mathbb{W}^s|_{2\alpha} \le |\mathbb{W}|_{2\alpha}$ and that the Chen's relation $\mathbb{W}^s_{u,v} - \mathbb{W}^s_{u,r} - \mathbb{W}^s_{r,v} = \delta W^s_{u,r} \otimes \delta W^s_{r,v}$ holds for any $u,r,v \in [0,T]$.
   Moreover, $\mathbf{W}^s$ is clearly weakly geometric if $\mathbf{W}$ is weakly geometric.
   Let us now consider any arbitrary sequence $(s^n)_n$ in $[0,T]$, converging to $s$ as $n \to +\infty$. 
   It is clear that $\sup_{n \in \mathbb{N}} |\mathbf{W}^{s^n}|_\alpha \le |\mathbf{W}|_\alpha$.  
    Therefore the last assertion of the proposition follows by interpolation (see \cref{prop:interpolation}), if we are able to show that \begin{equation*}
        \sup_{u,v \in [0,T]} \left(|\delta W^{s^n}_{u,v} - \delta W^s_{u,v}| + |\mathbb{W}^{s^n}_{u,v} - \mathbb{W}^s_{u,v}| \right) \to 0 \quad  \text{as $n \to +\infty$}.
    \end{equation*}
    Let $n \in \mathbb{N}$. Without loss of generality we can assume $T-s^n \le T-s$, otherwise we exchange $s^n$ with $s$ in all the following computations. Let $u,v \in [0,T]$ be arbitrary. We only consider the case $\max\{u,v\} \le T-s^n$. The proof of the other cases follows in analogous manner. Uniformly in $u,v$ it holds that \begin{equation*}
                |\delta W^{s^n}_{u,v} - \delta W^s_{u,v}| = |\delta W_{s^n+u,s^n+v} - \delta W_{s+u,s+v}| = |\delta W_{s+v,s^n+v} - \delta W_{s+u,s^n+u}| \le 2 |\delta W|_\alpha |s^n-s|^\alpha
            \end{equation*} and, using Chen's relation, \begin{equation*}
                \begin{aligned}
                    |\mathbb{W}^{s^n}_{u,v} - \mathbb{W}^s_{u,v}| &= |\mathbb{W}_{s^n+u,s^n+v} - \mathbb{W}_{s+u,s+v}| \le  |\mathbb{W}_{s^n+u,s^n+v} - \mathbb{W}_{s^n+u,s+v}| + |\mathbb{W}_{s^n+u,s+v} - \mathbb{W}_{s+u,s+v}| \\
                    &= |\mathbb{W}_{s+v,s^n+v} + \delta W_{s^n+u,s+v} \otimes \delta W_{s+v,s^n+v}| + |\mathbb{W}_{s^n+u,s+u} + \delta W_{s^n+u,s+u} \otimes \delta W_{s+v,s^n+v}| \\
                    &\le 2 |\mathbb{W}|_{2\alpha} |s^n-s|^{2\alpha} + 8T^\alpha |\delta W|_\alpha |s^n-s|^\alpha.  
                \end{aligned}
            \end{equation*}
\end{proof}
 
\begin{assumptions} \label{assumptions} 
    Let $\lambda \in (\frac{1}{\alpha},3]$, and consider $\delta,\eta \in [0,\alpha]$ and $\kappa,\theta \in [4,5]$ such that \footnote{At first reading, take $\lambda=3$, $\kappa=\theta=5$ and $\delta=\eta=\alpha$} \begin{equation*}
        \alpha + \min\{\delta,\eta\} > \frac{1}{2} \quad \text{and} \quad \alpha + \min\{\delta,\eta\} + \min\{(\lambda-2)\alpha,(\min\{\kappa,\theta\}-4)\alpha,\delta,\eta\}>1.
    \end{equation*}
   Let the following hold: 
    \begin{itemize}
        \item $g \in C^\lambda_b(\R^d;\R)$;
        \item $b \in C^0([0,T];C^2_b(\R^d;\R^d)) \cap \mathcal{B}([0,T];C^\lambda_b(\R^d;\R^d))$ ; 
        \item the map $\sigma=(\sigma_1,\dots,\sigma_m):[0,T] \times \R^d \to \mathcal{L}(\R^m;\R^d) \equiv \R^{d \times m}$ defined by $\sigma (t,x) y := \sigma_j(t,x) y^j, \ y \in \R^m$ belongs to $C^0([0,T];C^2_b(\R^d;\mathcal{L}(\R^m;\R^d))) \cap \mathcal{B}([0,T];C^\lambda_b(\R^d;\mathcal{L}(\R^m;\R^d))) $ ;  
        \item$c \in C^0([0,T];C^2_b(\R^d;\R)) \cap \mathcal{B}([0,T];C^\lambda_b(\R^d;\R))$;
        \item the maps $\beta=(\beta_1,\dots,\beta_n):[0,T] \times \R^d \to \mathcal{L}(\R^n;\R^d) \equiv \R^{d \times n}$ and $\beta':[0,T] \times \R^d \to \mathcal{L}(\R^n;\mathcal{L}(\R^n;\R^d))$ defined via $\beta'(t,x) y z := \beta'_{\mu \nu}(t,x) y^\nu z^\mu, \ y,z \in \R^n$ are such that $(\beta,\beta') \in \mathscr{D}_W^{2\delta} C_b^\kappa $ ;
        \item $(\gamma,\gamma') \in \mathscr{D}_W^{2\eta} C_b^\theta $ .
    \end{itemize}
  
\end{assumptions}

\begin{thm} \label{thm:mainresult_wellposednessroughPDEs}
   Under \cref{assumptions}. Let $B=(B_t)_{t \in [0,T]}$ be a Brownian motion on a complete filtered probability space $(\Omega,\mathcal{F},\{\mathcal{F}_t\}_{t \in [0,T]},\mathbb{P})$.
   For any $(s,x) \in [0,T] \times \R^d$, denote by $X^{s,x}$ the solution of the following rough SDE \begin{equation}\label{eq:RSDE_backwardroughPDE_main}
        X^{s,x}_t = x + \int_0^t b (s+r, X_r^{s,x}) dr + \int_0^t \sigma (s+r, X_r^{s,x}) dB_r + \int_0^t (\beta,\beta')(s+r,X^{s,x}_r) d\mathbf{W}^s_r,
    \end{equation} 
    where $(\beta,\beta')(s+r,X^{s,x}_r) := (\beta (s+r, X_r^{s,x}),D_x \beta (s+r, X_r^{s,x}) \beta (s+r, X_r^{s,x}) + \beta'(s+r,X^{s,x}_r))$ and $\mathbf{W}^s$ is defined as in \cref{prop:shiftingaroughpath}.     
    Define
    \begin{equation} \label{eq:definitionofu_mainresult} 
      u(s,x) := \E \left(g(X_{T-s}^{s,x}) e^{\int_0^{T-s} c(s+r,X^{s,x}_r) dr + \int_0^{T-s} (\gamma,\gamma')(s+r,X^{s,x}_r) d\mathbf{W}^s_r} \right).
    \end{equation}
    Then $$u \in C^{0,2}_b([0,T] \times \R^d ; \R) \cap \mathcal{B}([0,T];C^\lambda_b(\R^d;\R))  \cap C^{\alpha}([0,T];C^2_b(\R^d;\R))$$
    and it is the unique regular solution to $-du_t = L_t u_t + \Gamma_tu_t d\mathbf{W}_t$ in the sense of \cref{def:solutiontobacwardroughPDEs}  .
\end{thm}

\begin{proof}
    The space-time regularity of $u$ is established in \cref{prop:regularityofu} using the results results on parameter dependent rough SDEs from \cref{section:parameterdependence}.
    Existence is proved in \cref{thm:existenceRPDE}, with  \cref{prop:roughItoformula} and \cref{lemma:markovpropertyforRSDEs} as the main tools.  
    Uniqueness follows  from and \cref{thm:uniqueness}.  
\end{proof}

\begin{rmk}
    In our notation, the parameter $s$ in $X^{s,x}$ represents a time-shift in the coefficients and the rough driver, not the initial time. A more familiar formulation may be $$u(s,x) = \E\left(g(\bar{X}^{s,x}_T) e^{\int_s^T c_r(\bar{X}_r^{s,x}) dr + \int_s^T (\gamma,\gamma') (r,\bar{X}_r^{s,x}) d\mathbf{W}_r}\right),$$ 
    being $\bar{X}^{s,x}=(\bar{X}_t^{s,x})$ the solution of \begin{equation*}
        X_t = x + \int_s^t b(r,X_r) dr + \int_s^T \sigma(r,X_r) dB_r + \int_s^T (\beta,\beta')(r,X_r) d\mathbf{W}_r \qquad t \in [s,T].
    \end{equation*}
    We prefer to work with $X^{s,x}$ defined as in \eqref{eq:RSDE_backwardroughPDE_main}, since it treats $s$ as a parameter, allowing us to apply the parameter dependence theory we previously developed in this paper.
    The two approaches are consistent if one considers \cref{prop:independenceontheBrownianmotion} and performs a time-change in the integrals (as explained, for example, in the proof of \cref{prop:flowpropertyfroRSDEs}.2 below).
    Indeed, with the notation introduce in Section \ref{section:functionalsandmarkovproperty} it is possible to show that $Y=(Y_t := \bar{X}^{s,x}_{s+t})_{t \in [0,T-s]}$ belongs to $RSDE(s,x,B^s,\{\mathcal{F}^s_t\}_{t \in [0,T-s]})$,
    where $B^s_t:= B_{s+t} - B_s$ and $\mathcal{F}^s_t := \mathcal{F}_{s+t}$ (cf.\ \cref{prop:flowpropertyfroRSDEs}). 
\end{rmk}

\begin{rmk} \label{rmk:independenceontheBrownianmotion}
    As a consequence of \cref{prop:independenceontheBrownianmotion} below, the rough PDE solution $u$ given in \eqref{eq:definitionofu_mainresult} is uniquely determined by the functions $b,\sigma\sigma^T,\beta,\beta',c,\gamma,\gamma'$ and $g$. 
    As already pointed out in \cite[Remark 1.9.11]{KRYLOV}, the definition of $u(s,x)$ does not change if we take another filtered probability space and/or if we replace the Brownian motion with a $\bar{m}$-dimensional Brownian motion ($\bar{m} \ne m$) and take another coefficient $\bar{\sigma}:[0,T] \times \R^d \to \R^{d \times \bar{m}}$, provided that $\bar{\sigma}\bar{\sigma}^T \equiv \sigma \sigma^T$.  
\end{rmk}

\begin{prop} \label{prop:regularityofu}
    Let $u=u(s,x)$ be defined as in \eqref{eq:definitionofu_mainresult}. 
    Then under \cref{assumptions} \begin{equation*}
        u \in C_b^{0,2}([0,T] \times \R^d;\R) \cap \mathcal{B}([0,T], C^\lambda_b(\R^d;\R)) \cap C^\alpha([0,T];C^2_b(\R^d;\R)). 
    \end{equation*}
\end{prop}

\begin{proof} 
    Notice that, in general, equation \eqref{eq:RSDE_backwardroughPDE_main} only makes sense on the time interval $[0,T-s]$. 
    Redefining the coefficients by $b(t,x)=b(T,x)$, $\sigma(t,x)=\sigma(T,x)$ and $\beta(t,x)=\beta(T,x)$ if $t \ge T$, we can always assume it to be defined for any $t \in [0,T]$.
    We also extend $c(t,x) = c(T,x)$ and $\gamma(t,x) = \gamma(T,x)$ for $t \ge T$. 
    For any given $s \in [0,T]$, we define $b^s(t,x) := b(s+t,x)$ and a similar definition holds for $\sigma^s$, $c^s$, $\beta^s$ and $\gamma^s$. 
    We set $(\beta^s)' (t,x) := \beta'(s+t,x)$ and $(\gamma^s)' (t,x) := \gamma'(s+t,x)$. 
    We start by studying the continuity and differentiability properties of the family $\{X^{s,x}\}_{(s,x) \in [0,T] \times \R^d}$. 
    Under \cref{assumptions} it is straightforward to see that $\{b^s\}_{s \in [0,T]}$is a $\mathscr{L}C^0_b$-continuous family of deterministic bounded Lipschitz vector fields with $\|b^s \|_{Lip} \le \sup_{t \in [0,T]} |D_xb(t,\cdot)|_\infty$ uniformly in $s$. Moreover, $\{D_xb^s\}_{s \in [0,T]}$ and $\{D^2_{xx}b^s\}_{s \in [0,T]}$ are $\mathscr{L}C^0_b$-continuous. 
    The previous considerations apply if we replace $b^s$ by $\sigma^s$ and $c^s$. 
    On the other hand, it is also easy to see that, for any $s \in [0,T]$, $(\beta^s,(\beta^s)')$ is a (deterministic) controlled vector field in $\mathscr{D}_{W^s}^{2\delta}C_b^\kappa(\R^d;\R^{d \times n})$ with \begin{equation*}
        [(\beta^s,(\beta^s)')]_\kappa + [(\beta^s,(\beta^s)')]_{W^s;2 \delta ; \kappa} \le [(\beta,\beta')]_\kappa + [(\beta,\beta')]_{W;2 \delta ; \kappa} < +\infty . 
    \end{equation*}
    Analogously, one shows that $(\gamma^s,(\gamma^s)') \in \mathscr{D}_{W^s}^{2\eta}C_b^\theta(\R^d;\R^{n})$.
    \textcolor{black}{Moreover, by assumption and for any $(s^n)_{n \ge 1} \subseteq [0,T]$ with $|s^n - s| \to_{n \to +\infty} 0$,  
    \begin{equation*}
        \sup_{t \in [0,T]} |\beta^{s^n}(t,\cdot) - \beta^{s}(t,\cdot)|_{C^4_b} + |(\beta^{s^n})'(t,\cdot) - (\beta^{s})'(t,\cdot)|_{C^{3}_b} \to 0 \qquad \text{as $n \to +\infty$}.
    \end{equation*}
    By interpolation (see \cref{rmk:interpolation}) this implies that $[\beta^{s^n}-\beta^{s}]_{{\delta'};4} + [(\beta^{s^n})'-(\beta^{s})']_{{\delta'};3} + [R^{\beta^{s^n}}-R^{\beta^{s}}]_{{2\delta'};3} \to 0$ as $n \to +\infty$, for any $\delta' \in [0,\delta)$, and consequently that $\{(\beta^s,(\beta^s)')\}_{s \in [0,T]}$ is $\{\mathbf{D}_{W^s}^{2\delta'}\mathscr{L}C^4_b\}_s$-continuous. 
    Arguing in a very similar way, we deduce that $\{(\gamma^s,(\gamma^s)')\}_{s \in [0,T]}$ is $\{\mathbf{D}_{W^s}^{2\eta'}\mathscr{L}C^4_b\}_s$-continuous, for any $\eta' \in [0,\eta)$. }
    From \cref{prop:shiftingaroughpath} we know that $\{\mathbf{W}^s\}_{s \in [0,T]}$ is $\mathscr{C}^{\bar{\alpha}}$-continuous, for any $\bar{\alpha} \in (0,\alpha)$, and $\sup_{s \in [0,T]} |\mathbf{W}^s|_\alpha \le |\mathbf{W}|_\alpha$.
    Without loss of generality we can assume $\alpha \ne \frac{1}{2}$, let $\{\mathbf{W}^s\}_{s \in [0,T]}$ be $\mathscr{C}^{{\alpha}}$-continuous, and choose $\delta',\eta'$ so that $\alpha + \min\{\delta',\eta'\}> \frac{1}{2}$ and $\alpha + \min\{\delta',\eta'\} + \min\{(\lambda-2)\alpha, (\min\{\kappa,\theta\}-4)\alpha, \delta',\eta'\} >1 $.
    Indeed, notice that a great majority of the results concerning continuity and differentiability for (possibly linear) rough SDEs in \cref{section:parameterdependence} were stated for $\alpha$-rough paths with $\alpha \ne \frac{1}{2}$ (for instance, cf.\ \cref{thm:differentiability_RSDEs}). However, from the rest of the proof one can realize that such restrictions on the parameters do not cause any loss of generality in the statement of the proposition.

    From \cref{thm:wellposednessandaprioriestimatesforRSDEs} it follows that, for any $(s,x) \in [0,T] \times \R^d$, $X^{s,x}$ is well-defined as the unique $L_{p,\infty}$-integrable solution to \eqref{eq:RSDE_backwardroughPDE_main} for any $p \in [2,\infty)$. 
     From the 
     a priori estimate in \cref{thm:wellposednessandaprioriestimatesforRSDEs} we have \begin{equation} \label{eq:aprioriestimateforRSDEin(s,x)}
        \sup_{(s,x) \in [0,T] \times \R^d} \|(X^{s,x},\beta^s(X^{s,x}))\|_{\mathbf{D}_{W^s}^{\alpha,\delta}L_{p,\infty}} <+\infty \qquad \text{for any $p \in [2,\infty)$}
    \end{equation}
    and, from \cref{thm:stabilityforRSDEs_compendium}, \begin{equation} \label{eq:stabilityestimatesforRSDEinx}
        \|X^{s,x},\beta^s(X^{s,x});X^{s,y},\beta^s(X^{s,y})\|_{W^s;\alpha,\delta;p} \lesssim |x-y|,
    \end{equation} for any $x,y \in \R^d$, for any $s \in [0,T]$ and for any $p \in [2,\infty)$,
    where the implicit constant can be chosen as uniform in $s \in [0,T]$ and $x,y \in \R^d$. By \cref{thm:continuity_RSDEs}, $\{(X^{s,x},\beta^s(X^{s,x})\}_{(s,x) \in [0,T] \times \R^d}$ is $\{\mathbf{D}_{W^s}^{\alpha, \delta} \mathscr{L}\}_s$-continuous and, for any $s \in [0,T]$, $\{(X^{s,x},\beta^s(X^{s,x})\}_x$ is twice $\mathbf{D}_{W^s}^{\alpha, \delta} \mathscr{L}$-continuously $\mathbf{D}_{W^s}^{\alpha, \delta} \mathscr{L}$-differentiable by \cref{thm:differentiability_RSDEs}.

    Moreover, \cref{prop:stochasticcontinuousvectorfields_properties} shows that $\{c^s(X^{s,x})\}_{(s,x) \in [0,T] \times \R^d}$ is $\mathscr{L}(\R)$-continuous, and by \cref{prop:stochasticlinearvectorfields_properties} $\{c^s(X^{s,x})\}_{x \in \R^d}$ is twice $\mathscr{L}$-continuously $\mathscr{L}$-differentiable for any fixed $s \in [0,T]$. 
    From \cref{prop:stochasticcontrolledlinearvectorfields_properties} it is possible to deduce that $$\{ (\gamma^s,(\gamma^s)')(X^{s,x}) := (\gamma^s(X^{s,x}),D_x\gamma^s(X^{s,x}) \beta^s(X^{s,x}) + (\gamma^s)'(X^{s,x}))\}_{(s,x) \in [0,T] \times \R^d}$$ is $\{\mathbf{D}_{W^s}^{\eta',\min\{\delta'\eta'\}} \mathscr{L}\}_s$-continuous and that $\{(\gamma^s(X^{s,x}),(\gamma^s,(\gamma^s)')(X^{s,x}))\}_{x \in \R^d}$ is twice $\mathbf{D}_{W^s}^{\eta',\bar{\eta}} \mathscr{L}(\R^{d \times n})$-continuously $\mathbf{D}_{W^s}^{\eta',\bar{\eta}} \mathscr{L}(\R^{d \times n})$-differentiable, where $\bar{\eta} \in [0,\min\{\eta',\delta'\})$ can be chosen so that $\alpha + \eta' + \bar{\eta} >1$. 
    We can therefore deduce that \begin{equation} \label{eq:definitionofI}
        (I^{s,x}, (I^{s,x})') := \left(\int_0^\cdot c^s(r,X^{s,x}_r) dr + \int_0^\cdot (\gamma^s,(\gamma^s)')(r,X^{s,x}_r) d\mathbf{W}^s_r, \ \gamma^s(X^{s,x}) \right) 
    \end{equation}
    is well defined as a stochastic controlled rough path in $\mathbf{D}_{W^s}^{\alpha,\eta}\mathscr{L}(\R)$, for any $(s,x) \in [0,T] \times \R^d$. 
    The a priori bound in statement \textit{1.}\ of \cref{prop:continuityanddifferentiability_forcingterm} and \eqref{eq:aprioriestimateforRSDEin(s,x)} show that \begin{equation*} \begin{aligned}
        \sup_{(s,x) \in [0,T] \times \R^d}\|(I^{s,x}, (I^{s,x})')\|_{\mathbf{D}_{W^s}^{\alpha,\eta}L_{p,\infty}}
        < +\infty,
    \end{aligned} 
    \end{equation*}
    for any $p \in [2,\infty)$, while the proof of statement \textit{2.}\ of \cref{prop:continuityanddifferentiability_forcingterm} and \eqref{eq:stabilityestimatesforRSDEinx} show instead that, for any $x,y \in \R^d$, 
    \begin{equation*} \begin{aligned}
        &\|I^{s,x}, (I^{s,x})';I^{s,y}, (I^{s,y})'\|_{W^s;\alpha,\eta;p}
        \lesssim |x-y|
    \end{aligned}
    \end{equation*}
    where the implicit constant is uniform in $s,x,y$.
    From \cref{prop:continuityanddifferentiability_forcingterm} we deduce that $\{I^{s,x},(I^{s,x})'\}_{(s,x) \in [0,T] \times \R^d}$ is $\{\mathbf{D}_{W^s}^{\alpha,\eta'}\mathscr{L}(\R)\}_s$-continuous, and by applying \cref{prop:continuityanddifferentiability_forcingterm} it is possible to show that $\{(I^{s,x}, (I^{s,x})')\}_{x \in \R^d}$ is twice $\mathbf{D}_{W^s}^{\alpha,\eta'} \mathscr{L}(\R)$-continuously $\mathbf{D}_{W^s}^{\alpha,\eta'} \mathscr{L}(\R)$-differentiable, for any fixed $s \in [0,T]$. 
    Arguing as in the proof of \cref{prop:propertiesoftheexponential}, one can use the two previous inequalities to show that, for any $p \in [2,\infty)$, \begin{equation} \label{eq:aprioriestimatesontheexponential}
        \sup_{(s,x) \in [0,T] \times \R^d} \|(e^{I^{s,x}},(I^{s,x})'e^{I^{s,x}})\|_{\mathbf{D}_{W^s}^{\alpha,\eta}L_p} < +\infty
    \end{equation} and \begin{equation} \label{eq:stabilityoftheexponentialinx}
        \|e^{I^{s,x}},(I^{s,x})'e^{I^{s,x}};e^{I^{s,y}},(I^{s,y})'e^{I^{s,y}}\|_{W^s;\alpha,\eta;p} \lesssim |x-y|.
    \end{equation}
    
    \

    We now prove that $u \in C^{0,2}([0,T] \times \R^d;\R) \cap \mathcal{B}([0,T]; C^\lambda_b(\R^d;\R))$, which by definition is equivalent to show that the following four properties are satisfied: 
    \begin{itemize}
        \item[a)] $u:[0,T] \times \R^d \to \R$ is continuous and (globally) bounded;
        \item[b)] for any $i \in \{1,\dots,d\}$, there exists $\frac{\partial u}{\partial x^i}:[0,T] \times \R^d \to \R$ and it is continuous and bounded;
        \item[c)] for any $i,j \in \{1,\dots,d\}$, there exists $\frac{\partial^2 u}{\partial x^j \partial x^i}:[0,T] \times \R^d \to \R$ and it is continuous and bounded; 
        \item[d)] uniformly in $s \in [0,T]$,  \begin{equation*}
            \left[\frac{\partial^2 u}{\partial x^j \partial x^i}(s,\cdot) \right]_{\lambda-2} := \sup_{x,y \in \R^d, x \ne y} \frac{1}{|x-y|^{\lambda-2}} \left| \frac{\partial^2 u}{\partial x^j \partial x^i}(s,x) - \frac{\partial^2 u}{\partial x^j \partial x^i}(s,y) \right| <+\infty.
        \end{equation*}
    \end{itemize}
    
    a) \ To prove that $u:[0,T] \times \R^d \to \R$ is continuous, we show that the family of random variables $\{g(X^{s,x}_{T-s}) e^{I^{s,x}_{T-s}}\}_{(s,x) \in [0,T] \times \R^d}$ is $\mathscr{L}(\R)$-continuous and we apply \cref{prop:convergenceundermean}.
    Let $(s^0,x^0) \in [0,T] \times \R^d$ and let $((s^n,x^n))_{n\ge 1} \subseteq [0,T] \times \R^d \setminus \{(s^0,x^0)\}$ be any arbitrary sequence with $|s^n-s^0| + |x^n-x^0| \to 0$ as $n \to +\infty$.
    For any $p \in [2,\infty)$, \begin{equation} \label{eq:differentiabilityasarandomvariable}
        \begin{aligned}
            &\|X^{s^n,x^n}_{T-s^n} - X^{s^0,x^0}_{T-s^0}\|_p \le \|X^{s^n,x^n}_{T-s^n} - X^{s^n,x^n}_{T-s^0}\|_p + \|X^{s^n,x^n}_{T-s^0} - X^{s^0,x^0}_{T-s^0}\|_p \\
            &\le \|(X^{s^n,x^n},\beta^{s^n}(X^{s^n,x^n})\|_{W^{s^n};\alpha,\delta';p} |s^n-s^0|^{\alpha}  +\|X^{s^n,x^n},\beta^{s^n}(X^{s^n,x^n}); X^{s^0,x^0},\beta^{s^n}(X^{s^0,x^0})\|_{W^{s^n},W^{s^0};\alpha,\delta';p}
        \end{aligned}
    \end{equation}
    and the right-hand side tends to zero as $n \to +\infty$. 
    Similarly, one shows that $\{I^{s,x}_{T-s}\}_{(s,x) \in [0,T] \times \R^d}$ is a $\mathscr{L}(\R)$-continuous family of random variables, which implies that $\{e^{I^{s,x}_{T-s}}\}_{(s,x) \in [0,T] \times \R^d}$ is $\mathscr{L}(\R)$-continuous, by \cref{prop:continuityanddifferentiabilityofexponentials}. 
    Applying \cref{prop:stochasticlinearvectorfields_properties} to the function $(x,t) \mapsto g(x)y$ and considering \cref{prop:convergenceundermean}, this proves that $u \in C^0([0,T] \times \R^d;\R)$.
    By construction, $u$ is also globally bounded. Indeed, \begin{equation*}
        \begin{aligned}
            |u(s,x)| &\le |g|_\infty \E(|e^{I^{s,x}_{T-s}}|) \le |g|_\infty \|\delta e^{I^{s,x}}\|_{\alpha;1} T^\alpha
        \end{aligned}
    \end{equation*} 
    and the right-hand side is uniformly bounded in $(s,x) \in [0,T] \times \R^d$ by \eqref{eq:aprioriestimatesontheexponential}. \\

    b) \ Note that, for any fixed $s \in [0,T]$, the family $\{g(X^{s,x}_{T-s}) e^{I^{s,x}_{T-s}}\}_{x \in \R^d}$ is twice $\mathscr{L}$-continuously $\mathscr{L}$-differentiable.
    Denote by $e_1,\dots,e_d$ the elements of the canonical basis of $\R^d$, which in particular belong to $\mathbb{S}(\R^d)$.  
    Given any $i \in \{1,\dots,d\}$ and denoting by $Y^{s,x,i}=\mathbf{D}_W^{\alpha,\delta'}\mathscr{L}\text{-}\frac{\partial}{\partial e_i} X^{s,x}$, it holds that the latter is the solution of the following linear rough SDE: \begin{equation} \label{eq:equationforY}
        \begin{aligned}
            Y^{s,x,i}_t &= e_i + \int_0^t D_xb^s(r,X_r^{s,x}) Y^{s,x,i}_r dr + \int_0^t D_x\sigma^s(r,X_r^{s,x}) Y^{s,x,i}_r dB_r + \\
            & \quad + \int_0^t (D_x\beta^s(r,X^{s,x}_r), D^2_{xx}\beta^s(r,X^{s,x}_r)\beta^s(r,X_r^{s,x}) + D_x(\beta^s)'(r,X^{s,x}_r))Y^{s,x,i}_r d \mathbf{W}^s_r \\
            &=: e_i + \int_0^t G_r^{s,x} Y_r^{s,x,i} dr + \int_0^t S_r^{s,x} Y_r^{s,x,i} dB_r + \int_0^t (f^{s,x}_r, (f^{s,x})'_r) Y_r^{s,x,i} d\mathbf{W}^s_r.
        \end{aligned}
    \end{equation}
    From \cref{prop:convergenceundermean}  \begin{equation*}
        \frac{\partial}{\partial x^i} u(s,x) 
        = \E \left( \frac{\partial g}{\partial x^k} (X^{s,x}_{T-s}) (Y^{s,x,i}_{T-s})^k e^{I^{s,x}_{T-s}} + g(X^{s,x}_{T-s}) \left( \mathscr{L}\text{-}\frac{\partial}{\partial e_i} I^{s,x}_{T-s}  \right) e^{I^{s,x}_{T-s}}  \right),
    \end{equation*}
    where $(Y^{s,x,i}_{T-s})^k$ denotes the $k$-th component of the $\R^d$-valued random vector $Y^{s,x,i}_{T-s}$ and  \begin{multline*} 
        \mathscr{L}\text{-} \frac{\partial}{\partial e_i} {I^{s,x}_{T-s}} =  \int_0^{T-s} D_x c^s(r,X^{s,x}_r) Y^{s,x,i}_r  dr + \\ + \int_0^{T-s} \big(D_x \gamma^s(r,X^{s,x}_r), D^2_{xx}\gamma^s(r,X^{s,x}_r) \beta^s(r,X^{s,x}_r) + D_x(\gamma^s)'(r,X^{s,x}_r)\big) Y_r^{s,x,i}  d\mathbf{W}^s_r . 
    \end{multline*}
    Boundedness of $\frac{\partial u}{\partial x^i}$ essentially follows from \eqref{eq:aprioriestimateforRSDEin(s,x)} and \eqref{eq:aprioriestimatesontheexponential}, as already pointed out for $u$. We already know that $\{g(X^{s,x}_{T-s})\}_{(s,x) \in [0,T] \times \R^d}$, $\{D_xg(X^{s,x}_{T-s})\}_{(s,x)}$ and $\{e^{I^{s,x}_{T-s}}\}_{(s,x)}$ are $\mathscr{L}$-continuous, hence the continuity of $\frac{\partial u}{\partial x^i}$ as a map $[0,T] \times \R^d \to \R$ if we show that $\{Y^{s,x,i}_{T-s}\}_{(s,x)\in [0,T] \times \R^d}$ and $\{\mathscr{L}\text{-}\frac{\partial}{\partial e_i}I^{s,x}_{T-s}\}_{(s,x) \in [0,T] \times \R^d}$ are $\mathscr{L}$-continuous. From \cref{prop:stochasticlinearvectorfields_properties} and \cref{rmk:identificationstochasticlinearvectorfields}, $\{G^{s,x}\}_{(s,x) \in [0,T] \times \R^d}$ is $\mathscr{L}\mathcal{L}$-continuous, and the same holds for $\{S^{s,x}\}_{(s,x) \in [0,T] \times \R^d}$.
    Being $\{(D_x\beta^s,D_x(\beta^s)')\}_{s \in [0,T]}$ $\{\mathbf{D}_{W^s}^{2\delta'}\mathscr{L}C^3_b\}_s$-continuous by construction, by applying \cref{prop:stochasticcontrolledlinearvectorfields_properties}.1 we can deduce that $\{(f^{s,x},(f^{s,x})')\}_{(s,x) \in [0,T] \times \R^d}$ is $\{\mathbf{D}_{W^s}^{2\delta'}\mathscr{L}\mathcal{L}\}_s$-continuous. 
    Hence, by \cref{prop:continuity_linearRSDEs} we have that $\{(Y^{s,x,i},f^{s,x}Y^{s,x,i})\}_{(s,x)}$ is $\{\mathbf{D}_{W^s}^{\alpha,\delta'}\mathscr{L}\}_s$-continuous.     
    Similarly, it is possible to show that $\{D_xc^s(X^{s,x})\}_{(s,x) \in [0,T] \times \R^d}$ is $\mathscr{L}\mathcal{L}(\R^d;\R)$-continuous and that $\{(D_x\gamma^s,D_x(\gamma^s)')(X^{s,x})\}_{(s,x) \in [0,T] \times \R^d}$ is $\{\mathbf{D}_{W^s}^{2\eta'}\mathscr{L}\}_s$-continuous. Hence $\{(\mathscr{L}\text{-}\frac{\partial}{\partial e_i} I^{s,x},D_x\gamma^s(X^{s,x}) Y^{s,x,i})\}_{(s,x) \in [0,T] \times \R^d}$ is $\{\mathbf{D}_{W^s}^{\alpha,\min\{\eta',\delta'\}}\mathscr{L}\}_s$-continuous by \cref{prop:continuityanddifferentiability_forcingterm}. \\

    c) \ For any $i,j \in \{1,\dots,d\}$ denote by $Z^{s,x,i,j}= \mathbf{D}_W^{\alpha,\delta'}\mathscr{L}\text{-}\frac{\partial}{\partial e_j} Y^{s,x,i}$. The latter solves the following linear rough SDE: \begin{equation*}
        \begin{aligned}
            Z^{s,x,i,j}_t 
            &= F^{s,x,i,j}_t + \int_0^t G_r^{s,x} Z_r^{s,x,i,j} dr + \int_0^t S_r^{s,x} Z_r^{s,x,i,j} dB_r + \int_0^t (f^{s,x}_r, (f^{s,x})'_r) Z_r^{s,x,i,j} d\mathbf{W}^s_r, 
        \end{aligned}
    \end{equation*}
    where \begin{equation*} \begin{aligned}
        F^{s,x,i,j}_t &= \int_0^t D^2_{xx}b^s(r,X_r^{s,x}) Y^{s,x,j}_r  Y^{s,x,i}_r dr + \int_0^t D^2_{xx}\sigma^s(r,X_r^{s,x}) Y^{s,x,j}_r Y^{s,x,i}_r dB_r + \int_0^t (\phi^{s,x,i,j}_r,(\phi^{s,x,i,j})'_r) d\mathbf{W}^s_r \\
        \phi^{s,x,i,j}_t &= D^2_{xx}\beta^s(t,X^{s,x}_t)Y^{s,x,j}_t Y^{s,x,i}_t \\
        (\phi^{s,x,i,j})'_t &= D^3_{xxx}\beta^s(t,X^{s,x}_t) \beta^s(t,X_t^{s,x}) Y^{s,x,j}_t Y^{s,x,i}_t + D^2_{xx}\beta^s(t,X^{s,x}_t) D_x\beta^s(t,X^{s,x}_t) Y^{s,x,j}_t Y^{s,x,i}_t + \\
        & \quad + D^2_{xx}\beta^s(t,X^{s,x}_t) Y^{s,x,j}_t D_x\beta^s(t,X^{s,x}_t) Y^{s,x,i}_t +  D^2_{xx}(\beta^s)'(t,X^{s,x}_t) Y^{s,x,j}_t Y^{s,x,i}_t . 
    \end{aligned}
    \end{equation*}
    Analogously as before, one deduces that there exists \begin{equation} \label{eq:secondderivativeofu} \begin{aligned}
        \frac{\partial ^2 u}{\partial x^j \partial x^i} (s,x) 
        &= \E \bigg( \frac{\partial^2 g}{\partial x^l \partial x^k}(X^{s,x}_{T-s}) (Y^{s,x,i}_{T-s})^k (Y^{s,x,i}_{T-s})^l e^{I^{s,x}_{T-s}} + \frac{\partial g}{\partial x^k} (X^{s,x}_{T-s}) (Z^{s,x,i,j}_{T-s})^k e^{I^{s,x}_{T-s}} + \\
        & \quad + \frac{\partial g}{\partial x^k} (X^{s,x}_{T-s}) (Y^{s,x,i}_{T-s})^k \left( \mathscr{L}\text{-}\frac{\partial}{\partial e_j} I^{s,x}_{T-s}  \right) e^{I^{s,x}_{T-s}} + \frac{\partial g}{\partial x^k}(X^{s,x}_{T-s}) (Y^{s,x,j})^k \left( \mathscr{L}\text{-}\frac{\partial}{\partial e_i} I^{s,x}_{T-s}  \right) e^{I^{s,x}_{T-s}} + \\
        & \quad + g(X^{s,x}_{T-s}) \left( \mathscr{L}\text{-}\frac{\partial}{\partial e_j} \mathscr{L}\text{-}\frac{\partial}{\partial e_i} I^{s,x}_{T-s}  \right) e^{I^{s,x}_{T-s}} + g(X^{s,x}_{T-s}) \left( \mathscr{L}\text{-}\frac{\partial}{\partial e_i} I^{s,x}_{T-s}  \right) \left( \mathscr{L}\text{-}\frac{\partial}{\partial e_j} I^{s,x}_{T-s}  \right) e^{I^{s,x}_{T-s}} \bigg)
    \end{aligned}
    \end{equation}
    and that it is globally bounded in $(s,x)$. 
    To prove the continuity of $\frac{\partial^2u}{\partial x^j \partial x^i}:[0,T] \times \R^d \to \R$, we are left to show that $\{Z^{s,x,i,j}_{T-s}\}_{(s,x)\in [0,T] \times \R^d}$ and $\{\mathscr{L}\text{-}\frac{\partial}{\partial e_j} \mathscr{L}\text{-}\frac{\partial}{\partial e_i} I^{s,x}_{T-s}\}_{(s,x)\in [0,T] \times \R^d}$ are $\mathscr{L}$-continuous. 
    Taking into account \cref{prop:continuity_linearRSDEs}, we just need to show that $\{F^{s,x,i,j}\}_{(s,x) \in [0,T] \times \R^d}$ is $\{\mathbf{D}_W^{\alpha,\delta'}\mathscr{L}\}_s$-continuous, which is a trivial consequence of \cref{prop:continuityanddifferentiability_forcingterm}, and to apply \cref{prop:continuityanddifferentiability_forcingterm} for the second family of random variables. \\
    
    d) \ We are left to prove that $[\frac{\partial^2 u}{\partial x^j \partial x^i}(s,\cdot)]_{\lambda-2}<+\infty$, uniformly in $s \in [0,T]$.
    To do so, we just combine \eqref{eq:stabilityestimatesforRSDEinx} and \eqref{eq:stabilityoftheexponentialinx} with \cref{thm:stabilityforlinearRSDEs_compendium} and the local Lipschitz estimates on integrals deduced, for instance, in the proof of \cref{prop:continuityanddifferentiability_forcingterm}.2.
    Let $s \in [0,T]$ be fixed and let $x,y \in \R^d$.
    It holds that, for any arbitrary $p \in [2,\infty)$, \begin{equation*}
        \begin{aligned}
            \E \left( \int_0^T |G^{s,x}_r - G^{s,y}_r|^p \right)^\frac{1}{p} \le |D_x b|_\infty \E \left( \int_0^T |X^{s,x}_r - X^{s,y}_r|^p \right) \le |D_x b|_\infty T \sup_{t \in [0,T]} \|X_t^x - X_t^y\|_p \lesssim |x-y|.
        \end{aligned}
    \end{equation*}
    A similar conclusion holds replacing $b$ with $\sigma$. 
    In addition, \cref{prop:compositionwithstochasticcontrolledvectorfields_compendium} leads to \begin{equation*}
        \|f^{s,x},(f^{s,x})';f^{s,y},(f^{s,y})'\|_{W^s;2\delta';p} \lesssim \|X^{s,x},\beta^s(X^{s,x});X^{s,y},\beta^s(X^{s,y})\|_{W^s;2\delta';p} \lesssim |x-y|.
    \end{equation*}
    We can therefore conclude (cf.\ \cref{thm:stabilityforlinearRSDEs_compendium}) that \begin{equation*}
        \|Y^{s,x,i},f^{s,x} Y^{s,x,i}; Y^{s,y,i},f^{s,y} Y^{s,y,i}\|_{W^s;\alpha,\delta';p} \lesssim |x-y|.
    \end{equation*}
    Notice that $\|\phi^{s,x,i,j}_r,(\phi^{s,x,i,j})'_r ; \phi^{s,y,i,j}_r,(\phi^{s,y,i,j})'_r\|_{W^s;2\delta';p} \lesssim |x-y|$ and, due to our assumptions on $b$, \begin{equation*}
        \begin{aligned}
            \E \left( \int_0^T |D_{xx} b^s(r,X_r^{s,x}) - D^2_{xx} b^s(r,X^{s,x}_r)|^p dr \right)^\frac{1}{p} &\le \sup_{t \in [0,T]} [D^2_{xx}b(t,\cdot)]_{\lambda-2} \E \left( \int_0^T |X_r^{s,x}- X^{s,x}_r|^{(\lambda-2)p} dr \right)^\frac{1}{p} \\
            &\lesssim_T [D^2_{xx}b(t,\cdot)]_{\lambda-2} \sup_{t \in [0,T]} \|X_t^{s,x}- X^{s,x}_t\|_{(\lambda-2)p}^{\lambda-2}  \lesssim |x-y|^{\lambda-2}.
        \end{aligned}
    \end{equation*}
    A similar conclusion holds if we replace $b$ with $\sigma$. 
    An application of H\"older's inequality is sufficient to show that \begin{equation*}
        \|F^{s,x,i,j},\phi^{s,x,i,j}; F^{s,y,i,j},\phi^{s,y,i,j} \|_{W^s;\alpha,\delta';p} \lesssim |x-y|^{\lambda-2} \qquad \text{for any $p \in [2,\infty)$}. 
    \end{equation*}
    This implies that (cf.\ \cref{thm:stabilityforlinearRSDEs_compendium}) \begin{equation*}
        \|Z^{s,x,i,j}, (Z^{s,x,i,j})' ; Z^{s,y,i,j}, (Z^{s,y,i,j})'\|_{W^s;\alpha,\delta';p} \lesssim |x-y|^{\lambda-2}.
    \end{equation*}
    Analogous estimates show that \begin{equation*}
        \bigg\|\mathscr{L}\text{-}\frac{\partial}{\partial e_i}I^{s,x},\left(\mathscr{L}\text{-}\frac{\partial}{\partial e_i}I^{s,x}\right)'; \mathscr{L}\text{-}\frac{\partial}{\partial e_i}I^{s,y}, \left(\mathscr{L}\text{-}\frac{\partial}{\partial e_i}I^{s,y}\right)'  \bigg\|_{W^s;\alpha,\min\{\eta',\delta'\};p} \lesssim |x-y|^{\lambda-2},
    \end{equation*}
    ans the same holds by replacing $\mathscr{L}\text{-}\frac{\partial}{\partial e_i}I^{s,x}$ with $\mathscr{L}\text{-}\frac{\partial}{\partial e_j}\mathscr{L}\text{-}\frac{\partial}{\partial e_i}I^{s,x}$.
    Considering \eqref{eq:secondderivativeofu} and the fact that $D^2_{xx}g(\cdot)$ is $(\lambda-2)$-H\"older continuous on $\R^d$ by assumption, the last statement of the assertion is proved.\\

    To conclude the proof, we show that $u \in C^{\alpha}([0,T];C^2_b(\R^d;\R))$ using the results of \cref{section:functionalsandmarkovproperty}.
    Notice that $u_t(\cdot) \in C^2_b(\R^d;\R)$ for any $t \in [0,T]$.
    Let $0 \le s \le t \le T$ be arbitrary.
    Let $Y$ be defined as in \eqref{eq:definitionofY_flowproperty} and let $J=(J_\tau)_{\tau \in [0,T-(t-s)]}$ be defined by \begin{equation*}
        J_\tau := \int_0^\tau c^t(r,Y_r) dr + \int_0^\tau (\gamma^t(r,Y_r),D_x\gamma^t(r,Y_r)\beta^t(r,Y_r)) d\mathbf{W}^t_r. 
    \end{equation*}
    By means of \cref{lemma:markovpropertyforRSDEs} it holds that
    \begin{equation*}
        \begin{aligned}
            &|u_t(x) -u_s(x)| 
            = \left| \E \left( g(X_{T-t}^{t,x})e^{I^{t,x}_{T-t}} - e^{I^{s,x}_{t-s}} g(Y_{T-t})e^{J_{T-t}} \right) \right| \\
            &\le \E \left( \left| g(X_{T-t}^{t,x})e^{I^{t,x}_{T-t}} \left(1 - e^{I^{s,x}_{t-s}}\right) \right| \right) + \E \left( \left| e^{I^{s,x}_{t-s}} g(X_{T-t}^{t,x}) \left(e^{I^{t,x}_{T-t}} - e^{J_{T-t}}  \right) \right| \right) \\
            & \quad  + \E\left( \left| e^{I^{s,x}_{t-s}} \left(g(X_{T-t}^{t,x}) - g(Y_{T-t})\right)e^{J_{T-t}} \right| \right) \lesssim_M \|e^{I^{s,x}_{t-s}}-1\|_2 + \| e^{I^{t,x}_{T-t}}-e^{J_{t-t}}\|_2 + \|X^{t,x}_{T-t} - Y_{T-t}\|_2 
        \end{aligned}
    \end{equation*}
    where $M$ is any positive constant (independent on $x$) such that $|g|_{C^1_b} +  \|\delta e^{I^{t,x}_\cdot}\|_{\alpha;4} + \|\delta e^{I^{s,x}_\cdot}\|_{\alpha;4} + \|\delta e^{J_\cdot}\|_{\alpha;4} \le M$.
    Notice that $\|e^{I^{s,x}_{t-s}}-1\|_2 \le \|\delta e^{I^{s,x}_\cdot}\|_{\alpha;4} |t-s|^\alpha$. 
    Moreover, from \cref{thm:stabilityforRSDEs_compendium} and uniformly in $x \in \R^d$ \begin{equation*}
         \sup_{r \in [0,T]} \|X^{t,x}_r - Y_r\|_2 \lesssim \|X_{t-s}^{s,x}-x\|_2 \le \|\delta X^{s,x}\|_{\alpha;2,\infty} |t-s|^\alpha.
    \end{equation*} 
    Similarly, \begin{equation*}
        \sup_{r \in [0,T]} \| e^{I^{t,x}_r}-e^{J_r}\|_2 \lesssim \|X^{t,x},\beta^t(X^{t,x});Y,\beta^t(Y)\|_{W^t;\alpha,\delta';2} \lesssim \|x-X^{s,x}_{t-s}\|_2 \lesssim |t-s|^\alpha. 
    \end{equation*}
    Indeed, recall that $Y$ solves the same rough SDE as $X^{t,x}$, but with $X^{s,x}_{t-s}$ as initial condition. 
    We have therefore showed that \begin{equation*}
        \sup_{x \in \R^d} |u_t(x) - u_s(x)| \lesssim |t-s|^\alpha. 
    \end{equation*}
    Arguing in a very similar way, one can prove the same for the functions $\frac{\partial u}{\partial x^i}$ and $\frac{\partial^2 u}{\partial x^j \partial x^i}$. 
\end{proof}

\begin{thm}[Existence] \label{thm:existenceRPDE}
    Let $u=u(s,x)$ be defined as in \eqref{eq:definitionofu_mainresult}. 
    Then, under \cref{assumptions}, $u$ is a regular solution of \  $-d u_t = L_t u_t \, dt + \Gamma_t u_t \, d\mathbf{W}_t$ \ in the sense of \cref{def:solutiontobacwardroughPDEs}. 
\end{thm}

\begin{proof} 
    We prove that $u$ satisfies all the requirements of \cref{def:solutiontobacwardroughPDEs}. 
    By \cref{prop:regularityofu}, $u \in C^{0,2}([0,T] \times \R^d; \R)$ and, by construction, $u_T(x) = g(x)$ for any $x \in \R^d$. 
    Let $(s,t) \in \Delta_{[0,T]}$ and $x \in \R^d$ be arbitrary. 
    Recall that $u \in C^\alpha([0,T];C^2_b(\R^d;\R))$. Hence, under our assumptions on $\beta,\beta'$ and $\gamma,\gamma'$ and looking at \eqref{eq:explicitGammaGamma}, it is easy to deduce that, for any $\mu,\nu=1,\dots,n$ and uniformly in $x$,
    \begin{equation*}
        |((\Gamma_s)_\mu (\Gamma_s)_\nu - (\Gamma'_s)_{\mu \nu}) u_s (x) - ((\Gamma_t)_\mu (\Gamma_t)_\nu - (\Gamma_t')_{\mu \nu}) u_t (x)| \lesssim |t-s|^{\min\{\alpha,\delta,\eta\}}.
    \end{equation*} 
    Moreover, for any $\mu = 1,\dots,n$ we can write \begin{equation*}
        \begin{aligned}
            (\Gamma_s)_\mu u_s(x) - (\Gamma_t)_\mu u_t(x) &= (\beta_\mu^i(s,x) - \beta^i_\mu(t,x)) \frac{\partial u_t}{\partial x^i}(x) + \beta^i_\mu(t,x) \left( \frac{\partial u_s}{\partial x^i}(x) - \frac{\partial u_t}{\partial x^i}(x) \right)+ \\
            & \quad - (\gamma_\mu(t,x) - \gamma_\mu(s,x)) u_t(x) +  \gamma_\mu(t,x) (u_s(x)-u_t(x)) + P^1_{s,t}(x) = \\
            &= - (\beta_{\mu \nu}')^i(t,x) \delta W_{s,t}^{\nu}   \frac{\partial u_t}{\partial x^i}(x) + \beta^i_\mu(t,x) \left( \frac{\partial u_s}{\partial x^i}(x) - \frac{\partial u_t}{\partial x^i}(x) \right) + \\
            & \quad - \gamma'_{\mu \nu}(t,x) \delta W^\nu_{s,t} u_t(x) + \gamma_\mu(t,x) (u_s(x)-u_t(x)) + P^2_{s,t}(x)
        \end{aligned}
    \end{equation*}
    with $\sup_{x \in \R^d}|P^i_{s,t}(x)| \lesssim |t-s|^{\alpha + \min\{\delta,\eta\}}$ ($i=1,2$). 
    From \cref{lemma:markovpropertyforRSDEs}  $u_s(x) = \E \left(e^{I^{s,x}_{t-s}} \ u_t(X_{t-s}^{s,x}) \right)$, which leads to \begin{equation} \label{eq:tobemultipliedbygamma} \begin{aligned}
        u_s(x) - u_t(x) &= \E\left( e^{I^{s,x}_{t-s}} \ u_t(X_{t-s}^{s,x}) - u_t(x) \right) = \\
        &= \left( \gamma_\nu(s,x) u_t(x) + \beta_\nu^j(s,x) \frac{\partial u_t}{\partial x^j}(x) \right) \delta W_{s,t}^\nu + \E\left(R^{e^{I^{s,x}} u(X^{s,x})}_{0,t-s} \right) \\
        &= \left( \gamma_\nu(t,x) u_t(x) + \beta_\nu^j(t,x) \frac{\partial u_t}{\partial x^j}(x) \right) \delta W_{s,t}^\nu + Q_{s,t}(x) 
    \end{aligned}
    \end{equation} 
    with $\sup_{x \in \R^d} |Q_{s,t}(x)| \lesssim |t-s|^{\alpha+\min\{\delta,\eta\}}$. 
    The previous identities hold because $u_t(\cdot) \in C^2_b(\R^d;\R)$, $e^{I^{s,x}_0}=1$ and, by \cref{prop:compositionwithstochasticcontrolledvectorfields_compendium} and \cref{prop:propertiesoftheexponential}, $$(e^{I^{s,x}}u_t(X^{s,x}), \gamma^s(X^{s,x}) e^{I^{s,x}}u_t(X^{s,x}) + e^{I^{s,x}} D_x u_t(X^{s,x}) \beta^s(X^{s,x})) \in \mathbf{D}_{W^s}^{\alpha,\min\{\delta,\eta\}}\mathscr{L}(\R). $$ 
    Recall that by construction $|\E(R^{e^{I^{s,x}} u(X^{s,x})}_{0,t-s})| \le \E(|\E_0(R^{e^{I^{s,x}} u(X^{s,x})}_{0,t-s})|) \lesssim |t-s|^{\alpha+\min\{\delta,\eta\}}$.
    From \cref{prop:convergenceundermean} we get that, for any $i=1,\dots,d$, \begin{equation*} \begin{aligned}
        &\frac{\partial}{\partial x^i} u_s(x) = \frac{\partial}{\partial x^i} \E\left( e^{I^{s,x}_{t-s}} \ u_t(X_{t-s}^{s,x}) \right) = \E\left(  e^{I^{s,x}_{t-s}} \ D_x u_t(X_{t-s}^{s,x}) Y_{t-s}^{s,x,i} + \mathscr{L}\text{-}\frac{\partial }{\partial e_i} I^{s,x}_{t-s} \ e^{I^{s,x}_{t-s}} \ u_t(X_{t-s}^{s,x}) \right),
    \end{aligned}
    \end{equation*}
    where $e_i$ denotes the $i$-th element of the canonical basis of $\R^d$, $Y^{s,x,i}$ is defined as in \eqref{eq:equationforY} and \begin{equation*} \begin{aligned}
        \mathscr{L}\text{-}\frac{\partial }{\partial e_i} I^{s,x} &= \int_0^\cdot D_x c(s+r,X_r^{s,x}) Y^{s,x,i}_r dr + \\
        & \quad +  \int_0^\cdot (D_x\gamma(s+r,X_r^{s,x}), D^2_{xx} \gamma(s+r,X_r^{s,x}) \beta(s+r,X_r^{s,x}) + D_x \gamma' (s+r,X^{s,x}_r) ) Y^{s,x,i}_r d\mathbf{W}^s_r.
    \end{aligned} 
    \end{equation*}
    Similarly as before one can show that, for any $i=1,\dots,d$, \begin{equation*} \label{eq:tomultiplybyBeta}\begin{aligned}
         &\frac{\partial u_t}{\partial x^i}(x) - \frac{\partial u_s}{\partial x^i}(x) = \E \left(e^{I^{s,x}_{t-s}} \ D_x u_t(X_{t-s}^{s,x}) Y_{t-s}^{s,x,i} + \mathscr{L}\text{-}\frac{\partial }{\partial e_i} I^{s,x}_{t-s} \ e^{I^{s,x}_{t-s}} \ u_t(X_{t-s}^{s,x}) - \frac{\partial}{\partial x^i} u_t(x) \right) = \\
         &\quad = \left( \gamma_\nu (s,x) \frac{\partial u_t}{\partial x^i}(x) + \frac{\partial^2 u_t}{\partial x^k \partial x^i}(x) \beta^k_\nu (s,x) + \frac{\partial u_t}{\partial x^j}(x) \frac{\partial \beta^j_\nu}{\partial x^i}(s,x) + \frac{\partial \gamma_\nu}{\partial x^i}(s,x) u_t(x) \right) \delta W_{s,t}^\nu + S^1_{s,t}(x) \\
         &\quad = \left( \gamma_\nu (t,x) \frac{\partial u_t}{\partial x^i}(x) + \frac{\partial^2 u_t}{\partial x^k \partial x^i}(x) \beta^k_\nu (t,x) + \frac{\partial u_t}{\partial x^j}(x) \frac{\partial \beta^j_\nu}{\partial x^i}(t,x) + \frac{\partial \gamma_\nu}{\partial x^i}(t,x) u_t(x) \right) \delta W_{s,t}^\nu + S^2_{s,t}(x)
     \end{aligned}
     \end{equation*} 
     with $\sup_{x \in \R^d} |S^i_{s,t}(x)| \lesssim |t-s|^{\alpha+\min\{(\lambda-2)\alpha,\delta,\eta\}} \ (i=1,2)$ . 
    The previous identity hold because $D_x u_t (\cdot) \in C^{\lambda-1}_b$ with $\lambda-1 \in (1,2]$, $Y^{s,x,i}_0 = e_i$ and since $e^{I^{s,x}} \ D_x u_t(X^{s,x}) Y^{s,x,i} + \mathscr{L}\text{-}\frac{\partial }{\partial e_i} I^{s,x} \ e^{I^{s,x}} \ u_t(X^{s,x})$ can be proven to be stochastically controlled.
     We therefore deduce that, for any arbitrary $\mu=1,\dots,n$ and uniformly in $x \in \R^d$, \begin{equation*}
         |(\Gamma_t)_\kappa u_t(x) - (\Gamma_s)_\kappa u_s(x) - ((\Gamma_t)_\kappa (\Gamma_t)_\mu - (\Gamma_t')_{\mu \nu}) u_t(x) \delta W^\mu_{s,t}| \lesssim |t-s|^{\min\{(\lambda-2)\alpha,\delta,\eta\}}. 
     \end{equation*}
     Notice that, from \ref{assumptions}, $\alpha + \min\{\delta,\eta\} + \min\{(\lambda-2)\alpha,\delta,\eta\} >1$. That is, condition ii) of \cref{def:solutiontobacwardroughPDEs} is satisfied with $\alpha'= \min\{\delta,\eta\}$ and $\alpha''=\min\{(\lambda-2)\alpha,2\delta,2\eta\}$. \\

     The final step of this proof consists in showing that $u$ satisfies the Davie-type expansion stated in \cref{def:solutiontobacwardroughPDEs}.iii. 
     It is useful to recall that, for any $(t_1,t_2) \in \Delta_{[0,T]}$ and for any $i=1,\dots,d$, we can write \begin{equation} \label{eq:davideexpansionofX}
        \begin{aligned}
            (X^{s,x}_{t_2})^i &= (X^{s,x}_{t_1})^i + \int_{t_1}^{t_2} b^i(s+r,X_r^{s,x}) dr + \int_{t_1}^{t_2} \sigma_\nu^i(s+r,X_r^{s,x}) dB_r^\nu +  \beta^i_\mu (s+t_1,X_{t_1}^{s,x}) (\delta W^s_{t_1,t_2})^\mu + \\
            &\quad + \left(\frac{\partial \beta_\nu^i}{\partial x^j} (s+t_1,X_{t_1}^{s,x}) \beta_\mu^j(s+t_1,X_{t_1}^{s,x}) + (\beta_{\nu \mu}')^i(s+t_1,X^{s,x}_{t_1}) \right)(\mathbb{W}^s_{t_1,t_2})^{\mu \nu} + (X^{s,x,\natural}_{t_1,t_2})^i,  
        \end{aligned}
    \end{equation}
    with $\|X^{s,x,\natural}_{t_1,t_2}\|_p \lesssim |t-s|^{\alpha + \delta}$ and $\|\E_{t_1}(X^{s,x,\natural}_{t_1,t_2})\|_p \lesssim |t-s|^{\alpha + 2\delta}$, for any $p \in [2,\infty)$. 
    While from \cref{prop:propertiesoftheexponential} it also follows that \begin{equation} \label{eq:davieexpansionofexponential}
        \begin{aligned}
            e^{I^{s,x}_{t_2}} - e^{I^{s,x}_{t_1}} &= \int_{t_1}^{t_2} c(s+r,X^{s,x}_r) e^{I_r^{s,x}} dr + \gamma_\mu(s+t_1,X_{t_1}^{s,x}) e^{I^{s,x}_{t_1}} \delta W^\mu_{t_1,t_2} +  \bigg( \frac{\partial \gamma_\nu}{\partial x^j}(s+t_1,X_{t_1}^{s,x}) \beta_\mu^j(s+t_1,X_{t_1}^{s,x}) + \\
            &\quad + \gamma_{\nu \mu}'(s+t_1,X^{s,x}_{t_1}) + \gamma_\mu(s+t_1,X_{t_1}^{s,x}) \gamma_\nu(s+t_1,X_{t_1}^{s,x}) \bigg) e^{I^{s,x}_{t_1}} \mathbb{W}^{\mu \nu}_{t_1,t_2} + e^{s,x,\natural}_{t_1,t_2}.
        \end{aligned}
    \end{equation}
    Here $ \|e^{s,x,\natural}_{t_1,t_2}\|_p \lesssim |t-s|^{\alpha + \eta}$ and $\|\E_{t_1}(e^{s,x,\natural}_{t_1,t_2})\|_p \lesssim |t-s|^{\alpha + \eta + \min\{\delta,\eta\}}$, for any $p \in [2,\infty)$.
     Recalling that $u_t(\cdot) \in C^\lambda_b(\R^d;\R)$ and by applying \cref{prop:roughItoformula} to the product process $e^{I^{s,x}_\cdot}u_t(X^{s,x}_\cdot)$ we have \begin{equation} \label{eq:roughItoformula_existence}
        \begin{aligned}
            &e^{I^{s,x}_{t-s}}u_t(X^{s,x}_{t-s}) - u_t(x) = \\
            & = \int_0^{t-s} \underbrace{c(s+r,X_r^{s,x}) e^{I^{s,x}_r} u_t(X_s^{s,x}) + e^{I^{s,x}_r}\frac{\partial u_t}{\partial x^i}(X_r^{s,x}) b^i(s+r,X_r^{s,x}) + e^{I^{s,x}_r}\frac{1}{2} \frac{\partial^2 u_t}{\partial x^j \partial x^i} (\sigma^i_a \sigma^j_a)(s+r,X_r^{s,x})}_{= e^{I^{s,x}_r}  L_{s+r} u_t(X_r^{s,x})} \ dr + \\
            & \quad  + \int_0^{t-s} e^{I^{s,x}_r} \frac{\partial u_t}{\partial x^i} (X_r^{s,x}) \sigma^i_a (s+r,X^{s,x}_r) dB^a_r + \bigg(\gamma_\mu(s,x) u_t(x) + \frac{\partial u_t}{\partial x^i}(x) \beta^i_\mu(s,x)  \bigg) \delta (W^s_{0,t-s})^\mu + \\
            & \quad  + \bigg( \frac{\partial \gamma_\nu}{\partial x^i}(s,x) \beta^i_\mu(s,x) u_t(x) + \gamma'_{\nu \mu} (s,x) u_t(x) + \gamma_\mu(s,x) \gamma_\nu (s,x) u_t(x) + \gamma_\mu(s,x) \frac{\partial u_t}{\partial x^i}(x) \beta^i_\nu(s,x) + \\
            & \qquad  + \frac{\partial u_t}{\partial x^i}(x) \beta^i_\mu(s,x) \gamma_\nu (s,x) +  \frac{\partial u_t}{\partial x^i}(x) (\beta'_{\nu \mu})^i(s,x) + \frac{\partial u_t}{\partial x^i}(x) \frac{\partial \beta^i_\nu}{\partial x^j}(s,x) \beta^j_\mu (s,x) + \\
            & \qquad  + \frac{\partial^2 u_t}{\partial x^j \partial x^i}(x) \beta^i_\mu(s,x) \beta^j_\mu(s,x) \bigg) (\mathbb{W}^s_{0,t-s})^{\mu \nu} + P^{s,x\natural}_{0,t-s}, 
        \end{aligned}
    \end{equation}
    where $\|P^{s,x,\natural}_{0,t-s}\|_p \lesssim |t-s|^{\alpha + \alpha'}$ and $\|\E_0(P^{s,x,\natural}_{0,t-s})\|_p \lesssim |t-s|^{\alpha + \alpha' + \min\{\alpha',\alpha''\}}$, for any $p \in [2,\infty)$. 
    The conclusion follows by taking the expectation of the previous identity, up to slightly rearrange some terms. 
    Similarly as before, $|\E(P^{s,x,\natural}_{0,t-s})| \lesssim |t-s|^{\alpha + \alpha' + \min\{\alpha',\alpha''\}}$, and due to the martingale property of the It\^o integral $\E \big(\int_0^{t-s} e^{I_r^{s,x}} \frac{\partial u_t}{\partial x^i}(X_r^{s,x}) \sigma_a^i(s+r,X_r^{s,x}) dB^a_r\big) =0$ (notice that $e^{I^{s,x}_\cdot} \frac{\partial u_t}{\partial x^i} (X_\cdot^{s,x}) \sigma (s+\cdot,X^{s,x}_\cdot) \in L^2(\Omega \times [0,T];\R^{m})$).
     Moreover, we can write $\int_0^{t-s} e^{I^{s,x}_r}  L_{s+r} u_t(X_r^{s,x}) dr$ as $\int_s^{t} e^{I^{s,x}_{r-s}}  L_{r} u_t (X_{r-s}^{s,x}) dr$ and deduce that   \begin{equation} \label{eq.Lebesgueintegralisalittleo}
        \begin{aligned}
            & \left|\E \left( \int_s^t e^{I^{s,x}_{r-s}} L_r u_t (X_{r-s}^{s,x}) dr \right) -\int_s^t L_r u_r(x) dr  \right|  \\
            &\le \E \left( \left| \int_s^t (e^{I^{s,x}_{r-s}}-1) L_r u_t (X_{r-s}^{s,x}) dr \right| \right) + \E \left( \left| \int_s^t L_r u_t (X_{r-s}^{s,x}) - L_r u_t (x) dr \right| \right) + \E \left( \left| \int_s^t L_r u_t (x) - L_r u_r(x) dr \right| \right) \\
            &\le |u_t(\cdot)|_{C^2_b} \left(|b|_{C^1_b} + |D_x\sigma|^2_{C^2_b} \right) \left( \| \delta e^{I^{s,x}}\|_{\alpha;1} +  \|\delta X^{s,0,x}\|_{\alpha;1,\infty} + [u]_\alpha \right) |t-s|^{1+\alpha},
        \end{aligned}
    \end{equation}
    The proof is concluded since, from the assumptions on $\beta,\beta'$ and $\gamma,\gamma'$,
    \begin{equation*}
         \sup_{x \in \R^d}|(\Gamma_s)_\mu (\Gamma_s)_\nu u_t(x) \mathbb{W}_{s,t}^{\mu \nu} - (\Gamma_t)_\mu (\Gamma_t)_\nu u_t(x) \mathbb{W}_{s,t}^{\mu \nu}| \lesssim |t-s|^{2\alpha + \min\{\delta,\eta\}}.
     \end{equation*} and, being $\mathbf{W}$ is weakly geometric, \begin{equation*} \begin{aligned}
        &\left(\gamma_\mu(s,x) u_t(x) +  \beta^i_\mu(s,x) \frac{\partial u_t}{\partial x^i}(x)  \right) \delta W_{s,t}^\mu - \left(\gamma_\mu(t,x) u_t(x) +  \beta^i_\mu(t,x) \frac{\partial u_t}{\partial x^i}(x)  \right) \delta W_{s,t}^\mu = \\
        &\quad =  - \left(\gamma'_{\mu \nu}(t,x) \delta W^\nu_{s,t} u_t(x) + (\beta'_{\mu \nu})^i(t,x) \delta W^\nu_{s,t} \frac{\partial u_t}{\partial x^i}(x) \right)  \delta W_{s,t}^\mu   + T^1_{s,t}(x) = \\
        & \quad =  - \left(\gamma'_{\mu \nu}(t,x)  u_t(x) + (\beta'_{\mu \nu})^i(t,x)  \frac{\partial u_t}{\partial x^i}(x) \right)  (\mathbb{W}_{s,t}^{\nu \mu} + \mathbb{W}^{\mu \nu}_{s,t}) + T^2_{s,t}(x)   
    \end{aligned} 
    \end{equation*}
    with $\sup_{x \in \R^d} |T^i_{s,t}(x)| \lesssim |t-s|^{\alpha + \min\{2\delta,2\eta\}}, \ (i=1,2)$.
\end{proof}

\begin{prop} \label{prop:robustnessinW}
    Let $g,b,\sigma,\beta,c$ and $\gamma$ be fixed as in \cref{assumptions}. 
    For any weakly geometric $\alpha$-rough path $\mathbf{W}$, denote by $u^\mathbf{W}=u^\mathbf{W}(s,x)$ the solution to $-du_t = L_t u_t dt + \Gamma_t u_t d\mathbf{W}_t$ defined as in \eqref{eq:definitionofu_mainresult}. 
    Then the map \begin{equation*}
        \mathscr{C}_g^\alpha([0,T];\R^n) \ni \mathbf{W} \longmapsto u^\mathbf{W} \in C^{0,2}_b([0,T]\times \R^d;\R)
    \end{equation*}
    is locally Lipschitz continuous. 
\end{prop}

\begin{proof}
    Let $R>0$ and let $\mathbf{W}, \mathbf{Z} \in \mathscr{C}^\alpha_g([0,T];\R^n)$ be a pair of weakly geometric rough paths such that $|\mathbf{W}^n|_\alpha,|\mathbf{Z}^n|_\alpha \le R$. 
    Recall that \begin{equation} \label{eq:uniformcontinuity}
        |u^{\mathbf{W}} - u^{\mathbf{Z}}|_{C^{0,2}_b} := |u^{\mathbf{W}} - u^{\mathbf{Z}}|_\infty + |D_xu^{\mathbf{W}} - D_xu^{\mathbf{Z}}|_\infty + |D^2_{xx}u^{\mathbf{W}} - D^2_{xx}u^{\mathbf{Z}}|_\infty.
    \end{equation} 
    Let $(s,x) \in [0,T] \times \R^d$ be fixed. 
    With an obvious notation, let us denote by $X^{s,x;\mathbf{W}}$ the solution to equation \eqref{eq:RSDE_backwardroughPDE_main} driven by $\mathbf{W}^s$ and by $X^{s,x;\mathbf{Z}}$ the solution to the same equation but driven by $\mathbf{Z}^s$. 
    Similarly, we define $I^{s,x;\mathbf{W}}$ and $I^{s,x;\mathbf{Z}}$ (see \eqref{eq:definitionofI}). 
    It holds \begin{equation*}
        \begin{aligned}
            |u^{\mathbf{W}}(s,x) - u^{\mathbf{Z}}(s,x)| 
            &\le |D_xg|_\infty \E \left(|X^{s,x;\mathbf{W}}_{T-s}-X^{s,x;\mathbf{Z}}_{T-s}| e^{I^{s,x;\mathbf{W}}_{T-s}}\right) + |g|_\infty \E \left(\big|e^{I^{s,x;\mathbf{W}}_{T-s}}-e^{I^{s,x;\mathbf{Z}}_{T-s}}\big|\right). 
        \end{aligned}
    \end{equation*}
    From \cref{thm:stabilityforRSDEs_compendium} we have that $\|X^{s,x;\mathbf{W}}_{T-s}-X^{s,x;\mathbf{Z}}_{T-s}\|_p \lesssim_R \rho_\alpha(\mathbf{W},\mathbf{Z})$ for any $p \in [2,\infty)$ and uniformly in $(s,x)$.
    From the estimate in \cref{prop:continuityanddifferentiability_forcingterm}.1 we have that $\sup_n \|I^{s,x;\mathbf{W}}\|_p < +\infty$ for any $p \in [2,\infty)$, and arguing as in the proof of \cref{prop:continuityanddifferentiability_forcingterm}.2 we obtain $\|I^{s,x;\mathbf{W}}-I^{s,x;\mathbf{Z}}\|_p \lesssim \rho_\alpha(\mathbf{W},\mathbf{Z})$ for any $p \in [2,\infty)$ and for an implicit constant that depends on $R$ but not on $s$ nor $x$. 
    Arguing as in the proof of \cref{prop:continuityanddifferentiabilityofexponentials}.1, we can show that $\|e^{I^{s,x;\mathbf{W}}_{T-s}}-e^{I^{s,x;\mathbf{Z}}_{T-s}}\|_p \lesssim_R \rho_\alpha(\mathbf{W},\mathbf{Z})$ for any $p \in [2,\infty)$ and uniformly in $s,x$. 
    Thus, we conclude that \begin{equation*}
        |u^{\mathbf{W}} - u^{\mathbf{Z}}|_\infty := \sup_{(s,x) \in [0,T] \times \R^d} |u^{\mathbf{W}}(s,x) - u^{\mathbf{Z}}(s,x)| \lesssim_R \rho_\alpha(\mathbf{W},\mathbf{Z}). 
    \end{equation*}
    In the proof of \cref{prop:regularityofu} we deduce explicit expressions for higher order space-derivatives of a solution $u$. Hence we can argue as before - also using local Lipschitz estimates for linear rough SDEs contained in \cref{thm:stabilityforlinearRSDEs_compendium} - to conclude the proof of the assertion. 
\end{proof}

The following version of the rough It\^o formula is needed in the proof of the previous existence result and to prove uniqueness (cf.\ \cref{thm:uniqueness} below).
It easily follows from the rough It\^o formula for functions of polynomial growth developed in \cite[Proposition 2.16]{BCN24}. 
With respect to the original formulation in \cite[Theorem 4.13]{FHL21}, here the $p$-integrability of the rough It\^o process $X$ is not a priori fixed, and the use functions of polynomial growth (i.e.\ not globally bounded) allows to derive a rough product formula.

\begin{prop}[Rough It\^o's formula] \label{prop:roughItoformula}
Let $\mathbf{W}=(W,\mathbb{W})$ be a weakly geometric $\alpha$-rough path, with $\alpha \in (\frac{1}{3},\frac{1}{2}]$, and let $B$ be an $m$-dimensional Brownian motion on $(\Omega,\mathcal{F},\{\mathcal{F}_t\}_{t \in [0,T]},\mathbb{P})$ complete filtered probability space. 
Let $\kappa \in (\frac{1}{\alpha},3]$ and let $F \in C^\kappa(\R^{d_1};\R^{d_2})$ such that \begin{equation*}
    |D_x F(x)| + \dots + |D_{x\dots x}^{\lfloor \kappa \rfloor} F(x)| \le K(1 + |x|^k), 
\end{equation*} 
for some $K>0$, $k \ge 0$ and for any $x \in \R^{d_1}$. 
Let $(\phi,\phi') \in \mathbf{D}_W^{\gamma,\gamma'}\mathscr{L}(\mathcal{L}(\R^n;\R^{d_1}))$ be a stochastic controlled rough path for some $\gamma \in [0,\alpha], \gamma' \in [0,1]$ with $\alpha + \gamma > \frac{1}{2}$ and $\alpha + \gamma + \min\{(\kappa-2)\alpha,\gamma,\gamma'\}>1$, and such that $\sup_{t \in [0,T]} \|\phi_t\|_p < +\infty$ for any $p \in [2,\infty)$. Let $\mu \in \mathscr{L}(\R^{d_1})$, $\nu \in \mathscr{L}(\R^{d_1 \times m})$.
Let $X=(X_t)_{t \in [0,T]}$ be an continuous and adapted $\R^{d_1}$-valued stochastic process such that \begin{equation} \label{eq:roughItoprocess}
    X_t = X_0 + \int_0^t \mu_r dr + \int_0^t \nu_r dB_r + \int_0^t (\phi_r,\phi_r') d\mathbf{W}_r \qquad t \in [0,T].
\end{equation}
Then \begin{equation} \label{eq:controllabilityinroughItoformula}
    (D_xF(X) \phi,D_xF(X)\phi'+ D^2_{xx}F(X)(\phi, \phi)) \in \mathbf{D}_W^{\gamma,\min\{(\kappa-2)\gamma,\gamma'\}}\mathscr{L}(\mathcal{L}(\R^n;\R^{d_2}))
\end{equation} and, $\mathbb{P}$-a.s. and for any $t \in [0,T]$, \begin{equation*} \begin{aligned}
    F(X_t) &= F(X_0) + \int_0^t D_xF(X_r) \mu_r + \frac{1}{2} \text{tr}\left( D^2_{xx}F(X_r) \nu_r \nu_r^\top \right) dr + \int_0^t D_xF(X_r) \nu_r dB_r + \\
    & \quad + \int_0^t (D_xF(X_r) \phi,D_xF(X_r)\phi'_r+ D^2_{xx}F(X_r) (\phi_r, \phi_r))d\mathbf{W}_r.
\end{aligned} 
\end{equation*}
Equivalently, for any $s \le t$, for any $i=1,\dots,d_2$ and $\mathbb{P}$-almost surely, it holds that \begin{equation*} \begin{aligned}
    F^i(X_t) &= F^i(X_s) + \int_s^t \left( \frac{\partial F^i}{\partial x^j} (X_r) \mu_r^j + \frac{1}{2} \sum_{a=1}^m \frac{\partial^2 F^i}{\partial x^k \partial x^j} (X_r) (\nu_r)_a^j (\nu_r)_a^k \right) dr + \int_s^t \frac{\partial F^i}{\partial x^j} (X_r) (\nu_r)^j_a dB^a_r + \\
    &\quad + \frac{\partial F^i}{\partial x^j} (X_s) (\phi_r)^j_a \delta W_{s,t}^a + \left(\frac{\partial F^i}{\partial x^j} (X_s) (\phi_r')^j_{ab}  + \frac{\partial^2 F^i}{\partial x^k \partial x^j} (X_s) (\phi_r)^j_a (\phi_r)^k_b \right)\mathbb{W}_{s,t}^{ab} + F^{\natural,i}_{s,t},
\end{aligned}
\end{equation*}
where $\|F^{\natural,i}_{s,t}\|_p \lesssim_K |t-s|^{\alpha + \gamma}$ and $\|\E(F^{\natural,i}_{s,t} \mid \mathcal{F}_s)\|_p \lesssim_K |t-s|^{\alpha + \gamma + \min\{(\kappa-2)\alpha,\gamma,\gamma'\}}$ for any $p \in [2,\infty)$. 
\end{prop}

\begin{thm} \label{thm:uniqueness}
Let $\mathbf{W}=(W,\mathbb{W}) \in \mathscr{C}^\alpha_g([0,T];\R^n)$, with $\alpha \in (\frac{1}{3},\frac{1}{2}]$.
Let $\kappa,\lambda \in (\frac{1}{\alpha},3]$, $\theta \in (1,2]$ and $\delta,\eta \in [0,\alpha]$ be such that \begin{equation*}
    \alpha + \min\{\delta,\eta\} > \frac{1}{2} \qquad \text{and} \qquad \alpha + \min\{\delta,\eta\} + \min\{(\min\{\kappa,\lambda\}-2)\alpha,(\theta-1)\alpha,\delta,\eta\}>1.
\end{equation*}
Let $b,\sigma$ be two deterministic bounded Lipschitz vector fields from $\R^d$ to $\R^d$ and $\R^{d \times m}$, respectively, and let $(\beta,\beta') \in \mathscr{D}_W^{2\delta}C_b^\kappa(\R^d;\R^{d \times n})$ be a deterministic controlled vector field.
Let $c:[0,T]\times \R^d \to \R$ be a bounded Lipschitz vector field and let $(\gamma,\gamma') \in \mathscr{D}_{W}^{2\eta}C^\theta_b(\R^d;\R)$. 
Let $v$ be a solution to $-dv_t = L_t v_t \, dt + \Gamma_t v_t \, d\mathbf{W}_t$ in the sense of \cref{def:solutiontobacwardroughPDEs}, such that $v \in \mathcal{B}([0,T];C^\lambda_b(\R^d;\R)) \cap C^{\alpha'}([0,T];C^2_b(\R^d;\R))$ for some $\alpha' \in (0,1]$.
Then, for any $(s,x) \in [0,T] \times \R^d$, \begin{equation*}
    v(s,x) = \E \left( g(X_{T-s}^{s,x}) e^{\int_0^{T-s} c(s+r,X^{s,x}_r) dr + \int_0^{T-s} (\gamma,\gamma')(s+r,X^{s,x}_r) d\mathbf{W}^s_r} \right),
\end{equation*}
where $(\gamma,\gamma')(s+r,X^{s,x}_r) := (\gamma(s+r,X^{s,x}), D_x\gamma(s+r,X^{s,x}_r) \beta(s+r,X^{s,x}_r) + \gamma'(s+r,X^{s,x}))$,  $\mathbf{W}^s$ is defined as in \cref{prop:shiftingaroughpath} and $X^{s,x}$ denotes the solution of \begin{equation} \label{eq:roughSDE_uniqueness}
        X^{s,x}_t = x + \int_0^t b (s+r, X_r^{s,x}) dr + \int_0^t \sigma (s+r, X_r^{s,x}) dB_r + \int_0^t (\beta,\beta') (s+r, X_r^{s,x}) d\mathbf{W}^s_r
        \end{equation} 
on any given complete filtered probability space $(\Omega,\mathcal{F},\{\mathcal{F}_t\}_t,\mathbb{P})$ endowed with an $m$-dimensional Brownian motion $B$.
\end{thm}

\begin{proof}
    Notice that, under the assumptions of the theorem, equation \eqref{eq:roughSDE_uniqueness} is well posed and it is possible to define $I^{s,x}$ as in \eqref{eq:definitionofI}.
    Let $s \in [0,T]$ be fixed and, for any $(t,x) \in [s,T] \times \R^d$, let us define the following auxiliary function \begin{equation*}
        w_s(t,x) := \E\left( v_t(X_{t-s}^{s,x}) e^{\int_0^{t-s} c(s+r,X^{s,x}_r) dr + \int_0^{t-s} (\gamma,\gamma')(s+r,X^{s,x}_r) d\mathbf{W}^s_r} \right).
    \end{equation*}
    We show that, for any $x \in \R^d$, the mapping $[s,T] \ni t \mapsto w_s(t,x)$ is constant. 
    This would prove the assertion, since \begin{equation*} \begin{aligned}
        v_s(x) &= w_s(s,x) = w_s(T,x)  
        = \E \left( g(X_{T-s}^{s,x}) e^{\int_0^{T-s} c(s+r,X^{s,x}_r) dr + \int_0^{T-s} (\gamma,\gamma')(s+r,X^{s,x}_r) d\mathbf{W}^s_r} \right).
    \end{aligned}
    \end{equation*}
    Let $(t_1,t_2) \in \Delta_{[s,T]}$. We can write \begin{equation} \label{eq:uniqueness_keyidentity} \begin{aligned}
        w_s(t_2,x) - w_s(t_1,x) &= \E\left(v_{t_2}(X_{{t_2}-s}^{s,x}) e^{I^{s,x}_{t_2-s}}\right) - \E\left(v_{t_1}(X_{{t_1}-s}^{s,x})  e^{I^{s,x}_{t_1-s}}\right)  = \\
        &= \E \left(v_{t_2}(X^{s,x}_{t_2-s}) e^{I^{s,x}_{t_2-s}} - v_{t_2}(X^{s,x}_{t_1-s}) e^{I^{s,x}_{t_1-s}} \right) + \E\left( \left(v_{t_2}(X^{s,x}_{t_1-s}) - v_{t_1}(X^{s,x}_{t_1-s}) \right) e^{I^{s,x}_{t_1-s}} \right).
    \end{aligned}
    \end{equation}
    We can apply \cref{prop:roughItoformula} to with $F=v_{t_2}$ in a very similar way to what we did in \eqref{eq:roughItoformula_existence} (with the only difference that we work on a generic interval $[t_1,t_2]$ and we do not have $e^{I^{s,x}_{t_1}}=1$ anymore) to obtain that \begin{equation*}
        \begin{aligned}
            & v_{t_2}(X^{s,x}_{t_2-s}) e^{I^{s,x}_{t_2-s}} - v_{t_2}(X^{s,x}_{t_1-s}) e^{I^{s,x}_{t_1-s}} = \\
            &= \int_{t_1}^{t_2} L_r v_{t_2} (X_{r-s}^{s,x}) e^{I^{s,x}_{r-s}} dr + \int_{t_1-s}^{t_2-s} \frac{\partial v_{t_2}}{\partial x^j} (X_r^{s,x}) \sigma_a^j(s+r,X_r^{s,x}) e^{I^{s,x}_{r}} dB_r^a + (\Gamma_{t_2})_\mu v_{t_2} (X_{t_1-s}^{s,x}) e^{I^{s,x}_{t_1-s}} \delta W^\mu_{t_1,t_2} + \\
            & \quad + ((\Gamma_{t_2})_\mu (\Gamma_{t_2})_\nu - (\Gamma'_{t_2})_{\mu \nu}) v_{t_2} (X_{t_1-s}^{s,x}) e^{I^{s,x}_{t_1-s}} \mathbb{W}_{t_1,t_2}^{\mu \nu} + Q^\natural_{t_1-s,t_2-s},
        \end{aligned}
    \end{equation*}
    where by construction $|\E(Q^\natural_{t_1-s,t_2-s})| \lesssim |t_2-t_1|^{\alpha + \min\{\delta, \eta\} + \min\{(\kappa-2)\alpha, (\lambda-2)\alpha, (\theta-1)\alpha, \delta , \eta\}}$. 
    Moreover, being $v$ a solution to $-dv_t = L_t v_t \, dt + \Gamma_t v_t \, d\mathbf{W}_t$ it follows that \begin{equation*}
        \begin{aligned}
            v_{t_2}(X^{s,x}_{t_1-s}) - v_{t_1}(X^{s,x}_{t_1-s}) &= -\int_{t_1}^{t_2} L_r v_r (X_{t_1-s}^{s,x}) dr - (\Gamma_{t_2})_\mu v_{t_2} (X_{t_1-s}^{s,x}) \delta W^\mu_{t_1,t_2} + \\
            &\quad - ((\Gamma_{t_2})_\mu (\Gamma_{t_2})_\nu - (\Gamma_t')_{\mu \nu}) v_{t_2} (X_{t_1-s}^{s,x}) \mathbb{W}_{t_1,t_2}^{\mu \nu} + v_{t_1,t_2}^\natural(X^{s,x}_{t_1-s}) \qquad \mathbb{P}\text{-a.s.},
        \end{aligned}
    \end{equation*}
    and $|\E(v_{t_1,t_2}^\natural(X^{s,x}_{t_1-s}))| \le \sup_{x \in \R^d} |v^\natural_{t_1,t_2}(x)| \lesssim |t_s-t_1|^{\alpha + \min\{\alpha + \alpha'\}+\alpha''}$ for some $\alpha',\alpha''$ as in \cref{def:solutiontobacwardroughPDEs}. 
    Hence, from \eqref{eq:uniqueness_keyidentity} and due to the martingale property of the It\^o integral, it follows that \begin{equation*}
        \begin{aligned}
            w_s(t_2,x) - w_s(t_1,x) &= \E \left( \int_{t_1}^{t_2} L_r v_{t_2} (X_{r-s}^{s,x}) e^{I^{s,x}_{r-s}} - L_r v_r (X_{t_1-s}^{s,x}) e^{I^{s,x}_{t_1-s}} dr \right) 
            + w^\natural_{t_1,t_2}
        \end{aligned}
    \end{equation*} 
    where $|w^\natural_{t_1,t_2}|\lesssim |t_2-t_1|^{1+\varepsilon }$ for some $\varepsilon >0$. 
    Arguing very similarly to what we did in \eqref{eq.Lebesgueintegralisalittleo}, we can conclude that, for any $x \in \R^d$ and for any $s \le t_1 \le t_2 \le T$, $|v_s(t_2,x)-v_s(t_1,x)| \lesssim |t_2-t_1|^{1+\varepsilon}$ 
    for some $\varepsilon >0$, which implies that $v_s(\cdot,x)$ is constant. 
\end{proof}

\subsection{The exponential term} \label{section:exponentialterm}

The presence of the stochastic exponential term $$e^{\int_0^{T-s} c(s+r,X^{s,x}_r) dr + \int_0^{T-s} (\gamma(s+r,X^{s,x_r}),D_x\gamma(s+r,X_r^{s,x})\beta(s+r,X^{s,x}_r) d\mathbf{W}^s_r}$$ in the representation of the solution $u$ given in \eqref{eq:definitionofu_mainresult} requires particular attention. 
The two main tools that allow us to treat such a random variable and to prove meaningful results are the John-Nirenberg inequality proved in \cite[Proposition 2.7]{FHL21} and the mean value theorem, according to which
\begin{equation*}
        |e^y - x^x| \le e^{|x|+|y|}  \ |y-x| \qquad \text{for any $x,y \in \R$}. 
    \end{equation*}

\begin{prop} \label{prop:continuityanddifferentiabilityofexponentials}
    \begin{enumerate}
        \item Let $(\mathcal{U},\rho)$ be a metric space and let $o \in U \subseteq \mathcal{U}$ be a subset.
        Let $\{X^\zeta\}_{\zeta \in U}$ be a $\mathscr{L}(\R)$-continuous family of random variables such that $\sup_{\rho(\zeta,o) \le R} \|e^{|X^\zeta|}\|_p < +\infty$ for any $p \in [1,\infty)$ and for any $R>0$. Then $\{e^{X^\zeta}\}_{\zeta \in U}$ is $\mathscr{L}(\R)$-continuous. 
        \item Let $(\mathcal{U},|\cdot|)$ be a normed vector space and let $U \subseteq \mathcal{U}$ be a non-empty open subset. 
        Let $k \in \mathbb{N}_{\ge 1}$ and let $\{X^\zeta\}_{\zeta \in \R^d}$ be a family of $k$-times ($\mathscr{L}(\R)$-continuously) $\mathscr{L}(\R)$-differentiable random variables such that $\sup_{|\zeta| \le R} \|e^{|X^\zeta|}\|_p < +\infty$ for any $p \in [1,\infty)$ and for any $R>0$.
        Then $\{e^{X^\zeta}\}_{\zeta \in U}$ is $k$-times ($\mathscr{L}(\R)$-continuously) $\mathscr{L}(\R)$-differentiable. 
        In particular, for any $\zeta \in U$ and for any $l \in \mathbb{S}(\mathcal{U})$,  \begin{equation*}
            \mathscr{L}\text{-} \frac{\partial}{\partial l} e^{X^\zeta} = \left(\mathscr{L}\text{-} \frac{\partial}{\partial l} X^\zeta  \right) e^{X^\zeta}.
        \end{equation*}
    \end{enumerate}
\end{prop}

\begin{proof} 
    \textit{1. } Let $\zeta^0 \in U$ be an accumulation point, and let $(\zeta^n)_{n \ge 1} \subseteq U \setminus \{\zeta^0\}$ be any sequence such that $\rho(\zeta^n,\zeta^0) \to 0$ as $n \to +\infty$. 
    For any $p \in [1,\infty)$ and for any $n \ge 1$, we can apply the mean value theorem to get  \begin{equation*}
        \begin{aligned}
            \big\|e^{X^{\zeta^n}} - e^{X^{\zeta^0}} \big\|_p 
            \le \left(\sup_{n \ge 0} \|e^{|X^{\zeta^n}|}\|^2_{4p}\right) \|X^{\zeta^n}-X^{\zeta^0}\|_{2p}.
        \end{aligned}
    \end{equation*}
    The right-hand side goes to zero as $n \to +\infty$ assumption, and the assertion is therefore proved. \\
    
    \textit{2. } We argue recursively on $k$. First, assume $k=1$.
    Let $\zeta^0 \in U$ and let $l \in \mathbb{S}(\mathcal{U})$ be fixed. For any $\varepsilon \in \R \setminus\{0\}$ and for any $p \in [1,\infty)$, we can apply the mean-value theorem to get that \begin{equation*}
        \begin{aligned}
            \left\| \frac{1}{\varepsilon} \left(e^{X^{\zeta^0+\varepsilon l}} - e^{X^{\zeta^0}}\right)  - \left(\mathscr{L}\text{-} \frac{\partial}{\partial l} X^{\zeta^0}  \right) e^{X^{\zeta^0}}  \right\|_p &\le \left\| \frac{1}{\varepsilon} \left({X^{\zeta^0+\varepsilon l}} - {X^{\zeta^0}} - \mathscr{L}\text{-} \frac{\partial}{\partial l} X^{\zeta^0}  \right) e^{X^{\zeta^0}}  \right\|_p + \\
            & \quad + \left\| \int_0^1 e^{(1-\theta)X^{\zeta^0} + \theta X^{\zeta^0}} d\theta \ \frac{1}{\varepsilon} (X^{\zeta^0+\varepsilon l} - X^{\zeta^0}) (X^{\zeta^0+\varepsilon l} - X^{\zeta^0}) \right\|_p.
        \end{aligned}
    \end{equation*}  
    As $\varepsilon \to 0$, the first term on the right-hand side goes to 0 by the $\mathscr{L}$-differentiability assumption on $\{X^\zeta\}_\zeta$ and due to the locally uniform boundedness of $e^{|X^\zeta|}$ in $L^p$-norm. 
    The last term on the right-hand side goes to 0 as well, taking also into account the $\mathscr{L}$-continuity of $\{e^{X^{\zeta}}\}_\zeta$ showed in \textit{1}.
    The $\mathscr{L}$-continuity of $\{(\mathscr{L}\text{-} \frac{\partial}{\partial l} X^\zeta) e^{X^\zeta}\}_{\zeta \in U}$ is an easy consequence of the $\mathscr{L}$-continuity of $\{X^\zeta\}_{\zeta \in U}$. \\
    Let us now assume that the assertion is proved for $k=k_0$ and that all the assumptions hold for $k=k_0+1$. 
    Let us show that, for any $l \in \mathbb{S}(\mathcal{U})$, $\{(\mathscr{L}\text{-} \frac{\partial}{\partial l} X^\zeta) e^{X^\zeta}\}_{\zeta \in U}$ is $k_0$-times ($\mathscr{L}$-continuously) $\mathscr{L}$-differentiable. 
    By assumption, $\{\mathscr{L}\text{-} \frac{\partial}{\partial l} X^\zeta\}_{\zeta \in U}$ is $k_0$-times ($\mathscr{L}$-continuously) $\mathscr{L}$-differentiable, while by recursion we also have that $\{ e^{X^\zeta}\}_{\zeta \in U}$ is $k_0$-times ($\mathscr{L}$-continuously) $\mathscr{L}$-differentiable. 
    Hence, to conclude it is sufficient to apply \cref{prop:stochasticlinearvectorfields_properties} with $g(x,y)=xy$. 
\end{proof}

\begin{prop} \label{prop:propertiesoftheexponential}
    Let $\mathbf{W}=(W,\mathbb{W}) \in \mathscr{C}_g^\alpha([0,T];\R^n)$ be a geometric rough path, with $\alpha \in (\frac{1}{3},\frac{1}{2}]$. 
    Let $\mu \in \mathscr{L}(\R)$ and let $(\phi,\phi') \in \mathbf{D}_W^{\gamma,\gamma'}L_{1,\infty}(\mathcal{L}(\R^d;\R))$, for some $\gamma \in [0,\alpha], \ \gamma' \in [0,1]$ such that $\alpha + \gamma > \frac{1}{2}$ and $\alpha + \gamma + \gamma' >1$.
    Define \begin{equation*}
        (I, I') := \left( \int_0^\cdot \mu_r dr + \int_0^\cdot (\phi_r,\phi_r') d\mathbf{W}_r, \phi \right). 
    \end{equation*} 
    \begin{enumerate}
        \item For any $p \in [1,\infty)$, \begin{equation*}
            \E \left( e^{p \cdot \sup_{t \in [0,T]} |I_t|} \right) < +\infty
        \end{equation*}
        In particular, $e^{I_t} \in L^p(\Omega;\R)$ for any $p \in [1,\infty)$ and uniformly in $t \in [0,T]$. 
        \item $(e^I, \phi \ e^I)$ is a rough It\^o process, namely $(\phi e^I, (\phi \phi + \phi') e^I) \in \mathbf{D}_W^{\alpha,\gamma}\mathscr{L}(\R)$ and, for any $s \le t$, \begin{equation*} \begin{aligned}
         e^{I_t} - e^{I_s} &= \int_s^t \mu_r e^{I_r} dr +  (\phi_s)_\kappa e^{I_s} \delta W^\kappa_{s,t} + ((\phi_s)_\kappa (\phi_s)_\lambda + (\phi'_s)_{\kappa \lambda} )e^{I_s}\mathbb{W}^{\kappa,\lambda}_{s,t} + e^\natural_{s,t},
    \end{aligned}
    \end{equation*}
    where $\|e^\natural_{s,t}\|_p \lesssim |t-s|^{\frac{1}{2}+\varepsilon}$ and $\|\E_s(e^\natural_{s,t})\|_p \lesssim |t-s|^{1+\varepsilon}$, for some $\varepsilon > 0$ and for any $p \in [2,\infty)$.
    \end{enumerate}
\end{prop}

\begin{proof}
\textit{1. } Notice that, by construction, $\|\delta I\|_{\alpha;1,\infty} < +\infty$ and $I$ is $\mathbb{P}$-a.s.\ continuous, with $I_0 =0$.
Hence, from \cite[Proposition 2.7]{FHL21}, \begin{equation*}
    \E\left( e^{\lambda \sup_{t \in [0,T]} |\delta I_{0,t}|} \right) \le 2^{1 + (22\lambda \|\delta I\|_{\alpha;1,\infty})^\frac{1}{\alpha} T} \qquad \text{for any $\lambda>0$}.
\end{equation*}
\textit{2. } The proof of this fact is analogous to the proof of the rough It\^o formula in \cref{prop:roughItoformula}. 
Indeed, in that context the polynomial growth assumptions on $F$ and its derivatives can be replaced by the more general assumption $F(X_t),D_xF(X_t),\dots,D^{\lfloor \kappa \rfloor}_{x\dots x} F(X_t) \in L^p$ for any $p \in [2,\infty)$ and for any $t \in [0,T]$. 
For sake of completeness, we highlight the key points of the argument, which essentially relies on the mean value theorem. 
Indeed, for any $s \le t$ we can write \begin{equation*}
    |e^{I_t} - e^{I_s}| \le e^{|I_t|+|I_s|} (I_t-I_s).
\end{equation*}
From the fact that \begin{equation*}
    \begin{aligned}
        \phi_t e^{I_t} - \phi_s e^{I_s} - (\phi_s\phi_s + \phi'_s)e^{I_s} \delta W_{s,t} &= \phi_s(e^{I_t}-e^{I_s}-\delta \phi_{s,t} e^{\phi_s}) + \phi_s R^{\phi}_{s,t} e^{I_s}  + \delta \phi_{s,t} (e^{I_t}-e^{I_s}) + R^{\phi}_{s,t} e^{I_s} 
    \end{aligned}
\end{equation*}
one can easily deduce that $(\phi e^I, (\phi \phi + \phi') e^I) \in \mathbf{D}_W^{\alpha,\gamma}\mathscr{L}(\R)$. 
Moreover, \begin{equation*}
    \begin{aligned}
        e^{I_t} - e^{I_s} &= e^{I_s} \left(\int_s^t \mu_r dr + \int_s^t (\phi_r,\phi_r') d\mathbf{W}_r\right) + \frac{1}{2} e^{I_s} \left(\int_s^t (\phi_r,\phi_r') d\mathbf{W}_r \right)^2 + e^\natural_{s,t}
    \end{aligned}
\end{equation*}
and the conclusion follows arguing as in the proof of \cref{prop:roughItoformula}.
\end{proof}

\subsection{Functionals of solutions to rough SDEs} \label{section:functionalsandmarkovproperty}
According to \cref{thm:mainresult_wellposednessroughPDEs}, the solution $u$ to the linear bacward rough PDE \eqref{eq:roughPDE} admits a functional integral representation in terms of the solution to a suitable rough SDE. 
In \cref{prop:independenceontheBrownianmotion} we point out how such a representation does not depend on the underline Brownian motion.
In \cref{lemma:markovpropertyforRSDEs} we prove a useful identity that connects $u(s,\cdot)$ with $u(t,\cdot)$.
In doing so we also state and prove some flow properties and some measurability results about solutions to RSDEs (cf. \cref{prop:flowpropertyfroRSDEs} and \cref{lemma:measurabilityofRSDEsintheinitialcondition}, respectively). \\

Let $\mathbf{W}=(W,\mathbb{W}) \in \mathscr{C}^\alpha([0,T];\R^n)$ be a rough path, with $\alpha \in (\frac{1}{3},\frac{1}{2}]$, let $b:[0,T] \times \R^d \to \R^d$ and $\sigma:[0,T] \times \R^d \to \R^{d \times m}$ be two deterministic bounded Lipschitz vector fields, and let $(\beta,\beta') \in \mathscr{D}_W^{2\delta}C_b^\kappa$ be a deterministic controlled vector field, for some $\kappa \in (\frac{1}{\alpha},3]$ and $\delta \in [0,\alpha]$ with $\alpha + 2\delta > 1$. 
Let $t_0 \in [0,T]$.
Recall that a solution to the RSDE on $[t_0,T]$ starting from $\xi$, with coefficients $b,\sigma,(\beta,\beta')$ and driven by $\mathbf{W}$ is a quadruple \begin{equation*}
    (X,\xi,B,\{\mathcal{F}_t\}_{t \in [t_0,T]}),
\end{equation*}
where $\{\mathcal{F}_t\}_{t \in [t_0,T]}$ is a complete filtration on some complete probability space $(\Omega,\mathcal{F},\mathbb{P})$, $\xi$ is a $\mathcal{F}_{t_0}$-measurable $\R^d$-valued random variable, $B=(B_t)_{t \in [t_0,T]}$ is an $m$-dimensional Brownian motion on $(\Omega,\mathcal{F},\{\mathcal{F}_t\}_{t \in [t_0,T]},\mathbb{P})$ and $X$ is a continuous $\{\mathcal{F}_t\}_t$-adapted process solving  
\begin{equation*}
        X_t = \xi + \int_{t_0}^t b_{r}(X_r) dr + \int_{t_0}^t \sigma_{r} (X_r) dB_r + \int_{t_0}^t (\beta_{r},\beta'_r)(X_r) d\mathbf{W}_r
\end{equation*} 
in the sense of \cref{def:solutionRSDEs_compendium}, where $(\beta_{r},\beta'_r)(X_r) := (\beta_{r}(X_r), D_x\beta_{r}(X_r) \beta_{r}(X_r) + \beta'_r(X_r))$. 
In such a case we write $X \in RSDE(\xi,B,\{\mathcal{F}_t\}_{t \in [t_0,T]})$.
\begin{prop} \label{prop:independenceontheBrownianmotion}
    Let $c: [0,T] \times \R^d \to \R$ be bounded and continuous, and let $(\gamma,\gamma') \in \mathscr{D}_W^{2\eta}C_b^\lambda(\R^d;\R)$, for some $\eta \in [0,\alpha]$ and $\lambda \in (1,2]$ such that $\alpha + \eta > \frac{1}{2}$ and $\alpha + \eta + \min\{(\lambda-2)\alpha,\eta,\delta\}>1$.
    Let $(X,\xi,B,\{\mathcal{F}_t\}_{t \in [t_0,T]})$ and $(\bar{X},\bar{\xi},\bar{B},\{\bar{\mathcal{F}}_t\}_{t \in [t_0,T]})$ be two solutions on $(\Omega,\mathcal{F},\mathbb{P})$ and $(\bar{\Omega},\bar{\mathcal{F}},\bar{\mathbb{P}})$, respectively, and assume $Law(\xi)=Law(\bar{\xi})$.
    Then, for any non-negative measurable function $g: \R^d \to \R_{\ge 0}$ \footnote{By considering positive and negative parts, the result can be easily extended to any measurable function $g:\R^d \to \R$ satisfying \begin{equation*}
        \E\left( \Big| g(X_T) e^{\int_{t_0}^T c_r(X_r) dr + \int_{t_0}^T (\gamma_r,\gamma'_r) (X_r) d\mathbf{W}_r} \Big| \right), \  \bar{\E} \left( \Big| g(\bar{X}_T) e^{\int_{t_0}^T c_r(\bar{X}_r) dr + \int_{t_0}^T (\gamma_r,\gamma'_r) (\bar{X}_r) d\mathbf{W}_r} \Big| \right) < +\infty.
    \end{equation*}}, \begin{equation*}
        \E\left( g(X_T) e^{\int_{t_0}^T c_r(X_r) dr + \int_{t_0}^T (\gamma_r,\gamma'_r) (X_r) d\mathbf{W}_r} \right) = \bar{\E} \left( g(\bar{X}_T) e^{\int_{t_0}^T c_r(\bar{X}_r) dr + \int_{t_0}^T (\gamma_r,\gamma'_r) (\bar{X}_r) d\mathbf{W}_r} \right),
    \end{equation*}
    where $\E$ and $\bar{\E}$ denote the expectation on $(\Omega,\mathcal{F},\mathbb{P})$ and $(\bar{\Omega},\bar{\mathcal{F}},\bar{\mathbb{P}})$, respectively. 
\end{prop}

\begin{proof}
    In case $(\gamma,\gamma')=0$, the conclusion is a trivial application of the uniqueness-in-law result \cite[Theorem 4.21]{FHL21}. 
    The presence of the rough stochastic exponential $e^{\int_{t_0}^T(\gamma_r,\gamma'_r)(X_r) d\mathbf{W}_r}$ requires a bit of analysis, since the pathwise continuity of $X$ is not sufficient to define it as a random variable.  Indeed, we need to observe that $(\gamma(X), D_x\gamma(X) \beta(X))$ is a stochastic controlled rough path in $\mathbf{D}_W^{\alpha,\min\{\eta,\delta\}}\mathscr{L}(\mathcal{L}(\R^n,\R))$ and to take into account \cite[Proposition 2.7]{FHL21} to get exponential integrability.
    For any $M \in \mathbb{N}$, define $$\psi_M(x) := e^x \mathbf{1}_{x \le M} + e^M \mathbf{1}_{x > M}, \ x \in \R. $$ 
    The family $\{\psi_M\}_M$ is by construction an increasing sequence of non-negative continuous functions, converging to $\psi(x):=e^x$ as $M \to +\infty$. 
    By monotone convergence and denoting by $\varphi:C^0([0,T];\R^d) \to \R$ the map $\varphi(y_\cdot) := g(y(T)) e^{\int_{t_0}^T c_r(y(r)) dr}$, we can write \begin{equation*}
        \begin{aligned}
            \E\left(\varphi(X_\cdot) e^{\int_{t_0}^T (\gamma_r,\gamma_r')(X_r) d\mathbf{W}_r} \right) &= \lim_{M \to +\infty} \E\left(\varphi(X_\cdot) \  \psi_M \left(\int_{t_0}^T (\gamma_r,\gamma_r')(X_r) d\mathbf{W}_r\right) \right)
        \end{aligned}
    \end{equation*}
    Let $(\pi^n)_n$ be a sequence of partitions of $[t_0,T]$ with $|\pi^n| \to 0$ as $n \to +\infty$. 
    Up to passing to a subsequence, we have that \begin{equation*}
        \int_{t_0}^T (\gamma_r,\gamma_r')(X_r) d\mathbf{W}_r = \lim_{n \to +\infty} \sum_{[u,v] \in \pi^n} \gamma_u(X_u) \delta W_{u,v} + (D_x\gamma_u(X_u) \beta_u(X_u) + \gamma'_u(X_u)) \mathbb{W}_{u,v} =: \lim_{n \to +\infty} f^n(X_\cdot) 
    \end{equation*}
    $\mathbb{P}$-a.s., where $f^n:C^0([0,T];\R^d) \to \R$ is measurable and bounded.
    By dominated convergence \begin{equation*}
        \begin{aligned}
            \E\left(\varphi(X_\cdot) \  (\psi_M \circ f^n)  \left(X_\cdot \right) \right) \longrightarrow \E\left(\varphi(X_\cdot) \  \psi_M \left(\int_{t_0}^T (\gamma_r,\gamma_r')(X_r) d\mathbf{W}_r\right) \right) \quad \text{ as $n \to +\infty$, for any $M$.}
        \end{aligned}
    \end{equation*}
    By \cite[Theorem 4.21]{FHL21}, we deduce that  $\E\left(\varphi(X_\cdot) \  (\psi_M \circ f^n)  \left(X_\cdot \right) \right) = \E\left(\varphi(\bar{X}_\cdot) \  (\psi_M \circ f^n)  \left(\bar{X}_\cdot \right) \right)$ for any $n,M$. 
    The same argument applies to $\bar{X}_\cdot$, and we obtain that $\E\left(\varphi(\bar{X}_\cdot) \  (\psi_M \circ f^n)  \left(\bar{X}_\cdot \right) \right)$ converges, as $n \to +\infty$, to $\E(\varphi(\bar{X}_\cdot) \  \psi_M (\int_{t_0}^T (\gamma_r,\gamma_r')(\bar{X}_r) d\mathbf{W}_r) )$.
    By passing the latter to the limit as $M \to +\infty$ (using monotone convergence), the assertion is proved. 
\end{proof}

Let $s \in [0,T]$ be a shift parameter.
In a very similar way to what is done at the beginning of the proof of \cref{thm:existenceRPDE}, we can extend the definition of $b_t(\cdot), \sigma_t(\cdot), \beta_t(\cdot)$ and $\beta'_t(\cdot)$ for $t \in [0,+\infty)$. 
A solution $X$ to the "shifted" RSDE on the time interval $[t_0,T]$ and starting from $\xi$ is given by a (complete) filtered probability space $(\Omega,\mathcal{F},\{\mathcal{F}_t\}_{t \in [t_0,T_0]},\mathbb{P})$ endowed with a Brownian motion $B=(B_t)_{t \in [t_0,T]}$, a $\mathcal{F}_{t_0}$-measurable random variable $\xi$ and by a continuous $\{\mathcal{F}_t\}_t$-adapted stochastic process $X$ which solves
\begin{equation*}
        X_t = \xi + \int_{t_0}^t b_{s+r}(X_r) dr + \int_{t_0}^t \sigma_{s+r} (X_r) dB_r + \int_{t_0}^t (\beta_{s+r},\beta'_{s+r})(X_r) d\mathbf{W}^s_r \qquad t \in [t_0,T]
\end{equation*}
in the sense of \cref{def:solutionRSDEs_compendium}. 
Here $\mathbf{W}^s$ is the shifted rough path defined in \cref{prop:shiftingaroughpath}. 
In such a case we write $X \in RSDE(s,\xi,B,\{\mathcal{F}\}_{t \in [t_0,T]})$.
Due to strong uniqueness of rough SDEs under our assumptions on $b,\sigma,\beta$, it is possible to prove the following two flow properties. 
\begin{prop} \label{prop:flowpropertyfroRSDEs}
    Let $s \in [0,T]$ and let $t_0 \in [0,T]$ with be fixed. Let $(\Omega,\mathcal{F},\{\mathcal{F}_t\}_{t \in [t_0,T]},\mathbb{P})$ be a complete filtered probability space, and let $B=(B_t)_{t \in [t_0,T]}$ be an $m$-dimensional Brownian motion on $(\Omega,\mathcal{F},\{\mathcal{F}_t\}_{t \in [t_0,T]},\mathbb{P})$.
    Let $\xi \in L^0(\Omega,\mathcal{F}_{t_0};\R^d)$ and let $X \in RSDE(s,\xi,B, \{\mathcal{F}_t\}_{t \in [t_0,T]})$. 
    \begin{enumerate}
        \item Let $t_1 \in [t_0, T]$ and, for any $t \in [t_1,T_0]$, define $Y_t := X_t$. Then
        $Y \in RSDE(s,X_{t_1},\bar{B},\{\bar{\mathcal{F}}_t\}_{t \in [t_1,T]})$, where $\bar{\mathcal{F}}_t := \mathcal{F}_{t}$ and $\bar{B}_t := B_t - B_{t_1}$ for $t \in [t_1,T_0]$.
        \item For any $t \in [0,T- t_0]$, define $Z_t := X_{t_0+t}$. Then
        $Z \in RSDE(s+t_0,\xi,\tilde{B},\{\tilde{\mathcal{F}}_t\}_{t \in [0,T-t_0]})$, where $\tilde{\mathcal{F}}_t := \mathcal{F}_{t_0+t}$ and $\bar{B}_t := B_{t+t_0} - B_{t_0}$ for $t \in [0,T_0-t_0]$.
    \end{enumerate}
\end{prop}

\begin{proof}
    \textit{1.} \quad Let $t_1 \in [t_0,T]$. For any $t \in [t_1,T]$, $X_t$ is $\mathcal{F}_t$-measurable and it holds that \begin{equation*}
            \begin{aligned}
                X_t 
                &= X_{t_1} + \int_{t_1}^t b_{s+r}(X_r) + \int_{t_1}^t \sigma_{s+r} (X_r) dB^{t_0}_r + \int_{t_1}^t (\beta_{s+r},\beta'_{s+r})(X_r)) d\mathbf{W}^s_r.
            \end{aligned}
        \end{equation*}
        The conclusion follows on noting that $\int_{t_1}^t \sigma_{s+r} (X_r) dB^{t_0}_r = \int_{t_1}^t \sigma_{s+r} (X_r) dB^{t_1}_r $ for any $t \in [t_1,T]$, by the properties of It\^o integration. \\

        \textit{2.} \quad Let $t \in [0,T-t_0]$ be arbitrary. By construction, $Z_t$ is $\mathcal{F}_{t_0+t}$-measurable. 
        Given a partition $\pi$ of $[t_0,t_0+t]$, we denote by $\pi^{t_0}$ the partition of $[0,t]$ obtained by applying to every point the translation map $s \mapsto s-t_0$. It holds that \begin{equation*}
            \sum_{[u,v] \in \pi} \sigma_{s+u}(X_u) (B_v - B_u) = \sum_{[u,v] \in \pi^{t_0}} \sigma_{s+u+t_0}(X_{u+ t_0}) (B_{v+t_0} - B_{t_0} - B_{u+t_0} + B_{t_0})
        \end{equation*} and \begin{equation*}
            \begin{aligned}
                &\sum_{[u,v] \in \pi} \beta_{s+u} (X_u) \delta W^s_{u,v} + (D_x \beta_{s+u} (X_u) \beta_{s+u} (X_u) + \beta'_{s+u}(X_u)) \mathbb{W}_{u,v}^s =  \\
                &= \sum_{[u,v] \in \pi^{t_0}} \beta_{s+u+t_0} (X_{u+t_0}) \delta W^s_{u+t_0,v+t_0} + (D_x \beta_{s+u+t_0} (X_{u+t_0}) \beta_{s+u+t_0} (X_{u+t_0}) + \beta'_{s+u+t_0}(X_{u+t_0})) \mathbb{W}_{u+t_0,v+t_0}^s
            \end{aligned}
        \end{equation*}
        Passing to the limit in probability the two previous identities, it holds that  \begin{equation*}
            \begin{aligned}
                Y_t &= X_{t_0+t}=  \xi + \int_{t_0}^{t_0+t} b_{s+r}(X_r) + \int_{t_0}^{t_0+t} \sigma_{s+r} (X_r) dB^{t_0}_r + \int_{t_0}^{t_0+t} (\beta_{s+r},\beta'_{s+r})(X_r) d\mathbf{W}^s_r = \\
                &= \xi + \int_{0}^{t} b_{s+t_0+r}(Y_r) + \int_{0}^{t} \sigma_{s+t_0+r} (Y_r) d\tilde{B}_r + \int_{0}^{t} (\beta_{s+t_0+r},\beta'_{s+t_0+r})(Y_r) d\mathbf{W}^{s+t_0}_r. 
            \end{aligned}
        \end{equation*}
\end{proof}

\begin{lemma} \label{lemma:measurabilityofRSDEsintheinitialcondition}
  Let $s, t_0 \in [0,T]$ be fixed.
  For any $x \in \R^d$, let $X^{x} \in RSDE(s,x,B,\{\mathcal{F}_t\}_{t \in [t_0,T]})$ be a solution to the rough SDE on $(\Omega,\mathcal{F},\mathbb{P})$, starting from $x$. 
  For any $t \in [t_0,T]$, let us define \begin{equation} \label{eq:definitionoffiltration}
      \mathcal{F}_t^{t_0} := \sigma (B_r - B_{t_0}, r \in [t_0, t]).
  \end{equation}
  \begin{enumerate}
    \item For any $t \in [t_0, T]$, define $f_t : \Omega \times
    \mathbb{R}^d \rightarrow \mathbb{R}^d$ to be the map \begin{equation*}
        f_t (\omega, x) := X_t^{x} (\omega).
    \end{equation*}
    Then $f_t$ is $(\mathcal{F}_t^{t_0}
    \otimes \mathcal{B} (\mathbb{R}^d), \mathcal{B}
    (\mathbb{R}^d))$-measurable, for any $t \in [t_0, T]$.
    \item Let $(\gamma,\gamma') \in \mathscr{D}_W^{2\eta}C_b^\lambda(\R^d;\R)$ be any deterministic controlled vector field, for some $\eta \in [0,\alpha]$ and $\lambda \in (1,2]$ such that $\alpha + \eta > \frac{1}{2}$ and $\alpha + \eta + \min\{(\lambda-2)\alpha,\eta,\delta\}>1$.
    For any $t \in [t_0, T]$, let $I^t : \Omega \times
    \mathbb{R}^d \rightarrow \mathbb{R}$ be the map defined as
    \[ I^t (\omega, x) := \left( \int_{t_0}^t (\gamma_{s + r}(X_r^{x}), D_x \gamma_{s+r}(X_r^{x}) \beta_{s+r}(X^x_r) + \gamma'_{s+r}(X_r^{x}))
        d \mathbf{W}_r^s \right) (\omega) . \]
    Then $I^t$ is $(\mathcal{F}_t^{t_0}
    \otimes \mathcal{B} (\mathbb{R}^d), \mathcal{B} (\mathbb{R}))$-measurable, for any $t \in [t_0, T]$.
  \end{enumerate}
\end{lemma}

\begin{proof} 
    \textit{1.} \  Let $s,t_0 \in [0,T]$ be fixed. Recall that $X^{x}$ can always be seen as the limit of
    suitable Picard-type approximations (cf.\ \cite[Proof of Theorem 4.6]{FHL21}).
    We denote by $X^{N,x}$ the $N$-th
    Picard's iteration and, for any $t \in [t_0,T]$, we define $f^{0}_t(\omega,x) := X^{0,x} = x$ and \begin{align*}
        f_t^{N+1}(\omega,x) &:= X^{N+1,x} = x + \int_{t_0}^t b_{s + r} (f^N_r(\omega,x)) dr  + \left( \int_{t_0}^t \sigma_{s + r} (f^{N}_r(\cdot,x)) dB_r\right) (\omega) + \\
        & \quad + \left( \int_{t_0}^t (\beta_{s+r}(f^N(\cdot,x)), D_x \beta_{s+r}(f^N(\cdot,x)) \beta_{s+r} (f^{N}(\cdot,x)) + \beta'_{s+r}(f^N(\cdot,x)) ) d\mathbf{W}^s_r \right)(\omega) \qquad N \ge 1
    \end{align*}
    By induction on $N$, we prove that $f_t^N$ is $(\mathcal{F}_t^{t_0}\otimes \mathcal{B} (\mathbb{R}^d), \mathcal{B}(\mathbb{R}^d))$-measurable, for any $t \in [t_0,T]$.
    This allows us prove the assertion, being by
    construction
    \[ f_t (\omega, x) = \lim_{N \to + \infty} f^N_t(\omega, x) . \]
    Measurability is trivial in the case $N = 0$.
    Assume that, for a certain $N_0 \in \mathbb{N}$ and for any $t \in [t_0, T]$, $f^{N_0}_t$ is $(\mathcal{F}_t^{t_0} \otimes \mathcal{B} (\mathbb{R}^d), \mathcal{B}(\mathbb{R}^d))$-measurable.
    Notice that we can write \begin{align*}
        &\int_{t_0}^t (\beta_{s+r}(f_r^{N_0}(\cdot,x)), D_x \beta_{s+r}(f_r^{N_0}(\cdot,x)) \beta_{s+r} (f_r^{N_0}(\cdot,x)) + \beta_{s+r}'(f_r^{N_0}(\cdot,x)) ) d\mathbf{W}^s_r = \\
        & \quad = \lim_{|\pi| \to 0} \sum_{[u, v] \in \pi} \beta_{s+u} (f_u^{N_0}(\cdot,x)) \delta W^s_{u, v} + (D_x \beta_{s + u}
       (f_u^{N_0}(\cdot,x)) \beta_{s + u} (f_u^{N_0}(\cdot,x)) + \beta_{s+u}'(f_u^{N_0}(\cdot,x)))  \mathbb{W}^s_{u, v},
    \end{align*}
    where the limit is in probability. Similarly \begin{align*}
        \int_{t_0}^t \sigma_{s + r} (f^{N_0}_r(\cdot,x)) dB_r &= \lim_{| \pi | \to 0} \sum_{[u, v] \in \pi} \sigma_{s + u}
       (f_u^{N_0}(\cdot,x)) (B_v - B_u) \\
       \int_{t_0}^t b_{s + r} (f^{N_0}_r(\cdot,x)) dr &= \lim_{|\pi | \to 0} \sum_{[u, v] \in \pi} b_{s + u} (f_u^{N_0}(\cdot,x)) (v - u)
    \end{align*}
    Hence, we conclude that that $f_t^{N_0+1}$ is $(\mathcal{F}_t^{t_0} \otimes
    \mathcal{B} (\mathbb{R}^d), \mathcal{B} (\mathbb{R}^d))$-measurable. \\
    
    \textit{2.} \ This fact trivially comes from the previous one.
    Indeed, we
    have that, for any $t \in [t_0, T]$, $I^t (\omega, x)$ is
    approximated (in probability) by a sequence of Riemann-type sums of the
    form \begin{equation*}
        \sum_{[u, v] \in \pi} \gamma_{s+u} (f_u(\omega,x)) \delta W^s_{u,v} + (D_x \gamma_{s+u} \left( {f_u(\omega,x)} \right) \beta_{s + u}(f_u(\omega,x)) + \gamma'_{s+u} (f_u(\omega,x))) \mathbb{W}^s_{u, v} .
    \end{equation*}
\end{proof}

\begin{prop}\label{lemma:markovpropertyforRSDEs}
    Let $c$ and $(\gamma,\gamma')$ as in \cref{prop:independenceontheBrownianmotion}.
    Let $g:\R^d \to \R$ be a bounded Borel-measurable function.
    For any $(s,x) \in [0,T] \times \R^d$, define \begin{equation*}
        u(s,x) := \E \left( g(X^{s,x}_{T-s}) e^{\int_0^{T-s} c_{s+r}(X_r^{s,x}) dr + \int_0^{T-s} (\gamma_{s+r},\gamma'_{s+r})(X_r^{s,x}) d\mathbf{W}_r^s} \right),
    \end{equation*} 
    where $X^{s,x} \in RSDE(s,x,B,\{\mathcal{F}_t\}_{t \in [0,T]})$, for any given complete filtered probability space $(\Omega,\mathcal{F},\{\mathcal{F}_t\}_{t \in [0,T]},\mathbb{P})$ endowed with an $m$-dimensional Brownian motion $B=(B_t)_{t \in [0,T]}$. 
   Then, for any $s \le t$, \begin{equation*}
        u(s,x) = \E \left( e^{\int_0^{t-s} c_{s+r}(X_r^{s,x}) dr + \int_0^{t-s} (\gamma_{s+r},\gamma'_{s+r})(X_r^{s,x}) d\mathbf{W}_r^s} \  u(t,X^{s,x}_{t-s}) \right).
    \end{equation*}
\end{prop}

\begin{proof}
    Let $s \le t$ and let $x \in \R^d$ be fixed. 
    For $t_0 \in [0,T]$ and $\tau \in [t_0,T]$, define $\mathcal{F}^{t_0}_\tau := \sigma(B_r - B_{t_0}, \ r \in [t_0,\tau])$.
    Properties of the exponential function and a simple time-change in the integrals show that \begin{equation*} \begin{aligned}
        &e^{\int_0^{T-s} c_{s+r}(X_r^{s,x}) dr + \int_0^{T-s} (\gamma_{s+r},\gamma'_{s+r})(X_r^{s,x}) d\mathbf{W}_r^s} = \\
        &= e^{\int_0^{t-s} c_{s+r}(X_r^{s,x}) dr + \int_0^{t-s} (\gamma_{s+r},\gamma'_{s+r})(X_r^{s,x}) d\mathbf{W}_r^s} \cdot e^{\int_{0}^{T-t} c_{t+r}(X_{(t-s)+r}^{s,x}) dr + \int_{0}^{T-t} (\gamma_{t+r},\gamma'_{t+r})(X_{(t-s)+r}^{s,x}) d\mathbf{W}_r^s}.
    \end{aligned}
    \end{equation*}
    By \cref{lemma:measurabilityofRSDEsintheinitialcondition} we get that $e^{\int_0^{t-s} c_{s+r}(X_r^{s,x}) dr + \int_0^{t-s} (\gamma_{s+r},\gamma'_{s+r})(X_r^{s,x}) d\mathbf{W}_r^s}$ is $\mathcal{F}_{t-s}^0$-measurable.
    By the tower property of conditional expectations, we have \begin{equation*}
        \begin{aligned}
            &u(s,x) = \\
            &= \E \left(e^{\int_0^{t-s} c_{s+r}(X_r^{s,x}) dr + \int_0^{t-s} (\gamma_{s+r},\gamma'_{s+r})(X_r^{s,x}) d\mathbf{W}_r^s} \E \left( g(Y_{T-t}) e^{\int_{0}^{T-t} c_{t+r}(Y_r) dr + \int_{0}^{T-t} (\gamma_{t+r},\gamma'_{t+r})(Y_r) d\mathbf{W}_r^s} \mid \mathcal{F}^0_{t-s} \right) \right),            
        \end{aligned}
    \end{equation*}
    where $Y_\tau := X^{s,x}_{(t-s)+\tau}$. 
    By combining the two results of \cref{prop:flowpropertyfroRSDEs}, we observe that \begin{equation} \label{eq:definitionofY_flowproperty}
        Y \in RSDE(t,X^{s,x}_{t-s},B^\star,\{\mathcal{F}^\star_\tau\}_{\tau \in [0,T-(t-s)]}),
    \end{equation} 
    where $B^\star_\tau:= B_{(t-s)+\tau} - B_{(t-s)}$ and $\mathcal{F}^\star_\tau := \mathcal{F}_{(t-s) + \tau}$, for any $\tau \in [0,T-(t-s)]$.
    By applying \cref{lemma:measurabilityofRSDEsintheinitialcondition}, it is possible to deduce that there exists a $(\mathcal{F}^{t-s}_{T-s} \otimes \mathcal{B}(\R^d),\mathcal{B}(\R^d))$-measurable map $\varphi: \Omega \times \R^d \to \R$ such that \begin{equation*}
        g(Y_{T-t}) e^{\int_{0}^{T-t} c_{t+r}(Y_r) dr + \int_{0}^{T-t} (\gamma_{t+r},\gamma'_{t+r})(Y_r) d\mathbf{W}_r^s} = \varphi(\omega,X^{s,x}_{t-s}(\omega)).
    \end{equation*}
    Such a map is precisely given by \begin{equation*}
        \varphi(\omega,y) := g(Z^y_{T-t}(\omega)) e^{\int_0^{T-t} c_{t+r}(Z_r^y(\omega)) dr + (\int_0^{T-t} (\gamma_{t+r},\gamma'_{t-r})(Z^y_r) d\mathbf{W}^t_r)(\omega)},
    \end{equation*}
    where $Z^y \in RSDE(t,y,{B}^\star,\{\mathcal{F}^\star_\tau\}_{\tau \in [0,T-(t-s)]})$.
    Recalling that $\mathcal{F}^{t-s}_{T-s}$ is independent from $\mathcal{F}^0_{t-s}$, by the properties of conditional expectation we conclude that \begin{equation}
        \E \left( g(Y_{T-t}) e^{\int_{0}^{T-t} c_{t+r}(Y_r) dr + \int_{0}^{T-t} (\gamma_{t+r},\gamma'_{t+r})(Y_r) d\mathbf{W}_r^s} \mid \mathcal{F}^0_{t-s} \right) = \E(\varphi(\cdot,y))|_{y = X_{t-s}^{s,x}},
    \end{equation}
    and the quantity on the right-hand side is $u(t,X_{t-s}^{s,x})$ by definition. 
\end{proof}

\printbibliography

\end{document}